\documentclass[11pt]{articlefederico2017}
\usepackage{graphicx}
\usepackage{amsfonts }
\usepackage{amsmath}
\usepackage{fullpage}
\usepackage{amssymb}
\usepackage{amsthm}
\usepackage{tikz}
\usepackage{lipsum}
\usepackage{mathrsfs}
\usepackage{hyperref}
\usepackage{float}
\usepackage{multicol,cleveref}
\usepackage{enumerate}
\usepackage{blkarray}
\usepackage{mathtools}
\usepackage{mathrsfs}
\usepackage{mathbbol}

\usepackage{xcolor}

 \definecolor{darkslategray}{rgb}{0.18, 0.31, 0.31}

\usepackage{etex}
\usepackage{xypic}
\usepackage{tikz}
\usepackage{tikz-cd}
\usetikzlibrary{matrix,arrows,decorations.pathmorphing}
\usepackage{MnSymbol}
\usepackage{latexsym}
\usepackage{pgf}
\usepackage{todonotes}
\usepackage{bbm}
\usepackage{pictex}
\usepackage{amsmath,amscd}
\usepackage{afterpage}
\usepackage{youngtab}
\usepackage{psfrag}
\usepackage[boxsize=.25 em]{ytableau}

\usepackage{verbatim}

\usepackage{graphics}

\newtheorem{thm}{Theorem}[section]
\newtheorem{theorem}[thm]{Theorem}

\newtheorem{corollary}[thm]{Corollary}

\newtheorem{lemma}[thm]{Lemma}

\newtheorem{prop}[thm]{Proposition}
\newtheorem{proposition}[thm]{Proposition}
\theoremstyle{definition}

\newtheorem{example}[thm]{Example}
\newtheorem{definition}[thm]{Definition}
\newtheorem{remark}[thm]{Remark}  

\newtheorem{maintheorem}{Theorem}

\numberwithin{equation}{section}

\newcommand*\one{\mathbbm{1}} 

\newcommand{\abs}[1]{\left\vert#1\right\vert}

\newcommand{\R}{\mathbb R}
\newcommand{\F}{\mathbb F}

\newcommand{\CSM}{\operatorname{CSM}}

\newcommand{\Q}{\mathbb{Q}}
\newcommand{\Z}{\mathbb{Z}}

\newcommand{\val}{\mathbb I}
\newcommand{\Ldim}{\underline{\dim}\,}

\newcommand{\N}{\mathbb N}

\newcommand{\mbf}{\mathbf}

\newcommand{\tb}[1]{\textbf{#1}}

\newcommand{\sgn}{\operatorname{sgn}}

\renewcommand{\b}{\operatorname{B}}

\newcommand{\espan}{\operatorname{span}}
\newcommand{\bg}{\operatorname{BG}}
\newcommand{\cbg}{\operatorname{cBG}}
\newcommand{\al}{\operatorname{A}}

\newcommand{\im}{\operatorname{Im}}
\renewcommand{\ker}{\operatorname{Ker}}
\newcommand{\relint}{\operatorname{relint}}

\newcommand{\Poin}{\operatorname{Poin}}

\begin{document}
\title{Valuations and the  Hopf Monoid of Generalized Permutahedra}
\author{\textsf{Federico Ardila\footnote{\noindent \textsf{San Francisco State University, Universidad de Los Andes; federico@sfsu.edu. Partially supported by NSF grant DMS1855610 and Simons Fellowship 613384.}}}\and \textsf{Mario Sanchez\footnote{\noindent \textsf{University of California, Berkeley; mario\_sanchez@berkeley.edu.  Partially supported by NSF Graduate Research Fellowship DGE 1752814.}}}}
\maketitle

\makeatletter
\providecommand\@dotsep{5}
\makeatother
\relax

\begin{abstract}
The goal of this paper is to show that valuation theory and Hopf theory are compatible on the class of generalized permutahedra.
We prove that the Hopf structure $\mbf{GP}^+$ on these polyhedra descends, modulo the inclusion-exclusion relations, to an indicator Hopf monoid $\mathbb{I}(\mbf{GP}^+)$ of generalized permutahedra that is isomorphic to the Hopf monoid of weighted ordered set partitions. This quotient Hopf monoid $\mathbb{I}(\mbf{GP}^+)$ is cofree. It is the terminal object in the category of Hopf monoids with polynomial characters; this partially explains the ubiquity of generalized permutahedra in the theory of Hopf monoids.

This Hopf theoretic framework offers a simple, unified explanation for many new and old valuations on generalized permutahedra and their subfamilies. 
Examples include, 
for matroids: the Chern-Schwartz-MacPherson cycles, Eur's volume polynomial, the Kazhdan-Lusztig polynomial, the motivic zeta function, and the Derksen-Fink invariant; for posets: the order polynomial, Poincar\'e polynomial, and poset Tutte polynomial; for generalized permutahedra: the universal Tutte character and the corresponding class in the Chow ring of the permutahedral variety.
We obtain several algebraic and combinatorial corollaries; for example: the existence of the valuative character group of $\mbf{GP}^+$, and the indecomposability of a nestohedron into smaller nestohedra.

\end{abstract}


\addtocontents{toc}{\protect\setcounter{tocdepth}{3}}
\tableofcontents

\section{Introduction}

\subsection{Algebraic and polyhedral structures in combinatorics} \label{sec:intro1}

Joni and Rota \cite{joni82:_coalg}, Schmitt \cite{schmitt93:_hopf}, and others showed that many families of combinatorial objects have natural operations of ``merging" and ``breaking" that give the family a Hopf algebraic structure. 
Edmonds \cite{edmonds70}, Lov\'asz \cite{lovasz09}, Postnikov \cite{postnikov09},
Stanley \cite{stanley1986two}, and others showed that many families of combinatorial objects can be modeled geometrically as polyhedra -- often part of the family of generalized permutahedra.
These algebraic and geometric structures reflect and shed light on the underlying combinatorial structure of the families under study.

Aguiar and Ardila \cite{AA17} unified these algebraic and polytopal points of view, showing that the family of generalized permutahedra has the structure of a Hopf monoid -- a refinement of Hopf algebras that is more convenient for combinatorial settings. They also showed that this is the largest family of polytopes that supports such a structure. This Hopf monoid $\mbf{GP}^+$ (or certain quotients of it) contains the Hopf monoids of graphs, (pre)posets, matroids, paths, hypergraphs, simplicial complexes, and building sets, among others. This framework allowed them to unify and prove numerous known and new results. The most relevant ones to this project are the following:

$\bullet$ By developing the character theory of Hopf monoids, Aguiar and Ardila showed that important polynomial and quasisymmetric invariants of combinatorial objects come from a simple character of $\mbf{GP}^+$. These include the chromatic polynomial of a graph, the order polynomial of a poset, and the Billera-Jia-Reiner polynomial of a matroid. The celebrated reciprocity theorems for these polynomials are  instances of the same Hopf-theoretic reciprocity phenomenon for $\mbf{GP}^+$.

$\bullet$ They gave the optimal formula for the antipode of the Hopf monoid $\mbf{GP}^+$. This gave, for the first time, cancellation-free formulas for the antipodes of graphs (also found by Humpert-Martin \cite{humpert12}), matroids, and posets, among others.

This work raises the following question.

\medskip

\noindent 
\textbf{Question.} Why are many important Hopf monoids related to generalized permutahedra?

\medskip

\noindent
This paper offers one possible answer to this question, in the Universality Theorem \ref{mainthm:universal}.

\subsection{Polyhedral valuations in combinatorics and geometry}\label{sec:intro2}

Valuations are ways of measuring polytopes that behave well under subdivision. More concretely, let $\mathcal{P}$ be a family of polytopes and $A$ be an abelian group. A function $f:\mathcal{P} \rightarrow A$ is a \textbf{weak valuation} if for any subdivision of a polytope $P \in \mathcal{P}$ into polyhedra $P_1, \ldots, P_k \in \mathcal{P}$ (where for any $a$ and $b$,  $P_a \cap P_b$ is either empty or one of the $P_c$s), we have the inclusion-exclusion relation
\begin{equation}\label{eq:weakval}
f(P) = \sum_{i=1}^k (-1)^{\dim P  - \dim P_i} f(P_i). 
\end{equation}
It is a \textbf{strong valuation} if there exists a linear function $\hat{f}$ such that $f(P) = \hat{f}(\one_P)$, where $\one_P$ is the indicator function of $P$, given by $\one_P(x) = 1$ for $x \in P$ and $\one_P(x) = 0$ for $x \notin P$. Any strong valuation is also a weak valuation. The converse is also true for the class $\mathcal{P}$ of generalized permutahedra, but not necessarily for its subclasses; see Section \ref{subsec:valuations}.

The volume, the number of lattice points, and the Ehrhart polynomial (given by $\mathrm{Ehr}_P(t) = |tP \cap \Z^d|$ for $t \in \N$) are natural ways of measuring a polytope, and they are strong valuations.
However, certain families $\mathcal{P}$ of polyhedra can also be measured using intriguing combinatorial and algebro-geometric valuations that, unexpectedly, also satisfy \eqref{eq:weakval}. These valuations include:

$\bullet$ the Tutte polynomial of a matroid \cite{Speyer},

$\bullet$ the Chern-Schwartz-MacPherson cycles of a matroid \cite{LRS},

$\bullet$ the Kazhdan-Lusztig polynomial of a matroid \cite{EPW},

$\bullet$ the motivic zeta function of a matroid \cite{JKU19},

$\bullet$ the Derksen-Fink invariant of a matroid \cite{DF10},

$\bullet$ the order polynomial of a poset \cite{stanley1986two},

$\bullet$ the Poincar\'e polynomial of a poset cone \cite{DBKR19},


\noindent 
For other examples, see Table \ref{table:valuations}. 
For concrete examples for matroids and posets, see Examples \ref{ex:u36 subdivision} and \ref{ex:posetvaluation}, respectively.
This raises the following question.

\medskip

\noindent 
\textbf{Question.} Why are many important invariants of  matroids and posets also polyhedral valuations?

\medskip

\noindent This paper offers one possible answer to this question within the framework of Hopf monoids in Theorems \ref{mainthm1} and \ref{mainthm:vals}.

Valuations of matroids are especially important because they offer ways of analyzing matroid subdivisions: these are the subdivisions of a matroid polytope into smaller matroid polytopes. Such subdivisions arise naturally in various algebro-geometric contexts, such as the compactification of the moduli space of hyperplane arrangements of Hacking, Keel, and Tevelev \cite{HKT} and Kapranov \cite{Kap}, the compactification of fine Schubert cells in the Grassmannian of Lafforgue \cite{Lafforgue1,Lafforgue2}, the K-theory of the Grassmannian \cite{SpeyerKtheory}, the stratification of the tropical Grassmannian \cite{SS04} and other tropical homogeneous spaces \cite{Rincon}, and the study of tropical linear spaces by Ardila and Klivans \cite{ArdilaKlivans} and Speyer  \cite{Speyer}. 

A foundational result by Derksen and Fink \cite{DF10} gave the universal valuation for matroids and for generalized permutahedra. Their result was extended by Eur, Sanchez, and Supina \cite{ESS} who gave the universal valuation for Coxeter matroids and for Coxeter generalized permutahedra.

\subsection{Hopf algebraic structures on generalized permutahedra and valuations}

The goal of this paper is to explain the intimate relationship between the Hopf algebraic structures of Section \ref{sec:intro1} and the valuations of Section \ref{sec:intro2}.
Let $\F$ be a field of characteristic $0$ and let $\mbf{GP}^+$ be the ($\F$-linear) Hopf monoid of extended generalized permutahedra, whose components are the vector spaces
\[
\mbf{GP}^+[I]  = \espan\{P \, | \, P \textrm{ is an extended generalized permutahedron in } \R^I\} 
\]
for all finite sets $I$. Consider the subdivisions of polyhedra in this family into polyhedra in this family; they give the \textbf{inclusion-exclusion subspecies} $\mbf{ie} \subset \mbf{GP}^+$ consisting of the vector spaces
\[
\mbf{ie}[I] := 
\espan \left\{P - \sum_i (-1)^{\dim P - \dim P_i} P_i \, | \,
\{P_i\} \textrm{ is a polyhedral subdivision of $P$} \right\} 
\subset \mbf{GP}^+[I].  
\]
Consider the indicator vector spaces of generalized permutahedra:
\begin{eqnarray*}
\mathbb{I}(\mbf{GP}^+)[I] & := & \espan\{\one_P \, | \, P \textrm{ is an extended generalized permutahedron in } \R^I\} \\
& \cong & \mbf{GP}^+[I] / \mbf{ie}[I],
\end{eqnarray*}
where $\one_P : \R^I \rightarrow \F$ is the indicator function of $P$, which equals $1$ in $P$ and $0$ outside of $P$.

The following are our main results.

\bigskip
\noindent \textsf{\textbf{The indicator Hopf monoid.}} 
The Hopf monoid $\mbf{GP}^+$ descends to the quotient $\mathbb{I}(\mbf{GP}^+)$:

\begin{maintheorem}\label{mainthm1} 
Let $\mbf{GP}^+$ be the Hopf monoid of extended generalized permutahedra.
\begin{enumerate}
\item
The inclusion-exclusion species $\mbf{ie}$ is a Hopf ideal of $\mbf{GP}^+$.
\item 
The quotient $\val(\mbf{GP}^+) =  \mbf{GP}^+/\mbf{ie}$ is a Hopf monoid. 
\item
The resulting \textbf{indicator Hopf monoid of extended generalized permutahedra} is isomorphic to the \textbf{Hopf monoid of weighted ordered set partitions}: 
\[
\val(\mbf{GP}^+) \cong \mbf{w\Sigma^*}.
\]
\item
For any Hopf submonoid $\mbf{H} \subseteq \mbf{GP}^+$, the subspecies $\val(\mbf{H}) \subseteq \val(\mbf{GP}^+)$ is a Hopf quotient of $\mbf{H}$; namely, $\val(\mbf{H}) \cong \mbf{H}/(\mbf{ie} \cap \mbf{H})$.
\end{enumerate}
\end{maintheorem}

\noindent It is also interesting to quotient $\val(\mbf{GP}^+)$ further by identifying $P$ with its translates $P+v$, as is done in the McMullen polytope algebra. We define the \textbf{extended McMullen subspecies}\footnote{This behaves very differently from the \textbf{McMullen subspecies}  $\mbf{Mc}$ of the Hopf monoid of \textbf{bounded} generalized permutahedra $\mbf{GP}$ and the quotient $\mbf{GP}/\mbf{Mc}$; see Sections \ref{sec:relatedwork} and \ref{sec:indicator2}.} 
\[
\mbf{Mc}^+[I] := \mbf{ie}[I] + \espan \left\{P - (P+v) \, | \,  P \in GP^+[I], v \in \R^I\right\}  \subset \mbf{GP}^+[I].  
\]
We prove that $\mbf{Mc}^+$ is also a Hopf ideal of $\mbf{GP}^+$, and the resulting quotient is isomorphic to the \textbf{indicator Hopf monoid of preposet cones}, the \textbf{indicator Hopf monoid of poset cones}, and the 
\textbf{Hopf monoid of ordered set partitions}; see Theorems \ref{thm:McMullen} and Proposition \ref{prop:P=PP}:
\[
\mbf{GP}^+/\mbf{Mc}^+ \cong \mathbb{I}(\mbf{P}) \cong \mathbb{I}(\mbf{PP}) \cong \mbf{\Sigma^*}.
\]


Building on Aguiar and Ardila's formula for the antipode of $\mbf{GP}^+$, Theorem \ref{mainthm1} gives the following elegant formula for the antipode of $\val(\mbf{GP}^+)$.

\begin{corollary}\label{cor:main}
The antipode of the indicator Hopf monoid of generalized permutahedra $\val(\mbf{GP}^+)$ is given by 
\[
s_I(P) = (-1)^{|I|-\dim P} P^\circ \qquad \text{ for } P \in \mbf{GP}^+[I], 
\]
where $P^\circ$ is the relative interior of $P$.
\end{corollary}

\bigskip
\noindent \textsf{\textbf{Cofreeness and universality.}} A priori, it seems very surprising that so many Hopf monoids of interest are closely related to the Hopf monoid of generalized permutahedra $\mbf{GP}^+$, as shown in \cite{AA17}. We give a possible explanation of this phenomenon, by showing that the indicator Hopf monoid of generalized permutahedra $\mathbb{I}(\mbf{GP}^+) = \mbf{GP}^+/\mbf{ie}$ and the further quotient by the extended McMullen subspecies $\mbf{GP}^+/\mbf{Mc}^+$ satisfy very natural universality properties among \textbf{all} Hopf monoids.

This is most elegantly stated for $\mbf{GP}_{\mathbb{N}}^+$, which consists of the extended generalized permutahedra whose supporting hyperplanes have non-negative integral coefficients.
Define a \textbf{polynomial character} on a Hopf monoid $\mbf{H}$ to be a multiplicative function from $\mbf{H}$ to the polynomial ring $\mathbb{F}[t]$. Define the \textbf{canonical character} $\beta: \mathbb{I}(\mbf{GP}_{\mathbb{N}}^+) \to \mathbb{F}[t]$ by
 \[
 \beta(\one_P) = \begin{cases}
(-1)^{|I|-\dim \textrm{Lin}(P)} t^{p} & \text{if $P$ is relatively bounded and lies on hyperplane $\displaystyle \sum_{i \in I} x_i = p$ in $\R^I$,} \\ 
 0 & \text{if $P$ is relatively unbounded},
 \end{cases} 
 \]
 for the indicator function of a polyhedron $P$, where Lin$(P)$ is the lineality space of $P$ and where a face $F$ of $P$ is \textbf{relatively bounded} if it is non-empty and $F/\text{Lin}(P)$ is a bounded face of $P/\text{Lin}(P)$ and relatively unbounded otherwise.

\begin{maintheorem}\label{mainthm:universal}
The terminal Hopf monoid with a polynomial character is $(\mathbb{I}(\mbf{GP}_{\mathbb{N}}^+), \beta)$.

\noindent Explicitly:
For any connected Hopf monoid $\mbf{H}$ and any polynomial character $\zeta$,  there exists a unique Hopf morphism $\hat{\zeta}: \mbf{H} \to \mathbb{I}(\mbf{GP}_{\mathbb{N}}^+)$ such that $\beta \circ \hat{\zeta} = \zeta$.
\end{maintheorem}

It also follows from these general results that the indicator Hopf monoid is cofree. This is shown in Theorem \ref{thm:cofree}.

Similarly, the terminal Hopf monoid with a character is the quotient $\mbf{GP}^+/\mbf{Mc}^+$ with the canonical character; 
see Theorem \ref{thm: universality of Sigma}.
Aguiar and Mahajan \cite{AM10} had proved this property for $\mbf{\Sigma^*}$, which we show is isomorphic to $\mbf{GP}^+/\mbf{Mc}^+$ in Theorem \ref{thm:McMullen}.

\bigskip
\noindent \textsf{\textbf{Hopf algebraic valuations on polytopes.}}
Theorem \ref{mainthm1} shows the compatibility between the Hopf structure on generalized permutahedra and the valuative functions on these polytopes. Many functions on generalized permutahedra can be seen as functions on the Hopf monoid $\mbf{GP}^+$, which descend to functions on the quotient Hopf monoid $\mathbb{I}(\mbf{GP}^+)$. Those functions must then be valuations. The same is true for submonoids of $\mbf{GP}^+$. 
The following is one concrete manifestation of this general principle:

\begin{maintheorem}\label{mainthm:vals} Let $\tb{H}$ be a Hopf submonoid of $\tb{GP}^+$. Let $S_1 \sqcup \cdots \sqcup S_k = I$ be a set decomposition and consider functions $f_i: \tb{H}[S_i] \to R$ for $1 \leq i \leq k$, where $R$ is a ring with multiplication $m$. 
Define the function $f_1 \star \cdots \star f_k : \tb{H}[I] \to R$ by $f_1\star \cdots \star  f_k := m \circ f_1 \otimes f_2 \otimes \cdots \otimes f_k \circ \Delta_{S_1,\ldots,S_k}$.
\[
\text{If } f_1, \ldots, f_k \text{ are strong valuations, then } f_1 \star \cdots \star f_k \text{ is a strong valuation.}
\]
\end{maintheorem}

Many new and known valuations on subfamilies of generalized permutahedra arise from applying Theorem \ref{mainthm:vals} to much simpler valuations $f_i$. The earlier proofs of those results, often quite subtle, are thus replaced by a uniform, straightforward computation. This applies to the following valuations.

\renewcommand{\arraystretch}{1.5}
\begin{table}[H]
    \centering
    \begin{tabular}{|p{0.28\linewidth}|p{0.26\linewidth}|p{0.41\linewidth}|}
        \hline
        Submonoid & Valuations $f_i$ & Valuation from Theorem B \\
        \hline
        Generalized permutahedra (Section \ref{sec:valGP})  & normalized volume & Exponential of the class in the Chow ring of the permutahedral variety \cite{FS94} \\
         & universal norm & Dupont, Fink, and Moci's universal Tutte character \cite{DFM17} \\
         Matroid morphisms & universal norm & Las Vergnas's Tutte polynomial \cite{Ver80}\\
        \hline
        Matroids (Section \ref{sec:valM})& beta invariant& Chern-Schwartz-MacPherson cycles \cite{LRS}\\
         &  characteristic polynomial & Eur's volume polynomial \cite{Eur20} \\
         &  characteristic polynomial & Kazhdan--Lusztig polynomial \cite{EPW} \\
         &  characteristic polynomial & motivic zeta function \cite{JKU19} \\   
         & having a unique basis & Billera, Jia, and Reiner's quasisymmetric function \cite{BJR09}\\
         & having only one element & Derksen-Fink invariant \cite{DF10}\\
         & universal norm & Tutte and characteristic polynomial \cite{Tutte67} \\
\hline
        Posets (Section \ref{sec:valP}) & being an antichain & Stanley's order polynomial \cite{Stanley70}
         \\
         & being an antichain, 1 & Gordon's Tutte polynomial  \cite{Gor93} \\
         & being an antichain & Dorpalen-Barry, Kim, and Reiner's Poincar\'e polynomial  \cite{DBKR19} \\
\hline
        Nestohedra (Section \ref{sec:valBS}) & constant function & $f$-polynomial \cite{PRW06} \\
        \hline
    \end{tabular}
    \caption{Examples of valuations from Theorem  \ref{mainthm:vals}}
    \label{table:valuations}
\end{table}

The character theory of Hopf monoids provides an especially useful corollary to Theorem \ref{mainthm:vals}.
It is explained in \cite{AA17, ABS06} that a multiplicative function from a Hopf monoid $\mbf{H}$ to a fixed field, known as a \textbf{character}, gives rise to a family of polynomials $f_\zeta(h)$, quasisymmetric functions $\Phi_\zeta(h)$ and linear combinations of ordered set partitions $O_\zeta(h)$ associated to each object $h$ of the Hopf monoid $\mbf{H}$. Examples include the order polynomial of a poset, the chromatic polynomial of a graph, and the Billera-Jia-Reiner quasisymmetric function of a matroid. When the character is also a valuation, we can say more.


\begin{corollary}\label{cor: character theory valuation} Let $\mbf{H}$ be a submonoid of $\mbf{GP}$. Let $\zeta$ be a character of $\mbf{H}$ such that $\zeta[I]$ is a strong valuation. Then the three maps 
\[
    h \mapsto f_{\zeta}(h)(t) \qquad
    h \mapsto \Phi_{\zeta}(h) \qquad 
        h \mapsto O_{\zeta}(h) \qquad  \text{ for } h \in \mbf{H}[I]
\]
are all strong valuations.
\end{corollary}

The multiplicative functions from $\mbf{H}$ to a fixed field, known as the characters of $\mbf{H}$, form a group $\mathbb{X}(\mbf{H})$ under convolution. The inverse of a valuative character is given by precomposing it with the antipode \cite{AA17, ABS06}. The above results give an interesting structural consequence, shown in Proposition \ref{prop:valchar}: the characters of $\mbf{H}$ that are valuative form a subgroup $\mathbb{X}(\mbf{H})^{val} \subseteq \mathbb{X}(\mbf{H})$ of the character group.

\subsection{Related work} \label{sec:relatedwork}

A Hopf algebra analog of Part 2 of Theorem \ref{mainthm1} was proved by Derksen and Fink \cite{DF10}. By working in the context of Hopf monoids and taking a more geometric approach, we are able to obtain several new consequences, including numerous results in \cite{AA17}, the simple formula for the antipode in Corollary \ref{cor:main}, and many new examples of valuations.

Results analogous to Parts 1 and 2 of Theorem \ref{mainthm1} and Corollary \ref{cor:main} were also obtained independently and simultaneously by Bastidas \cite{Bastidas20}, for the quotient $\mbf{GP}/\mbf{Mc}$ of the Hopf monoid of \textbf{bounded} generalized permutahedra by its McMullen subspecies.
The quotients $\mbf{GP}/\mbf{Mc}$
and $\mbf{GP}^+/\mbf{Mc}^+$ are very different from each other; in fact, all the bounded polytopes on a fixed ground set are identified in the quotient $\mbf{GP}^+/\mbf{Mc}^+$; see Proposition \ref{prop:psi}.
By including unbounded polyhedra, we obtain a structure that is more favorable for our purposes: the resulting indicator Hopf monoid 
$\mathbb{I}(\mbf{GP}^+) =  \mbf{GP}^+/\mbf{ie}$ 
is isomorphic to the Hopf monoid on weighted set partitions, is cofree, and is the terminal object in the category of Hopf monoids with a(n extended) polynomial character. 
Its quotient $\mbf{GP}^+/\mbf{Mc}^+$ is isomorphic to the Hopf monoid on  set partitions, is cofree, and is the terminal object in the category of Hopf monoids with a character.

\subsection{Outline}

In Section \ref{sec:background}, we introduce the relevant background for generalized permutahedra and Hopf monoids. We give many examples of Hopf monoids and of combinatorial objects that can be associated to generalized permutahedra. In Section \ref{sec:morphisms}, we construct the Brianchon-Gram Hopf morphism for polytopes and the aligning morphism for posets, which play an important role in our work. One is related to the Brianchon-Gram formula and the other describes cones in terms of ordered set partitions. In Section \ref{sec:indicator}, we prove Theorem \ref{mainthm1} on the existence of the indicator Hopf monoid $\mathbb{I}(\mathbf{GP}^+)$ and its quotient $\mathbf{GP}^+/\mbf{Mc}^+$. In Section \ref{sec:universal} we prove that these Hopf monoids are cofree and they are the terminal Hopf monoids with a (generalized polynomial) character, Theorem \ref{mainthm:universal}. 
In Section \ref{sec:valuations}, we prove Theorem \ref{mainthm:vals} and we use it to show that various Hopf monoidal constructions give rise to valuations on polytopes.

The remaining sections focus on some known and many new examples, as summarized in Table \ref{table:valuations}. Sections \ref{sec:valGP}, \ref{sec:valM}, \ref{sec:valP}, \ref{sec:valBS} focus on valuations on generalized permutahedra, matroids, posets, and building sets, respectively.  In Section \ref{sec:valBS}, we use this 
 to show that there are no nestohedral subdivisions.
We close with Appendix \ref{sec:appendix} where we summarize the main facts we need about Hopf monoids and prove the First Isomorphism Theorem for them.



\section{Background}\label{sec:background}

\subsection{Generalized permutahedra}

For a set $I$ of size $n$, the standard \textbf{permutahedron} $\pi_I$ is the convex hull of the $n!$ bijective functions $\pi: I \rightarrow [n]$.  We are interested in the deformations of the permutahedron, which are defined as follows. 
A \textbf{generalized permutahedron} is a polytope in $\R^I$ that satisfies the following equivalent conditions:

$\bullet$ Its edges are parallel to vectors in the root system $A_I=\{e_i - e_j \, | \, i, j \in I\}$, where $\{e_i \, : \, i \in I\}$ are the standard basis vectors.

$\bullet$ Its normal fan is a coarsening of the \textbf{braid arrangement} $\Sigma_I$ which is the hyperplane arrangement in $\R^I$ given by the hyperplanes $H_{i,j} = \{x \in \R^I \; \lvert \; x_i = x_j\}$ for $i, j \in I$.

$\bullet$ It is obtained from the standard permutahedron $\pi_I$ by moving the facets while preserving their directions, without letting a facet cross a vertex.

$\bullet$ It is given by the inequality description
\[
P = \{x \in \R^I \mid  \sum_{i\in I}x_i=z(I) \, \text{ and } \sum_{i\in A}x_i\leq z(A) \text{ for all }A\subseteq I\}
\]
for a function $z:2^I \rightarrow \R$ that is \textbf{submodular}; that is, it satisfies $z(A) + z(B) \geq z(A\cup B)+z(A\cap B)$ for all $A, B \subseteq I$.

\begin{figure}[!h]
\centering
\begin{center}
\includegraphics[scale=.4]{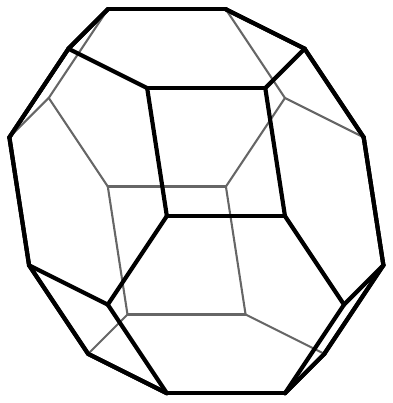} \qquad
\includegraphics[scale=.4]{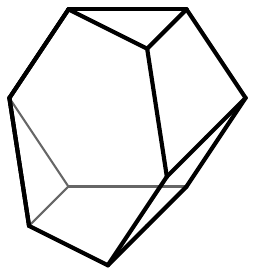}  \quad
\includegraphics[scale=.4]{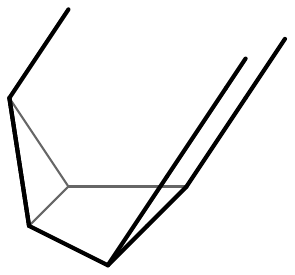} \quad
\includegraphics[scale=.4]{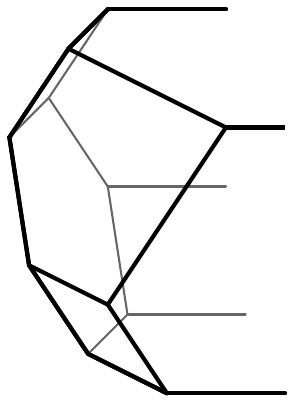} \quad
\includegraphics[scale=.4]{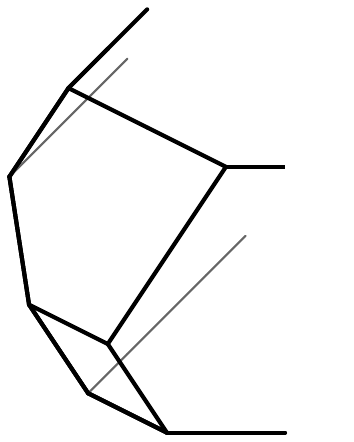}
\end{center}
\caption{The standard $3$-permutahedron and four other extended generalized permutahedra. \label{f:genperm}}
\end{figure}

\smallskip

More generally, an \textbf{extended generalized permutahedron} is a possibly unbounded polyhedron in $\R^I$ that satisfies the following equivalent conditions:

$\bullet$ Its faces lie on translates of subspaces spanned by roots in $A_I$.

$\bullet$ Its normal fan is a coarsening of a subfan of the braid fan.

$\bullet$ It is obtained from the permutahedron by moving the facets while preserving their directions, without letting a facet cross a vertex, possibly sending some facets to infinity.

$\bullet$ It is given by the inequality description
\[
P = \{x\in\R^I \mid  \sum_{i\in I}x_i=z(I) \, \text{ and } \sum_{i\in A}x_i\leq z(A) \text{ for all }A\subseteq I\}
\]
for a function $z:2^I \rightarrow \R \cup \{\infty\}$ that is submodular; that is, it satisfies $z(A) + z(B) \geq z(A\cup B)+z(A\cap B)$ for all $A, B \subseteq I$ such that $z(A)$ and $z(B)$ are finite.

\medskip

This ubiquitous family of polytopes was first studied in optimization under the name of polymatroids \cite{edmonds70, Fujishige}. Its combinatorial structure was studied in \cite{postnikov09} and its algebraic structure was studied in \cite{AA17}. Generalized permutahedra arise naturally in optimization (where they parameterize problems where the greedy algorithm successfully finds a solution \cite{Schrijver}), in algebraic geometry (where they are in correspondence with the numerically effective divisors of the permutahedral toric variety $X_{\Sigma_I}$ \cite{toricvarietiesbook}), and in algebra (where they describe the irreducible representations of the Lie algebra ${sl}_n$ \cite{FultonHarris}, and  they are the largest family of polytopes that carries the structure of a Hopf monoid \cite{AA17}.)

Generalized permutahedra are also of great importance in combinatorics, because they provide geometric models of many important combinatorial families: graphs, matroids, posets, preposets, ordered set partitions, hypergraphs, simplicial complexes, and building sets, among others. Furthermore, Aguiar and Ardila showed that the well-studied Hopf structures on these and other combinatorial families can all be unified within the framework of the Hopf monoid $\mbf{GP}^+$.

%

\bigskip
\noindent \textbf{\textsf{Matroids.}}
A \textbf{matroid $M$ on ground set $I$} is a nonempty collection $\mathcal{B}(M)$ of subsets of $I$ called \textbf{bases} that satisfy the \emph{basis exchange axiom}: if $A$ and $B$ are bases and $a \in A - B$, there exists $b \in B-A$ such that $(A - \{a\}) \cup \{b\}$ is a basis.

An \textbf{independent set} is a subset of a basis. 
The \textbf{rank function} of a matroid $M$ is the function $r:2^I \rightarrow \N$ given by
\[
r(J) = \max_{B \in \mathcal{B}(M)} |J \cap B| \qquad \textrm{ for } \emptyset \subseteq J \subseteq I.
\]
This is the size of any maximal subset of $J$ that is independent.

The \textbf{matroid polytope} $P(M)$ of a matroid $M$ on $I$ is given by
\[
 P(M) = \textrm{conv} \, \{e_{b_1} + \cdots + e_{b_r}  \mid  \{b_1, \ldots, b_r\} \textrm{ is a basis of } M\} \subseteq\R^I.
 \]
It is a generalized permutahedron \cite{ArdilaBenedettiDoker} whose vertices correspond to the bases and whose edges correspond to the elementary basis exchanges between them.

\bigskip
\noindent \textbf{\textsf{Posets and preposets.}} A \textbf{poset} or \textbf{partially ordered set} $p$ on a finite set $I$ is a relation $p \subseteq I \times I $, denoted $\leq$ or $(I, \leq)$, which is reflexive ($x \leq x$ for all $x \in I$), antisymmetric ($x \leq y$ and $y \leq x$ imply $x = y$ for all $x,y \in I$), and transitive ($x \leq y$ and $y \leq z$ imply $x \leq z$  for all $x,y,z\in I$).

More generally, a \textbf{preposet} on a finite set $I$ is a relation $q \subseteq I \times I $, denoted $\leq$ or $(I, \leq)$, that is reflexive and transitive, but not necessarily antisymmetric. We write $x<y$ if $x \leq y$ and $x \ngeq y$.

A preposet $q$ gives rise to an equivalence relation given by $x \sim y$ if $x \leq y$ and $y \leq x$, and to a poset $p = q/\sim$ on the equivalence classes of $\sim$ where $[x] \leq_p [y]$ if and only if $x \leq_q y$. Since we can recover the preposet $q$ from the poset $p$, we will identify the preposet $q=(I, \leq_q)$ with the poset $p=(I/\sim, \leq_p)$. The \textbf{size} of $q$ is the number of equivalence classes, $|q| := |I/\sim|$.

A \textbf{weighted preposet} $(w,q)$ consists of a preposet $q$ and a function $w:q/\sim \,  \rightarrow \R$; that is, a choice of a real weight $w(q_a)$ for each equivalence class $q_a$ of the preposet $q$.

The \textbf{(pre)poset cone} of a (pre)poset $q$ is
\[
\operatorname{cone}(q) = \operatorname{cone} \{e_i - e_j \, : \, i \geq j \textrm{ in } q\}.
\]
This is an extended generalized permutahedron, and its lineality space is $(|I|-k)$-dimensional where $q / \sim \,\, = \{q_1, \ldots, q_k\}$; it is cut out by the $k$ independent equations 
$
\sum_{i \in q_a} x_i = 0  \textrm{ for } 1 \leq a \leq k,
$
one for each equivalence class of $q$. 
\footnote{These cones are related to the preposet-cone dictionary given by \cite{PRW06}. For any (pre)poset $p$ on $I$, let $\sigma_p$ denote the cone
$\sigma_p = \{x \in \R^I \; \lvert \; x_i \geq x_j \text{ for all $i \geq_p j$}\}$. Then, the poset cone $\operatorname{cone}(p)$ is the dual cone to the cone $\sigma_p$.}

The \textbf{translated (pre)poset cone} of a weighted (pre)poset $(w,q)$ is
\begin{equation} \label{eq:P(w,F)}
\operatorname{cone}(w, q) = w^q + \operatorname{cone}(q)
\end{equation}
where $w^q$ is a vector in $\R^I$ such that $\sum_{i \in q_a} x_i = w(q_a)$ for each equivalence class $q_a$ of $q$. Any such vector $w^q$ will produce the same cone $\operatorname{cone}(w,q)$ thanks to the description of the lineality space of $\operatorname{cone}(q)$ given above. 

Translated preposet cones are precisely the cones that are extended generalized permutahedra \cite[Theorem 3.4.9]{AA17}.

\bigskip
\noindent \textbf{\textsf{Weighted ordered set partitions and plates.}} Ordered set partitions are of fundamental importance in the theory of Hopf monoids, and weighted ordered set partitions will play a central role in this project.

\begin{definition}\label{def:setcomp}
An \textbf{ordered set partition} (or \textbf{set composition}) of a finite set $I$ is an ordered sequence $\ell=\ell_1|\cdots|\ell_k$ of nonempty,  pairwise disjoint sets such that $\ell_1 \sqcup \cdots \sqcup \ell_k = I$. 

A \textbf{set decomposition} of $I$ is an ordered sequence $S_1|S_2|\cdots| S_k$ of possibly empty, pairwise disjoint sets such that $S_1 \sqcup \cdots \sqcup S_k = I$.
\end{definition}

The ordered set partitions of $I$ are in bijection with the faces of the braid arrangement in $\R^I$. They are also in bijection with the \textbf{totally ordered preposets}, where every pair of elements is comparable: the ordered set partition $\ell=\ell_1|\cdots|\ell_k$ corresponds to the preposet $q_\ell$ where $i \leq j$ for $i \in \ell_a, j \in \ell_b$ with $a \leq b$. This preposet is equivalent to the linear poset $\ell_1 < \cdots < \ell_k$ on its equivalence classes.

A \textbf{(pre)linear extension} $\ell$ of a (pre)poset $q$ is a totally ordered (pre)poset $\ell$ 
such that $x < y$ in $q$ implies $x < y$ in $\ell$ and $x \leq y$ in $q$ implies $x \leq y$ in $\ell$\footnote{This second condition ensures that an equivalence class in $q$ cannot be split into smaller equivalence classes in a linear extension of $q$.} We also think of $\ell$ as the associated ordered set partition $\ell$.

\begin{definition}\label{def:wsetcomp}
A \textbf{weighted ordered set partition} of $I$ is a pair $(w,\ell)$ consisting of an ordered set partition $\ell=\ell_1|\cdots|\ell_k$ of $I$ and an assignment $w:\{\ell_1, \ldots, \ell_k\} \rightarrow \R$ of a real weight $w(\ell_a)$ for each part $\ell_a$ of $\ell$. We also write $w=(w_1, \ldots, w_k)$.
\end{definition}

The following cones are in bijection with weighted ordered set partitions.

\begin{definition} \label{def:plate}
A \textbf{plate} is a cone of the form
\begin{eqnarray*}
\operatorname{cone}(w, \ell) 
&=& \{ x \in \R^I \, : \, x_{\ell_1} \geq w_1, \,\,
\ldots, \,\, 
x_{\ell_1 \sqcup \cdots \sqcup \ell_{k-1}} \geq w_1+\cdots + w_{k-1}, \,\,
x_{\ell_1 \sqcup \cdots \sqcup \ell_{k}} = w_1+\cdots + w_{k}\} \\
&=& w^\ell + \operatorname{cone} \{e_i - e_j \, : \, i \in \ell_a, j \in \ell_b, a \geq b\}
\end{eqnarray*}
%
for some weighted ordered set partition $(w,\ell)$, where 
$x_\ell := \sum_{l \in \ell} x_l$, and $w^\ell$ is any vector in $\R^I$ such that $\sum_{i \in \ell_a} (w^\ell)_i = w(\ell_a)$ for $1 \leq a \leq k$. 
If $w=0$ then the plate is called \textbf{centered}.
\end{definition}

These cones arise in numerous contexts. In this terminology, plates (also called \textbf{permutahedral plates} or \textbf{tectonic plates}) were introduced by Ocneanu \cite{Ocneanu} and studied by Early \cite{Early}. 
If we regard the (weighted) ordered set partition $\ell$ as a (weighted) preposet $q_\ell$, then the (weighted) plate of $\ell$ coincides with the (weighted) preposet cone of  $q_\ell$.

\subsection{Hopf monoids}

Hopf monoids are counterparts of Hopf algebras that are especially well-suited for combinatorial analysis. 
There are four natural functors from Hopf monoids to Hopf algebras, so everything that we do in this paper can also be done at the level of Hopf algebras. 

Although the formal definition of a Hopf monoid is technical, the intuition is simple.
We begin by giving an informal description of a Hopf monoid. For a precise definition, see the Appendix in Section \ref{sec:appendix}. 
For a combinatorial discussion and ``user's manual", see \cite{AA17}. For a thorough algebraic treatment, see the original monograph \cite{AM10} by Aguiar and Mahajan where these objects are introduced. 
We will give many examples in Section \ref{sec:examples}. 
%

\medskip

A \textbf{Hopf monoid} $\mathbf{H}$ consists of the following data, subject to some suitable axioms:
\begin{enumerate}
\item
A vector space $\mathbf{H}[I]$ for each finite set $I$ and an isomorphism from 
$\mathbf{H}[I]$ to $\mathbf{H}[J]$ for each bijection from $I$ to $J$. \smallskip
\\
(In many examples, a basis for $\mathbf{H}[I]$ is given by the different ``$H$-structures" that can be put on the ``ground set" $I$, and the isomorphisms are given by the natural maps obtained from relabeling the ground set.)
\item Compatible operations:
\begin{itemize}
\item
An associative \textbf{product} $m_{S,T}: \mathbf{H}[S] \otimes \mathbf{H}[T] \rightarrow \mathbf{H}[I]$ for each decomposition $I = S \sqcup T$.  \smallskip \\ 
(In many examples, this is given by a combinatorial rule to merge two $H$-structures on $S$ and $T$ into one $H$-structure on $I$.)
\item
A coassociative \textbf{coproduct} $\Delta_{S,T}: \mathbf{H}[I] \rightarrow \mathbf{H}[S] \otimes\mathbf{H}[T] $ for each decomposition $I = S \sqcup T$. \smallskip
 \\
(In many examples, this is given by a combinatorial rule to break one $H$-structure on $I$ into two $H$-structures on $S$ and $T$.) 
\item
An \textbf{antipode} $\mathrm{s}_I: \mathbf{H}[I] \rightarrow \mathbf{H}[I]$ for each finite set $I$. \smallskip
 \\
(This is given by an alternating sum of combinatorial objects, with many cancellations that are usually highly non-trivial and combinatorially interesting.)
\end{itemize}
\end{enumerate}

\noindent 
A \textbf{Hopf ideal} $\mbf{g} \subset \mbf{H}$ is a Hopf submonoid that satisfies:
 \begin{eqnarray*}
 m_{S,T}(\mbf{H}[S] \otimes \mbf{g}[T] +  \mbf{g}[S] \otimes \mbf{H}[T]) &\subset& \mbf{g}[I], \\
 \Delta_{S, T}(\mbf{g}[I]) &\subset& \tb{H}[S] \otimes \tb{g}[T] + \tb{g}[S] \otimes \tb{H}[T].
 \end{eqnarray*}
If $\mbf{g}$ is a Hopf ideal, then one can define a \textbf{quotient Hopf monoid} in the natural way. 
 
There is a natural notion of morphisms of Hopf monoids. It satisfies Noether's First Isomorphism Theorem in the following formulation.

 \begin{theorem}\label{th:isothm1} Let $\mbf{H}_1$ and $\mbf{H}_2$ be two Hopf monoids and $\alpha: \mbf{H}_1 \to \mbf{H}_2$ be a Hopf morphism. Then, the image of $\alpha$ is a Hopf submonoid of $\mbf{H}_2$, the kernel of $\alpha$ is a Hopf ideal of $\mbf{H}_1$ and we have the isomorphism of Hopf monoids
\[
\mbf{H}_1/\ker(\alpha) \cong \im(\alpha).
\]
\end{theorem}

\begin{proof}
See the Appendix.
\end{proof}

\subsection{Examples of Hopf monoids}\label{sec:examples}

Although it is not clear a priori from the definitions, most of the Hopf monoids that will appear in this paper are closely related to the following Hopf monoid of generalized permutahedra.

\bigskip
\noindent \textbf{\textsf{Generalized permutahedra.}} 
Let $P$ be a(n extended) generalized permutahedron in $\R^I$ and let $I = S \sqcup T$ be a decomposition. Let $e_{S,T}$ denote the linear functional $e_{S,T}(x) = \sum_{i \in S} x_i$. Let $P_{S,T}$ denote the face of $P$ maximized by the linear functional $e_{S,T}$. If this face is nonempty, then there exist two (extended) generalized permutahedra $P|_S \subset \R^S$ and $P/_S \subset R^T$ such that $P_{S,T} = P|_S \times P/_S$.

If $P \subset \R^S$ and $Q \subset \R^T$ are (extended) generalized permutahedra, then $P \times Q$ is a(n extended) generalized permutahedron in $\R^{I}$.

Now let $\mbf{GP}^+$ be the species given by 
\[
\mbf{GP}^+[I] = \F\{ \text{extended generalized permutahedra in } \R^I\}.
\]
A bijection from $I$ to $J$ induces a vector space isomorphism from $\R^I$ to $\R^J$,  which induces an isomorphism from $\mbf{GP}^+[I]$ to $\mbf{GP}^+[J]$; so this is indeed a species. To simplify (and slightly abuse) notation, we will write $P \in \mbf{GP}^+[I]$ whenever $P$ is an extended generalized permutahedron in $\R^I$.

\begin{definition} \cite{AA17}
The \textbf{Hopf monoid of (extended) generalized permutahedra} $\mbf{GP}^+$ is the species 
$\mbf{GP}^+[I] = \F\{\textrm{extended generalized permutahedra in } \R^I\}$ with product
 \[
 m_{S,T}(P,Q) = P \times Q \qquad \textrm{ for } P \in \mbf{GP}^+[S], Q \in \mbf{GP}^+[T]
 \]
and coproduct
 \[
 \Delta_{S,T}(P) = 
 \begin{cases} 
 P|_S \otimes P/_S & \text{ if $P$ is bounded in direction $e_{S,T}$ , or} \\
0 & \text{ otherwise},
\end{cases}
\qquad
\textrm{ for } P \in \mbf{GP}^+[I].
\]
\end{definition}

The Hopf monoid $\mbf{GP}^+$, its Hopf submonoids, and quotient Hopf monoids are our main algebraic objects of study. The following submonoids of $\mbf{GP}^+$ will play a role in what follows; see Theorem \ref{thm:CGP}:

\noindent
$\bullet$ the Hopf monoid $\mbf{GP}^+$ of extended generalized permutahedra,

\noindent
$\bullet$ the Hopf monoid $\mbf{GP}$ of bounded generalized permutahedra,

\noindent
$\bullet$ the Hopf monoid $\mbf{CGP}^+$ of conical generalized permutahedra,  

\noindent
$\bullet$ the Hopf monoid $\mbf{PCGP}^+$ of conical generalized permutahedra that are pointed,

\noindent
$\bullet$ the Hopf monoid $\mbf{CGP}_0^+$ of conical generalized permutahedra where the origin is in the apex,

\noindent
$\bullet$ the Hopf monoid $\mbf{PCGP}_0^+$ of conical generalized permutahedra that are pointed at the origin.

\noindent
Here the apex of a cone is its lineality space.

\bigskip
\noindent \textbf{\textsf{Matroids.}}
Consider a matroid $M$ on $I$ and a decomposition $I=S\sqcup T$. The \textbf{restriction} of $M$ to $S$ is the matroid $M|_S$ on ground set $S$ with 
\[
\mathcal{B}(M|_S) = \{\text{maximal intersections of the form } B \cap S \, | \,  B\in \mathcal{B}(M)\}.
\]
The \textbf{contraction} of $S$ from $M$ is the matroid on ground set $T$ defined as
\[
\mathcal{B}(M/_S) = \{B_T \subseteq T \mid \text{for a basis $B_S$ of $M|_S$ we have $B_S \cup B_T \in \mathcal{B}(M)$}\}.
\]
Let $M_1$ and $M_2$ be matroids on ground set 
$S$ and $T$, respectively, and $I=S\sqcup T$. Their \textbf{direct sum} is the matroid on ground set $I$ defined as
 \[
 M_1 \oplus M_2 = \{B_1 \cup B_2  \subseteq I\mid  B_1 \in \mathcal{B}(M_1), B_2 \in \mathcal{B}(M_2)\}.
 \]

The \textbf{Hopf monoid of matroids} $\mbf{M}$ is given by  $\mbf{M}[I] = \F\{ \text{matroids on $I$}\}$ where

\noindent $\bullet$ The product of $M_1\in\mbf{M}[S]$ and $M_2\in\mbf{M}[T]$ 
is their direct sum $M_1 \oplus M_2$.

\noindent $\bullet$  
The coproduct of $M \in \mbf{M}[I]$ is  $\Delta_{S,T}(M) = M|_S \otimes M/_S$.

\medskip

The map that sends a matroid $M$ to its matroid polytope $P(M)$ is an inclusion of Hopf monoids:

\begin{theorem}\cite{AA17} The Hopf monoid of matroids $\mathbf{M}$ is a submonoid of the Hopf monoid $\mbf{GP}$.
\end{theorem}

We will often use the following iterated coproduct formula for matroids.

\begin{lemma}\label{lemma: iterated coproduct matroids}\cite{AA17} Let $S_1 \sqcup \cdots \sqcup S_k = I$ be a set decomposition and let $F_i = S_1 \sqcup \cdots S_i$ for $0 \leq i \leq k$. Then for any matroid $M$ on $I$,
 \[
 \Delta_{S_1, \ldots, S_k}(M) = 
 M[F_0,F_1] \otimes M[F_1,F_2] \otimes \cdots \otimes M[F_{k-1},F_k],
  \]
  where $M[A,B] = (M|_B)/_A$ for $\emptyset \subseteq A \subseteq B \subseteq I$.
\end{lemma}

\bigskip
\noindent \textbf{\textsf{(Weighted) posets and preposets.}} 
For posets $p$ on $S$ and $q$ on $T$ let $p \sqcup q$ denote the disjoint union of the posets on $S \sqcup T$. For a poset $p$ on $I$ and $S \subset I$, let $p|_S$ be the poset restricted to the set $S$. We say that $S$ is a \textbf{lower ideal} of $p$ if for any $x \leq y$ we have that $y \in S$ implies that $x \in S$. We make the same definitions for preposets as well.

\medskip

The \textbf{Hopf monoid of (pre)posets} is the species $\mbf{(P)P}$ with product and coproduct
 \[ 
 m_{S,T}(p,q) = p \sqcup q, \qquad  \qquad\Delta_{S,T}(p) = \begin{cases}
    p|_S \otimes p|_T & \text{if $S$ is a lower ideal of $p$} \\
    0 & \text{otherwise}.
 \end{cases}
 \]
Similarly, the \textbf{Hopf monoid of weighted (pre)posets} $\mbf{w(P)P}$ has product and coproduct
 \[ 
 m_{S,T}((u,p),(v,q)) = ((u,v), p \sqcup q), \qquad \Delta_{S,T}(w,p) = \begin{cases}
    (w,p)|_S \otimes (w,p)|_T & \text{if $S$ is a lower ideal of $p$} \\
    0 & \text{otherwise}.
 \end{cases}
 \]

\begin{theorem}\label{thm:CGP}\cite{AA17} 
The maps 
 \[ 
\operatorname{cone}(p) = \operatorname{cone}(e_i - e_j \, | \,  i \geq_p j) , \qquad 
\operatorname{cone}(w,p) = w^p + \operatorname{cone}(p)
 \]
are isomorphisms of Hopf monoids
\[
\mbf{PP} \cong \mbf{CGP}_0^+, \qquad
\mbf{wPP} \to \mbf{CGP}^+, \qquad
\mbf{P} \cong \mbf{PCGP}_0^+, \qquad
\mbf{wP} \cong \mbf{PCGP}^+.
\]
\end{theorem}

The first part of this statement is \cite[Proposition 3.4.6]{AA17} while the others are simple modifications of it. We will sometimes identify (weighted) (pre)posets and their cones, and identify the Hopf monoids $\mbf{(w)(P)P}$ and $\mbf{(P)CGP_{(0)}}^+$, without saying so explicitly.

We have the following consequence.

\begin{corollary} \label{cor:proj}
The maps
\begin{eqnarray*}
\mbf{CGP}^+[I] &\rightarrow& \mbf{CGP_{0}}^+[I] \\
\operatorname{cone}(w,p) & \longmapsto & \operatorname{cone}(p)
\end{eqnarray*}
that shift the apex of a conical generalized permutahedron to the origin give morphisms of Hopf monoids $\mbf{CGP}^+ \rightarrow \mbf{CGP_{0}}^+$ and $\mbf{PCGP}^+ \rightarrow \mbf{PCGP_{0}}^+$.
\end{corollary}

\bigskip
\noindent \textbf{\textsf{Ordered set partitions.}} 
If $\ell=\ell_1| \cdots | \ell_k$ is an ordered set partition of $I$ and $S$ is a subset of $I$, then the \textbf{restriction} $\ell|_S$ of $\ell$ to $S$ is obtained from $(\ell_1 \cap S) | \cdots | (\ell_k \cap S)$ by removing all empty blocks. For example,
\[
(149|278|6|35) \, |_{\{1,2,3,4\}} = 14|2|3.
\]
Say an ordered set partition $n$ on $I$ is a \textbf{quasi-shuffle} of ordered set partitions $\ell$ and $m$ on $S$ and $T$, respectively, if $\ell=n|_S$ and $m = n|_T$. In particular, every part of $n$ is either a block of $\ell$, or a block of $m$, or a union of a block of $\ell$ and a block of $m$. 
For example,
\[
149|278|6|35 \textrm{ is a quasishuffle of } 14|2|35 \textrm{ and } 9|78|6.
\]

The \textbf{Hopf monoid of ordered set partitions} $\mbf{\Sigma^*}$ is the species of  ordered set partitions with multiplication given by 
\[ 
\ell \cdot \mathcal{m} m = \sum_{\substack{n \textrm{ quasishuffle} \\ \textrm{of $\ell$ and $m$}}} n
\]
and comultiplication given by
\[ 
\Delta_{S,T}(\ell_1 | \cdots | \ell_k) = 
\begin{cases}
        (\ell_1 | \cdots | \ell_j) \otimes (\ell_{j+1} | \cdots | \ell_k) 
        & \text{if $S=\ell_1 \sqcup \cdots \sqcup \ell_j$ for $0 \leq j \leq k$, } \\
        0 & \text{otherwise}.
    \end{cases} 
\]
This Hopf monoid is introduced in \cite[Proposition 12.20]{AM10}, after setting $q=1$.

\bigskip
\noindent \textbf{\textsf{Weighted ordered set partitions.}} We close this section by introducing a Hopf monoid that will play a central role in this project.
Say a weighted ordered set partition $(w,n)$ on $I$ is a \textbf{quasi-shuffle} of $(u,\ell)$ and $(v,m)$ on $S$ and $T$, respectively, if $n$ is a quasishuffle of $\ell$ and $m$, and
\[
w(n_i) = \begin{cases}
u(\ell_a) & \text{ if } n_i = \ell_a \\
v(m_b) & \text{ if } n_i = m_b \\
u(\ell_a) + v(m_b) & \text{ if } n_i = \ell_a \sqcup m_b 
\end{cases}
\]
for each block $n_i$ of $n$. 
For example,
\[
((a+d,b+e,f,c),149|278|6|35) \textrm{ is a quasishuffle of } ((a,b,c),14|2|35) \textrm{ and } ((d,e,f), 9|78|6)
\]

\begin{definition} \label{def:wSigma*} The \textbf{Hopf monoid of weighted ordered set partitions} $\mbf{w\Sigma^*}$ is the Hopf monoid given by the species $\mbf{w\Sigma^*}[I] =  \F\{ \text{weighted ordered set partitions on $I$}\}$ with multiplication given by 
\[ 
(u, \ell) \cdot (v,m) = \sum_{\substack{(w,n) \textrm{ quasishuffle} \\ \textrm{of $(u,\ell)$ and $(v,m)$}}} (w,n)
\]
and comultiplication given by
\[ 
\Delta_{S,T}((w,\ell_1 | \cdots | \ell_k)) = 
\begin{cases}
        (w|_S,\ell_1 | \cdots | \ell_j) \otimes (w/_S, \ell_{j+1} | \cdots | \ell_k) 
        & \text{if $S=\ell_1 \sqcup \cdots \sqcup \ell_j$ for $0 \leq j \leq k$, } \\
        0 & \text{otherwise},
    \end{cases} 
\]
where $w|_S$ and $w/_S$ are the restrictions of $w$ to $\{\ell_1, \ldots, \ell_j\}$ and $\{\ell_{j+1}, \ldots, \ell_k\}$, respectively.
\end{definition}

One may verify directly that $\mbf{w\Sigma^*}$ satisfies the axioms of a Hopf monoid, but we will prove it by interpreting this Hopf monoid geometrically in Theorem \ref{thm:quotient}.
Naturally, we have the following projection map of Hopf monoids:
\begin{eqnarray*}
\mbf{w\Sigma^*} &\rightarrow& \mbf{\Sigma^*} \\
(w,\ell) & \longmapsto & \ell.
\end{eqnarray*}

\section{The Brianchon-Gram, aligning, and Brion morphisms}\label{sec:morphisms}

In this section we introduce two Hopf morphisms: the Brianchon-Gram morphism on extended generalized permutahedra and the aligning morphism on (weighted) preposets. They will play a key role in our proof of Theorem \ref{thm:quotient} which states that the Hopf monoid structure on $\mbf{GP}^+$ descends to the quotient $\val(\mbf{GP}^+) = \mbf{GP}^+/\mbf{ie}$.

\subsection{The Brianchon-Gram morphism}

For a polyhedron $P \subseteq \R^I$ and a linear functional $w \in (\R^I)^*$, we define $P_w$ as the face of $P$ maximized by $w$, with the convention that $P_w = \emptyset$ if $P$ is unbounded in direction $w$.

For a polyhedron $P$ and a point $f \in P$, we define the \textbf{tangent cone} of $P$ at $f$ to be
\[
\operatorname{cone}_f(P) = \{f + x \, : \, f + \epsilon x \in P \textrm{ for all small enough } \epsilon > 0\}.
\]
For any face $F$ of $P$ we define
\[
\operatorname{cone}_F(P) = \operatorname{cone}_f(P) \qquad  \textrm{ for any } f \in \relint F;
\]
this does not depend on the choice of the point $f$ in the relative interior of $F$. \cite[Prop. 3.5.2]{BeckSanyal}.


Recall that a face $F$ of $P$ is relatively bounded if it is non-empty and $F/L$ is a bounded face of $P/L$ where $L$ is the lineality space of $P$. (For simplicity, we will  sometimes call such faces \textbf{bounded}.)
When a polyhedron $P$ has a lineality space $L$, we write
\begin{equation} \label{eq:dimF}
\Ldim P := (\textrm{dimension of } P) - (\textrm{dimension of } L).
\end{equation}

\begin{proposition}  \label{prop:bg}
The \textbf{Brianchon-Gram maps} $\bg[I]: \mbf{GP}^+[I] \rightarrow \mbf{CGP}^+[I]$, defined by
\[
\bg(P) = \sum_{\substack{F \leq P \textrm{ rel.} \\ \textrm{bounded}}}(-1)^{\Ldim F} \operatorname{cone}_F(P) \qquad \text{ for } P \in \mbf{GP}^+[I],
\]
where we sum over the relatively bounded faces $F$ of $P$, 
give a morphism of Hopf monoids.
\end{proposition}

Before proving this theorem, we need two technical lemmas about tangent cones.

\begin{lemma}\label{lem:tangentconeofproduct}
If $F$ is a face of a polyhedron $P$ and $G$ is a face of polyhedron $Q$ then $F \times G$ is a face of polyhedron $P \times Q$ and
\[
\operatorname{cone}_{F \times G}(P \times Q)  = \operatorname{cone}_{F}(P) \times  \operatorname{cone}_{G}(Q).
\]
\end{lemma}

\begin{proof} If $u$ and $v$ are linear functionals such that $F=P_u$ and $G=Q_v$ then the linear functional $(u,v)$ gives $F \times G = (P \times Q)_{(u,v)}$. Let us prove the claimed equality.

\smallskip

$\subseteq$: Consider an arbitrary point $h + z \in \operatorname{cone}_{F \times G}(P \times Q)$ where $h \in F \times G$ and $h + \epsilon z \in P \times Q$ for small $\epsilon >0$. Write $h = f+g$ and $z=x+y$ for $f, x \in \R^S$ and $g, y \in \R^T$. Since $h \in F \times G$, we have $f \in F$ and $g \in G$. Since $h+\epsilon z \in P \times Q$ for small $\epsilon$, we have $f + \epsilon x \in P$ and $g + \epsilon y \in Q$. It follows that $f + x \in \operatorname{cone}_F(P)$ and $g+y \in \operatorname{cone}_G(Q)$.

\smallskip

$\supseteq$: 
Conversely, consider points $f + x \in \operatorname{cone}_F(P)$ and $g+y \in \operatorname{cone}_G(Q)$. Since $f \in F$ and $g \in G$ we have $h:=f+g \in F \times G$. Since $f + \epsilon x \in P$ for small enough $\epsilon'>0$ and $g + \epsilon'' y \in Q$ for small enough $\epsilon''>0$, then $z:=x+y$ satisfies that $h+\epsilon z \in P \times Q$ for small enough $\epsilon>0$. We conclude that $h+z \in \operatorname{cone}_{F \times G}(P \times Q)$. 
\end{proof}

\begin{lemma}\label{lem:faceoftangentcone}
Let $P \subseteq \R^I $ be a polyhedron, $F$ a face of $P$, and $w \in \R^I$ a linear functional. Then
\[
(\operatorname{cone}_F(P))_w =
\begin{cases}
\emptyset &  \textrm{ if } F \nsubseteq P_w \\
\operatorname{cone}_F(P_w) & \textrm{ if } F \subseteq P_w 
\end{cases}
\]
\end{lemma}

\begin{proof} 
1. First consider the case $F \nsubseteq P_w$.
Assume contrariwise that $(\operatorname{cone}_F(P))_w \neq \emptyset$, and that the maximum value $m$ of the linear function $w$ in $\operatorname{cone}_F(P)$ is attained at a point $f+x$ where $f$ is in the face $F$ and $f+ \epsilon x \in P$ for small $\epsilon > 0$. 
For any $r>0$ we also have $f+rx \in \operatorname{cone}_F(P)$ and hence $m=w(f+x) \geq w(f+rx)$. This is only possible if $w(x) = 0$, so $w(f)=m$.

Since $F \subseteq P \subseteq \operatorname{cone}_F(P)$ and the $w$-maximum value of $\operatorname{cone}_F(P)$ is attained at $f \in F$, this must also be the $w$-maximum value of $P$. Therefore $f$ is in the $w$ maximal face $P_w$ of $P$.

Finally, for any other $f' \in F$ we have that $f' + r(f-f') \in \operatorname{cone}_F(P)$ for all $r>0$, so we must have $m \geq w(f'+r(f-f'))$; this is only possible if $w(f-f') \leq 0$, which implies $w(f')=m$; that is, $f' \in P_w$ as well. We conclude that $F \subseteq P_w$, a contradiction. This proves the first case.

\smallskip
2. Assume $F \subseteq P_w$ and let $m=w(f)$ for $f \in F$; this is the maximum value that $w$ takes in $P$.

\smallskip

$\subseteq$: Let $f+x \in (\operatorname{cone}_F(P))_w$ where $f \in F$ and $f+ \epsilon x \in P$ for small $\epsilon > 0$. As we saw above, this implies $w(x)=0$, so  $w(f+\epsilon x) = m$, and $f+\epsilon x \in P_w$ as well. This implies $f+x \in \operatorname{cone}_F(P_w)$ as desired.

\smallskip

$\supseteq$: Let $f+x \in \operatorname{cone}_F(P_w)$ where $f \in F$ and $f+\epsilon x \in P_w$ for small $\epsilon >0$. Since $f$ and $f+\epsilon x$ are in $P_w$, we have $w(f) = m$ and $w(x)=0$.

Now, $f+ \epsilon x \in P$ implies $f+x \in \operatorname{cone}_F(P)$. To show that $f+x$ is in the $w$-maximal face of this cone, consider any other point $f'+x' \in \operatorname{cone}_F(P)$ where $f' \in F$ and $f' + \epsilon x' \in P$. Then $w(f' + \epsilon x') \leq m = w(f')$, so $w(x') \leq 0$ and thus $w(f'+x') \leq m = w(f+x)$, as desired.
\end{proof}

With those lemmas at hand, we are now ready to prove that the Brianchon-Gram maps give a morphism of Hopf monoids.

\begin{proof}[Proof of Theorem \ref{prop:bg}] For any $P \in \mathbf{GP}^+[S]$ and $Q \in \mathbf{GP}^+[T]$ we have
\begin{eqnarray*} 
\bg(P) \cdot \bg(Q) 
&=& \left(\sum_{F \leq P} (-1)^{\Ldim F} \operatorname{cone}_F(P)\right)  \left(\sum_{G \leq Q} (-1)^{\Ldim G} \operatorname{cone}_G(Q)\right)  \\
&=& \sum_{F \times G  \leq P \times Q} (-1)^{\Ldim F \times G} \operatorname{cone}_{F \times G} (P \times Q) \\
&=& \bg(P \times Q),
\end{eqnarray*}
summing over bounded faces. Here we are using the fact that the bounded faces of $P \times Q$ are the products of a bounded face of $P$ and a bounded face of $Q$, combined with Lemma \ref{lem:tangentconeofproduct}. Thus the Brianchon-Gram maps preserve the monoid structure.

\smallskip
For the coproduct we have, for any $P \in \mathbf{GP}^+[I]$,
\begin{eqnarray*} 
\Delta_{S,T}(\bg(P)) &=& \Delta_{S,T}\left(\sum_{F \leq P} (-1)^{\Ldim F} \operatorname{cone}_F(P)\right) \\
&=& \sum_{F \leq P_{S,T}} (-1)^{\Ldim F} \Delta_{S,T}(\operatorname{cone}_F(P)) \\
\end{eqnarray*}
where each sum is over bounded faces, using the first part of Lemma \ref{lem:faceoftangentcone}. Every bounded face $F$ of $P_{S,T} = P|_S \times P/_S$ factors as $F=F|_S \times F/_S$ for a bounded face $F|_S$ of $P|_S$ and a bounded face $F/_S$ of $P/_S$, and every such pair of faces arises from a bounded face of $P_{S,T}$. We have
\[
(\operatorname{cone}_F(P))_{S,T} = 
\operatorname{cone}_F(P_{S,T}) = 
\operatorname{cone}_{F|_S \times F/_S}(P|_S \times P/_S) = 
\operatorname{cone}_{F|_S}(P|_S) \times \operatorname{cone}_{F/_S}(P/_S)
\]
combining Lemma  \ref{lem:faceoftangentcone} and the second part of Lemma \ref{lem:tangentconeofproduct}. Thus we may rewrite the equation above as
\begin{eqnarray*} 
\Delta_{S,T}(\bg(P)) 
&=& \sum_{\substack{F|_S \leq P|_S \\ F/_S \leq P/_S}} (-1)^{\Ldim F|_S} \operatorname{cone}_{F|_S}(P|_S) \otimes (-1)^{\Ldim F/_S} \operatorname{cone}_{F/_S}(P/_S)   \\
&=& \bg(P|_S) \otimes \bg(P/_S)
\end{eqnarray*}
as desired.
\end{proof}

\begin{lemma} \label{lem:bg^2} The Brianchon-Gram morphism $\bg$
\begin{enumerate}
\item
restricts to the identity on $\im(\bg) = \mbf{CGP}^+[I] \subset \mbf{GP}^+[I]$, and
\item
is idempotent: $\bg \circ \bg = \bg$.
\end{enumerate}
\end{lemma}

\begin{proof}
The first statement holds since the only relatively bounded face of a cone is its relative apex.
The second follows readily.
\end{proof}

For an extended generalized permutahedron $P$ and a bounded face $F$, let us write $\operatorname{preposet}_F(P)$ for the preposet $p$ such that
$\operatorname{cone}_F(P)$ is a translate of the cone of 
$\operatorname{preposet}_F(P)$. Recall \eqref{eq:dimF}.

\begin{corollary} 
The \textbf{combinatorial Brianchon-Gram maps}, defined by
\[
\cbg(P) = \sum_{F \leq P} (-1)^{\Ldim F} \operatorname{preposet}_F(P) \qquad \text{ for } P \in \mbf{GP}^+[I],
\]
where we sum over the relatively bounded faces $F$ of the polyhedron $P$, 
give a morphism of Hopf monoids $\cbg: \mbf{GP}^+ \rightarrow \mbf{PP}$.
\end{corollary}

\begin{proof}
This is the result of composing the Brianchon-Gram morphism with the projection map $\operatorname{cone}(w,p) \mapsto \operatorname{cone}(p)$ of Corollary \ref{cor:proj}.
\end{proof}

\subsection{The aligning morphism}

\begin{proposition} \label{prop:st}
The \textbf{aligning maps} $\al: \mbf{(w)PP}[I] \rightarrow \mbf{(w)\Sigma^*}[I]$ given by 
\begin{eqnarray*}
\al(p) = \sum_{\substack{\ell \text { prelin.} \\ \text{ext. of } p}} \,\ell && \qquad \textrm{ for a preposet $p$ on $I$},\\
\al(w,p) = \sum_{\substack{(v,\ell) \text { prelin.} \\ \text{ext. of } (w,p)}} \, (v,\ell)&&  \qquad \textrm{ for a weighted preposet $(w,p)$ on $I$}
\end{eqnarray*}
give morphisms of Hopf monoids $\mbf{PP} \rightarrow \mbf{\Sigma^*}$ and $\mbf{wPP} \rightarrow \mbf{w\Sigma^*}$ 
\end{proposition}

\begin{proof}
For any preposets $p$ on $S$ and $q$ on $T$ we have
\begin{eqnarray*}
\al(p \sqcup q) &=& \sum_{\ell \text{ prelin. ext. of }p \sqcup q} \, \ell \\
&=& \sum_{\substack{\ell_p \text{ prelin. ext. of } p  \\  \ell_q \text{ prelin. ext. of } q}} \,\,\,\,  \sum_{\ell \text{ quasishuffle of } \ell_p \textrm{ and } \ell_q}  \, \ell \\
&=& \sum_{\substack{ \ell_p \text{ prelin. ext. of } p \\ \ell_q \text{ prelin. ext. of } q }}  \, \ell_p \cdot \ell_q \\
&=& \al(p) \cdot \al(q),
\end{eqnarray*}
so $\al$ preserves the product.

\smallskip

To verify that $\al$ also preserves the coproduct, recall that the coproduct for (pre)posets is
\[
\Delta_{S,T}(q) = 
\begin{cases}
q|_S \otimes q|_T & \textrm{ if $S$ is a lower ideal of $q$,} \\
0 & \textrm{ otherwise}.
\end{cases}
\]
for a preposet $q$ on $I$.
If $S$ is not a lower ideal of $q$ then $S$ is not a lower ideal of any prelinear extension $\ell$ of $q$ either, so 
\[
\al_S \otimes \al_T (\Delta_{S,T}(q)) = 0
\qquad \textrm{ and } \qquad 
\Delta_{S,T}(\al_I(q)) = \sum_{\substack{\ell \text { prelin.} \\ \text{ext. of } q}} 
\Delta_{S,T}(\ell)  
= 0 
\]
If $S$ is a lower ideal of $q$, then there are two possibilities for a prelinear extension $\ell$ of $q$. If $S$ is not a lower ideal of $\ell$ then $\Delta_{S,T}(\ell)=0$. If $S$ is a lower ideal of $\ell$, then $\ell$ is the ordinal sum of $\ell|_S$ and $\ell|_T$, and every combination of prelinear extensions $\ell|_S$ and $\ell|_T$ of $q|_S$ and $q|_T$ arises from such an $\ell$. Thus 
\begin{eqnarray*}
\Delta_{S,T}(\al(q)) 
&=& \sum_{\ell \textrm{ prelin. ext. of } q}  \Delta_{S,T}(\ell) \\ 
&=& \sum_{\substack{\ell|_S \textrm{ prelin. ext. of } q|_S \\ \ell|_T \textrm{ prelin. ext. of } q|_T}} \, \ell|_S \otimes \ell|_T  \\ 
&=& \al(q|_S) \otimes \al(q|_T).
\end{eqnarray*}
The result follows. The weighted version of the statement holds by an analogous argument.
\end{proof}

We record two observations that are readily verified.

\begin{lemma}\label{lem:sgn}
The sign map $\sgn(p) = (-1)^{|p|}p$ is an automorphism of the Hopf monoid $\mbf{PP}$ of preposets.
\end{lemma}

\begin{lemma} \label{lem:st^2} The aligning map $\al$
\begin{enumerate}
\item
restricts to the identity on $\im(\al) = \mbf{(w)\Sigma^*} \subset \mbf{(w)PP}$, and
\item
is idempotent:  $\al \circ \al = \al$.
\end{enumerate}
\end{lemma}

\begin{proof}
The first statement holds since the only prelinear extension of a totally ordered preposet is that preposet itself. The second follows readily.
\end{proof}

It is worth remarking that the inclusion in Lemma \ref{lem:st^2}.1 above is an inclusion of comonoids that is incompatible with the monoid structures. Therefore the map $\al \circ \al$ of Lemma \ref{lem:st^2}.2 is well-defined, but it is not a Hopf morphism.

\subsection{The Brion morphism}

We conclude this section with the definition of a morphism with nice combinatorial properties that will be analyzed in a future project. 
For an extended generalized permutahedron $P$ and a vertex $v$, let us write $\operatorname{poset}_v(P)$ for the poset $p$ that is a translate of 
$\operatorname{cone}_v(P)$.

\begin{proposition}  \label{prop:brion}
The \textbf{Brion maps} $\b[I]: \mbf{GP}^+[I] \rightarrow \mbf{P}[I]$, defined by
\[
\b(P) = \sum_{v \textrm{ vertex of } P} \operatorname{poset}_v(P)
  \qquad  \text{ for } P \in \mbf{GP}^+[I]
\]
give a morphism of Hopf monoids 
from  $\mbf{GP}^+$ to $\mbf{P}$.
\end{proposition}

\begin{proof}
On preposets, consider the maps $ \mbf{(w)PP}[I] \rightarrow \mbf{(w)P}[I]$ that is the identity map on posets and the zero map on all other  preposets. One readily verifies that these are morphisms, and the combinatorial Brion morphism is obtained by composing the combinatorial Brianchon morphism with them.
\end{proof}

The  Brion morphism has several interesting combinatorial and algebraic properties that will be the subject of an upcoming paper.
Naturally, there is also a geometric Brion map from $\mbf{GP}^+[I]$ to $\mbf{PCGP}^+[I]$.

\section{The indicator Hopf monoid of generalized permutahedra}\label{sec:indicator}

\subsection{Valuations}\label{subsec:valuations}

Valuations are combinatorial abstractions of measures and have played an important role in various aspects of convex geometry and polyhedral geometry. One might wish to require that a measure of a polytope should behave well with respect to subdivisions and with respect to indicator functions, in the following sense.

A \textbf{polyhedral subdivision} of $P$ in $\mathcal{P}$ is a collection of polyhedra $\{P_i\}$ in $\mathcal{P}$ such that $\bigcup P_i = P$, any two polytopes $P_i$ and $P_j$ intersect in a common face $P_i \cap P_j = P_k$ that is in the collection, and every non-maximal $P_i$ is the intersection of maximal $P_j$s in the collection.

Let $\mathcal{P}$ be a set of polyhedra in $\R^n$ and $\F$ be a field of characteristic $0$. For any polyhedron $P \in \mathcal{P}$, its \textbf{indicator function} $\one_{P}: \R^n \to \F$ is the function defined by
 \[ \one_P(x) = \begin{cases}
 1 & \text{if $x \in P$} \\
 0 & \text{otherwise.}
 \end{cases}\]
Let $\val(\mathcal{P})$ denote the vector space over $\F$ spanned by the indicator functions $\one_{P}$ for $P \in \mathcal{P}$.

\begin{definition} Let $A$ be an abelian group. A function $f: \mathcal{P} \to A$ is

\begin{enumerate}
    \item a \textbf{weak valuation} if for any polytopal subdivision $\{P_i\}$ of $P$ we have
     \[ f(P) = \sum_{i} (-1)^{\dim P - \dim P_i} f(P_i). \]
    \item a \textbf{strong valuation} if it factors through the map $P \mapsto \one_P$, that is, there exists a (necessarily unique) linear function $\hat{f}: \val(\mathcal{P}) \to A$ such that for all $P \in \mathcal{P}$ we have
    \[
    f(P) = \hat{f}(\one_{P}).
    \]
\end{enumerate}

\end{definition}

These notions are illustrated in Examples \ref{ex:u36 subdivision} and \ref{ex:posetvaluation}, which show some of the relations satisfied by two weakly valuative functions: the Kazhdan-Lusztig polynomial of a matroid and the Poincar\'e polynomial of a poset, respectively.

A strong valuation is always a weak valuation, but the converse is not true in general. Derksen and Fink proved that when $\mathcal{P}$ is the set of generalized permutahedra, the situation is better.

\begin{theorem}\cite{DF10}\label{thm: GP strong weak equiv}
Let $\mathcal{P}$ be the family of extended generalized permutahedra, 
the family of generalized permutahedra, the family of matroid polytopes, or any family of polyhedra closed under intersections. Then for any function $f$ on $\mathcal{P}$,
\[
\textrm{$f$ is a strong valuation} \quad \Longleftrightarrow \quad \textrm{$f$ is a weak valuation}.
\]
\end{theorem}

An important example of a strong valuation is the constant function.

\begin{proposition}\cite{DF10}\label{prop: Euler characteristic} The function $f: \mbf{GP}[I] \to A$ which equals $1$ on all generalized permutahedra $P$ is a strong valuation. In particular, for any subdivision $\mathcal{P}$ of $P$ we have
 \[ 
 \sum_{P_i \in \mathcal{P}} (-1)^{\dim P - \dim P_i} = 1.
 \]
\end{proposition}

\subsection{The inclusion-exclusion subspace of generalized permutahedra}


For each finite set $I$ let $\one[I]: \mbf{GP}^+[I] \rightarrow \operatorname{Hom}(\R^I, \F)$ be the linear map that sends an extended generalized permutahedron $P$ in $\R^I$ to its indicator function $\one_P$ which is equal to $1$ in $P$ and to $0$ outside of $P$.

\begin{definition}
The \textbf{inclusion-exclusion species} consists of the vector subspaces
\begin{eqnarray*}
\mbf{ie}[I] 
&=&
\espan\left(P - \sum_{P_i \in \mathcal{P}} (-1)^{\dim P - \dim P_i} P_i \, : \,
\mathcal{P} \textrm{ is an ext. gen. perm. subdivision of } P \right)  \subset \mbf{GP}^+[I] \\
&=& \ker(\one[I])
\end{eqnarray*}
for each finite set $I$ and the natural maps between them. 
\end{definition}

The equivalence of these definitions is guaranteed by Theorem \ref{thm: GP strong weak equiv}.
To construct a convenient generating set for the inclusion-exclusion subspace $\mbf{ie}[I]$ of $\mbf{GP}^+[I]$, we recall the Brianchon-Gram theorem and prove a lemma about preposet cones.

\begin{theorem}(Brianchon-Gram Theorem)\cite{Bri37, DF10, Gra74}\label{thm:Brianchon Gram}
Let $P$ be a polyhedron.
Then
        \[
        \one_P = \sum_{F \leq P} (-1)^{\Ldim F}\one_{\operatorname{cone}_F(P)}
        \]
summing over the relatively bounded faces $F$ of $P$.
\end{theorem}

\noindent
In the statement above it is important to recall \eqref{eq:dimF}; this minor but necessary adjustment is missing from the original sources.

\begin{figure}[h]
\centering
\includegraphics[width=15cm]{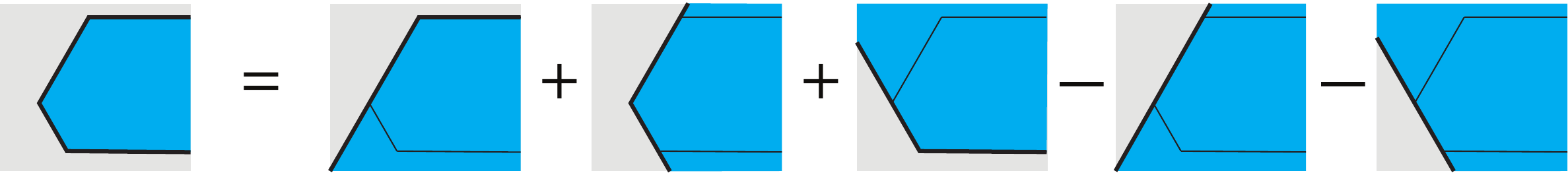}\qquad 
\end{figure}

\begin{lemma} \label{lemma:aligning}
For any preposet $q$  we have
\[
\one_{\operatorname{cone}(q)} = \sum_{\ell \textrm{ prelin. ext. of } q}  (-1)^{|q| - |\ell|} \one_{\operatorname{cone}(\ell)},
\]
where $|r|$ denotes the number of equivalence classes of elements of the preposet $r$.
\end{lemma}

\begin{figure}[h]
\centering
\includegraphics[width=15cm]{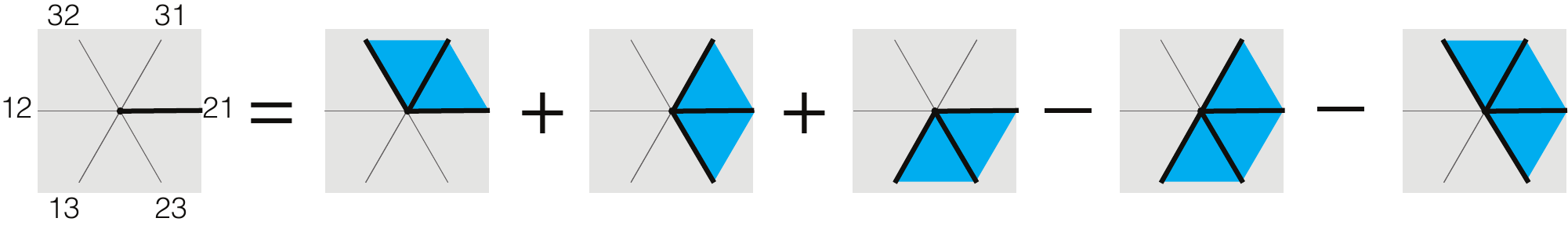}\\
\quad \,  \includegraphics[width=14cm]{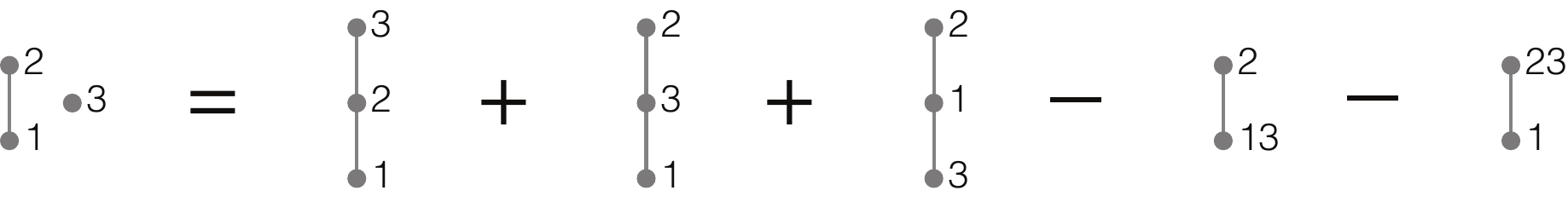}
\caption{A preposet cone equals the alternating sum of the cones of its prelinear extensions.\label{fig:cones}}
\end{figure}

\begin{proof}

Recall that the dual cone of a cone $C$ in $\R^d$ is $C^\Delta := \{d \in \R^I \, : \, \langle c,d \rangle \geq 0 \textrm{ for all } c \in C\}$, and we have $C^{\Delta\Delta} = C$. 
The dual cone to $\operatorname{cone}(q)$ is the \textbf{braid cone} $\sigma_q = \operatorname{cone}(q)^\Delta$ consisting of the points $x \in \R^I$ whose coordinates satisfy the relations of $q$, that is, $x_a \geq x_b$ whenever $a \geq b$ in $q$. 

The braid arrangement subdivides  the braid cone $\operatorname{cone}(q)^\Delta$ into the braid cones $\operatorname{cone}(\ell)^\Delta$ for the prelinear extensions $\ell$ of $q$: the relative interior of each subdividing braid cone consists of the points in $\operatorname{cone}(q)^\Delta$ whose coordinates are in a fixed relative order. The dimension of $\operatorname{cone}(q)^\Delta$ is $|q|$, so inclusion-exclusion gives the analogous relation for the braid cones
\begin{equation} \label{eq:braidcones}
\one_{\operatorname{cone}(q)^\Delta} = \sum_{\ell \textrm{ prelin. ext. of } q}  (-1)^{|q| - |\ell|} \one_{\operatorname{cone}(\ell)^\Delta}.
\end{equation}
This is illustrated in Figure \ref{fig:dualcones}, which is dual to Figure \ref{fig:cones}.

\begin{figure}[h]
\centering
\includegraphics[width=15cm]{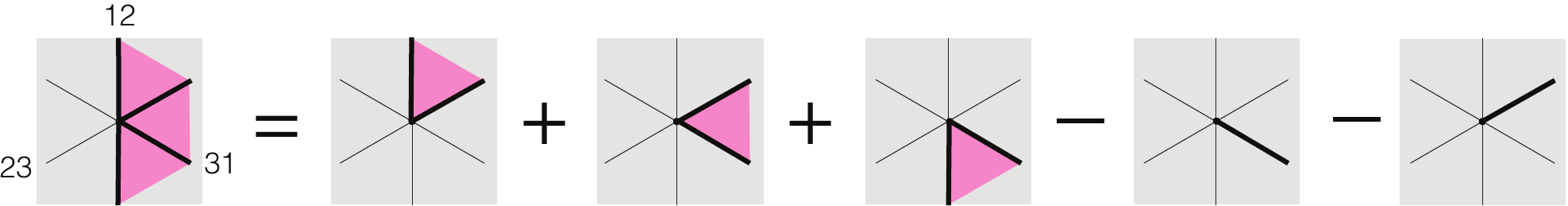}
\caption{A braid cone equals the alternating sum of the braid cones of its prelinear extensions.\label{fig:dualcones}}
\end{figure}

Now we show that the dual equation \eqref{eq:braidcones} implies the desired equation. For $x \in \R^I$ we have
\[
\one_{C}(x) = 1 \,\,  \Longleftrightarrow \,\,
x \in C \,\,  \Longleftrightarrow \,\,
 \langle d,x \rangle \leq 0 \textrm{ for all } d \in C^\Delta \,\,  \Longleftrightarrow  \,\,
 C^\Delta \cap H_x = \emptyset \,\,  \Longleftrightarrow  \,\,
 j_x(C^\Delta) = 1,
\]
where $H_x = \{ d \in \R^I \, : \, \langle d,x \rangle > 0\}$ is an open half-space and $j_x$ is the function on polyhedra given by $j_x(D) = 1$ if $D \cap H_x = \emptyset$ and $j_x(D) = 0$ otherwise. The function $j_x$ is a weak valuation for polyhedral subdivisions \cite[Prop. 4.5]{AFR}\footnote{In \cite{AFR} this statement is proved for functions on matroid polytopes, but the same proof applies to this setting.}, so for any $x \in \R^I$ we have
\[
\left(\one_{\operatorname{cone}(q)} - \sum_{\substack{\ell \textrm{ prelin.} \\ \textrm{ext. of } q}} (-1)^{|q|-|\ell|} \one_{\operatorname{cone}(\ell)}\right)(x)  \\
= j_x(\operatorname{cone}(q)^\Delta) - \sum_{\substack{\ell \textrm{ prelin.} \\ \textrm{ext. of } q}} (-1)^{|q|-|\ell|} j_x(\operatorname{cone}(\ell)^\Delta) = 0.
\]
This proves the desired result.
\end{proof}

For each polyhedron $P \in \mbf{GP}^+[I]$, define the \textbf{Brianchon-Gram generator} of $P$ to be
\[
P - \bg(P) := P - \sum_{F \leq P} (-1)^{\Ldim F}(\operatorname{cone}_F(P)) \,  \in \mbf{GP}^+[I]
\]
where the sum is taken over the relatively bounded faces $F$ of $P$. 
For each cone $C  \in \mbf{GP}^+[I]$ there is a preposet $q$ on $I$ and a translation vector $v$ such that $C = v + \operatorname{cone}(q)$; define the corresponding \textbf{aligning generator} of $C$ to be
\[
C - \al^-(C) := (v+\operatorname{cone}(q)) - \sum_{\ell \textrm{ prelin. ext. of } q}  (-1)^{|q| - |\ell|} (v+\operatorname{cone}(\ell)).
\]
Note that the summands on the right hand summation are plates.


\begin{lemma} \label{lem:gens}
The inclusion-exclusion subspace $\mbf{ie}[I]$ is generated by the Brianchon-Gram generators of the polyhedra $P \in \mbf{GP}^+[I]$ and the aligning generators of the cones $C \in \mbf{CGP}^+[I]$.
\end{lemma}

\begin{proof}
Theorem \ref{thm:Brianchon Gram} and Lemma \ref{lemma:aligning} tell us that these  these generators are in $\mbf{ie}[I]$. 
Now consider any element $a \in \mbf{ie}[I]$ of the inclusion-exclusion subspace. 
Using the Brianchon-Gram generator for each polyhedron appearing in $a$, and then the aligning generator for each resulting cone, we can write $a = b + c + d$ where $b$ is a linear combination of Brianchon-Gram generators, $c$ is a linear combination of aligning generators, and 
\[
d = \sum_{i=1}^n \lambda_i (v_i + \operatorname{cone}(\ell_i))
\]
is a linear combination of plates.
Since $d=a-b-c \in \mbf{ie}[I]$, we have
\[
\sum_{i=1}^n \lambda_i \one_{v_i + \operatorname{cone}(\ell_i)} = 0.
\]
Now, \cite[Theorem 2.7]{ESS} states that the indicator functions $\one_{v + \operatorname{cone}_F(P)}$ of the translates of the tangent cones $\operatorname{cone}_F(P)$ of a polytope are linearly independent. For the permutahedron, these translates are precisely the permutahedral plates, so they are linearly independent. We conclude that  $d=0$, and the desired result follows.
\end{proof}

\subsection{The indicator Hopf monoid of extended generalized permutahedra}

Let the species $\val(\mbf{GP}^+)$ of indicator functions on $\mbf{GP}^+$ consist of the vector spaces
\begin{eqnarray*}
\val(\mbf{GP}^+)[I] & := & \espan\{\one_P \, | \, P \textrm{ is a generalized permutahedron in } \R^I\} \subset \operatorname{Hom}(\R^I, \F)\\
& \cong & \mbf{GP}^+[I] / \mbf{ie}[I].
\end{eqnarray*}
for each finite set $I$ and the natural maps between them. We say that a species morphism $\mbf{f}: \mbf{GP}^+ \to \mbf{A}$ is a \textbf{strong valuation} if the map $\mbf{f}[I]$ is a strong valuation for all $I$.

Similarly, for a subspecies $\mbf{H}$ of $\mbf{GP}^+$ define
the species $\val(\mbf{H})$ of indicator functions on $\mbf{H}$ by
\begin{eqnarray*}
\val(\mbf{H})[I] & := & \espan\{\one_P \, | \, P \in \mbf{H}[I]\} \\
& \cong & \mbf{H}[I] / (\mbf{ie}[I] \cap \mbf{H}[I]).
\end{eqnarray*}
Every strong or weak valuation of $\mbf{GP}^+$ restricts to a strong or weak valuation of $\mbf{H}$, respectively, since $\val(\mbf{H})$ is a subspecies of $\val(\mbf{GP}^+)$. 
However, the classes of strong and weak valuations may no longer agree in $\mbf{H}$.
Thus $\mbf{ie}[I] \cap \mbf{H}[I]$ may not be generated by elements of the form $P - \sum_{P_i \in \mathcal{P}} (-1)^{\dim P - \dim P_i} P_i$ for subdivisions $\{P_i\}$ of $P$ with $P_i, P \in \mbf{H}[I]$.

%
%

In this section we prove our main structural Hopf theoretic result (Theorem \ref{mainthm1}), that the Hopf monoid structure on $\mbf{GP}^+$ descends to the quotient $\mathbb{I}(\mbf{GP}^+)$.

\begin{theorem}\label{thm:quotient} 
Let $\mbf{GP}^+$ be the Hopf monoid of extended generalized permutahedra.
\begin{enumerate}
\item
The inclusion-exclusion species $\mbf{ie}$ is a Hopf ideal of $\mbf{GP}^+$.
\item 
The quotient $\val(\mbf{GP}^+) \cong  \mbf{GP}^+/\mbf{ie}$ is a Hopf monoid. 
\item
The resulting \textbf{indicator Hopf monoid of extended generalized permutahedra} is isomorphic to the Hopf monoid of weighted ordered set partitions: 
\[
\val(\mbf{GP}^+) \cong \mbf{w\Sigma^*}.
\]
\item
For any Hopf submonoid $\mbf{H} \subseteq \mbf{GP}^+$, the subspecies $\val(\mbf{H}) \subseteq \val(\mbf{GP}^+)$ is a Hopf quotient of $\mbf{H}$, namely, $\val(\mbf{H}) \cong \mbf{H}/(\mbf{ie} \cap \mbf{H})$.
\end{enumerate}
\end{theorem}

\begin{proof}
Consider the composition $\varphi$ of the Brianchon-Gram morphism (Proposition \ref{prop:bg}), the isomorphism between conical generalized permutahedra and weighted preposets (Theorem \ref{thm:CGP}), the sign automorphism of weighted preposets (Lemma \ref{lem:sgn}), and the aligning morphism (Proposition \ref{prop:st}) as follows:
\begin{eqnarray*} 
\varphi &:& \mbf{GP}^+ \xrightarrow{\bg} \mbf{CGP}^+ \xrightarrow[\cong]{\operatorname{cone}^{-1}} \mbf{wPP} \xrightarrow{\sgn} \mbf{wPP} \xrightarrow{\al} \mbf{w\Sigma^*} \\
& & P  \longmapsto (-1)^{|I| - \dim \text{Lin}(P)} \sum_{\substack{F \leq P \\ \operatorname{cone}_F(P) = \operatorname{cone}(w,q)}}  \sum_{\substack{(v,\ell) \text{ prelin. ext. } \\ \text{ of } (w,q)}} (v,\ell),
\end{eqnarray*}
for a generalized permutahedron $P \in \mbf{GP}^+[I]$ with lineality space Lin$(P)$. To verify the correctness of the sign, notice that the sign on a summand $(v,\ell)$ in $\varphi(P)$ is $(-1)^{\Ldim F}(-1)^{|q|}$, recall \eqref{eq:dimF}, observe that $\text{Lin}(P) = \text{Lin}(F)$, and notice that the dimension of $F$ equals the dimension of the lineality space of $\operatorname{cone}_F(P) = \operatorname{cone}(w,q)$, which is $|I|-|q|$. 

Recall that plates are in bijection with weighted ordered set partitions, and notice that each plate $d=\operatorname{cone}(v,\ell)$ satisfies $\varphi(d) = (-1)^{|l|}(v,\ell)$, so $\varphi$ is surjective. 
We claim that the kernel of $\varphi$ is the inclusion-exclusion species:
\[
\ker \varphi = \mbf{ie}
\]

\smallskip

\noindent
$\supseteq$: Since $\bg \circ \bg = \bg$ by Lemma \ref{lem:bg^2}, we have
\[
\varphi(P-\bg(P)) = 0
\]
for all polyhedra $P$, so all Brianchon-Gram generators are in the kernel of $\varphi$.
Also notice that for any cone $C = \operatorname{cone}(q)$ (assuming the apex of $C$ contains $0$, without loss of generality), we have 
\[
C-\al^-(C) \xrightarrow{\bg}
C-\al^-(C) \xrightarrow[\cong]{\operatorname{cone}^{-1}}
q - \sum_{\substack{\ell \textrm{ prelin.} \\ \textrm{ ext. of } q}}  (-1)^{|q| - |\ell|} \ell  \xrightarrow{\sgn}
 (-1)^{|q|} (q - \al(q))  \xrightarrow{\al}
0
\]
 since $\al$ is idempotent, so
 \[
 \varphi(C-\al^-(C)) = 0
 \]
 as well.
Therefore, in light of Lemma \ref{lem:gens}, the inclusion-exclusion species is in the kernel of $\varphi$.


\smallskip

\noindent
$\subseteq$: 
Consider any element $a \in \ker \varphi$. We can use the Brianchon-Gram generators to rewrite each summand of $a$ in terms of its affine tangent cones, and then the aligning relations to write each of those cones in terms of plates. Therefore we have
\[
a=b+c+p
\]
where $b$ is a linear combination of Brianchon-Gram generators, $c$ is a linear combination of aligning generators, and $p$ is a linear combination of plates.
Since $a,b,c \in \ker \varphi$ we have $p \in \ker \varphi$ so $\varphi(p)=0$. 
But each plate $d=\operatorname{cone}(v,\ell)$ satisfies $\varphi(d) = (-1)^{|l|} (v,\ell)$ and there are no linear relations among weighted ordered set partitions in $\mathbf{w\Sigma^*}$, so in fact we must have $p = 0$ and $a = b+c \in \mbf{ie}$, as desired.

\bigskip

Claims 1--3 then follow from the First Isomorphism Theorem of Hopf moniods, which we prove in Theorem \ref{thm:Noether}. Furthermore, Claim 2 implies that the projection map $\one_{(-)}: \mbf{GP}^+[I] \to \mathbb I(\mbf{GP}^+)$ which sends a polytope to its indicator function is a Hopf monoid morphism. The restriction $\one_{(-)}|_{\mbf{H}}$ of this morphism to the submonoid $\mbf{H}$ has image $\mathbb I(\mbf{H})$ and kernel $(\mbf{ie} \cap \mbf{H})$, so Claim 4 follows by applying the First Isomorphism Theorem to this morphism.
\end{proof}

\begin{corollary}
The antipode of the indicator Hopf monoid of generalized permutahedra $\val(\mbf{GP}^+)$ is given by 
\[
s_I(\one_{P}) = (-1)^{|I|-\dim P} \one_{P^\circ} \qquad \text{ for } P \in \mbf{GP}^+[I], 
\]
where $P^\circ$ is the relative interior of $P$.
\end{corollary}

\begin{proof}
Aguiar and Ardila \cite{AA17} showed that the antipode in $\mbf{GP}^+$ is given by
\[
s_I(P) = (-1)^{\abs{I}}\sum_{Q \leq P} (-1)^{\dim Q} \, Q
\]
summing over the faces of $P$. Using the inclusion-exclusion relations that hold in the quotient $\val(\mbf{GP}^+)$, this simplifies to the desired result.
\end{proof}

For an extended generalized permutahedron $P$ in $\R^I$ and an ordered set partition $\ell$ of $I$, let the $\ell$-maximal face $P_\ell$ of $P$ be the $f$-maximal face $P_f$ for any vector $f$ whose entries are in the same relative order as $\ell$, that is, 
$f_i < f_j$ for $i \in \ell_a, j \in \ell_b, a < b$ and 
$f_i = f_j$ for $i,j \in \ell_a$.

\begin{proposition} \label{prop:phi}
The isomorphism of Theorem \ref{thm:quotient}.3 is realized by the map
\begin{eqnarray*}
\varphi: \mathbb{I}(\mbf{GP}^+) & \xrightarrow{\,\, \cong \,\, } &  \mbf{w\Sigma^*} \\
\one_P  &\longmapsto&   (-1)^{|I| - \dim \mathrm{Lin}(P)}\sum_{\substack{\ell \, : \, P_\ell \text{ is } \\ \text{ rel. bounded}}} (v_{P,\ell},\ell).
\end{eqnarray*}
for an extended generalized permutahedron $P \in \mbf{GP}^+[I]$, where the $\ell$-maximal face $P_\ell$ of $P$ lies on the subspace given by equations $\sum_{l \in \ell_i} x_l = v_{P,\ell}(\ell_i)$ for each block $\ell_i$ of $\ell$.
In particular, for the plate $P=\text{cone}(w,\ell)$ of a weighted ordered set partition $(w,\ell)$, we have
\[
\one_{\text{cone}(w,l)} \longmapsto (-1)^{|I|-\ell}(w,\ell).
\]
\end{proposition}

\begin{proof}
Consider a relatively bounded face $F$ of $P$ and let $\operatorname{cone}_F(P) = \operatorname{cone}(w,q) = w^q + \operatorname{cone}(q)$. 
Each prelinear extension $\ell$ of $q$ corresponds to an open face $\sigma_\ell^o$ of the braid fan contained in the open dual preposet cone $\sigma_q^o$, that is, an ordered set partition $\ell$ such that $P_\ell = F$. The corresponding prelinear extension $(v,\ell)$ of $(w,q)$ is obtained by grouping the weights $w$ of $q$ among the parts of $\ell$; since 
$\operatorname{cone}_F(P) = \operatorname{cone}(w,q)$ these are the weights described in the statement of the proposition.
\end{proof}

\begin{remark}
Eur, Sanchez, and Supina \cite{ESS} constructed the universal valuation $\mathcal{F}$ for extended generalized $\Phi$-permutahedra for any finite reflection group.  
When $\Phi$ is the symmetric group $S_n$, their map $\mathcal{F}$, once interpreted combinatorially, is identical to the map $\varphi$ of Proposition \ref{prop:phi}. Their work thus explains why our map $\varphi$ is a valuation; this is equivalent to the inclusion $\supseteq$ in the proof of Theorem \ref{thm:quotient}, which we reprove for completeness. Our work reveals that their map $\mathcal{F}$ is not just a linear map, but also a Hopf morphism. 
\end{remark}

\begin{remark}
The isomorphism $\varphi$ of Proposition \ref{prop:phi} has an unexpected sign twist, which we illustrate in the smallest interesting example. In $\mbf{\Sigma^*}$, the product of the trivial ordered set partitions $a \in \mbf{w\Sigma^*}[\{a\}]$ and $b \in \mbf{w\Sigma^*}[\{b\}]$ is
\[
a \cdot b = a|b + b|a + ab \, \in \mbf{w\Sigma^*}[\{a,b\}],
\]
whereas the corresponding plates, $\textrm{cone}(a) = \{0\} \subset \R^{\{a\}}$ and 
$\textrm{cone}(b) = \{0\} \subset \R^{\{b\}}$, when considered in $\mathbb{I}(\mbf{GP}^+)$, satisfy
\begin{eqnarray*}
\one_{\textrm{cone}(a)} \cdot \one_{\textrm{cone}(b)} 
&=& \one_{(0,0)}  \subset \R^{\{a,b\}} \\
&=& \one_{\R_{\geq 0}(e_a-e_b)} + \one_{\R_{\leq 0}(e_a-e_b)} - \one_{{\R}(e_a-e_b)} \\
&=& \one_{\textrm{cone}(b|a)} + \one_{\textrm{cone}(a|b)} - \one_{\textrm{cone}(ab)},
\end{eqnarray*}
which matches the expression for $a \cdot b$ after a sign correction.
\end{remark}

\subsection{The extended McMullen species and the indicator Hopf monoid of posets}\label{sec:indicator2}

Although it is less relevant to our goal of studying valuations on generalized permutahedra, the following version of Theorem \ref{thm:quotient} may be of independent interest. 
Consider the following extension of the inclusion-exclusion species that also identifies a polyhedron $P \subseteq \R^I$ with any translate $v+P$ of it, where $v \in \R^I$, as is done in the McMullen polytope algebra.

\begin{definition}
The \textbf{extended McMullen species} consists of the vector subspaces
\[
\mbf{Mc}^+[I] 
=
\mbf{ie}[I] + 
\espan\left(P - (P+v) \, | \, P \in \mbf{GP}^+[I], v \in \R^I \right) \subset \mbf{GP}^+[I]
\]
for each finite set $I$ and the natural maps between them. 
\end{definition}


Let the \textbf{indicator Hopf monoid of (pre)posets} $\mathbb{I}(\mbf{(P)P})$ be the submonoid of $\mathbb{I}(\mbf{GP}^+)$ generated by (pre)poset cones.


\begin{theorem}\label{thm:McMullen} 
Let $\mbf{GP}^+$ be the Hopf monoid of extended generalized permutahedra.
\begin{enumerate}
\item
The extended McMullen species $\mbf{Mc}^+$ is a Hopf ideal of $\mbf{GP}^+$.
\item
The quotient Hopf monoid 
is isomorphic to the Hopf monoid of ordered set partitions: 
\[
\mbf{GP}^+/\mbf{Mc}^+ \cong \mbf{\Sigma^*}.
\]
\item
For any Hopf submonoid $\mbf{H} \subseteq \mbf{GP}^+$ the 
subspecies of $\mbf{GP}^+/\mbf{Mc}^+$ generated by the images of the indicator functions of polyhedra in $\mbf{H}$ 
is a Hopf quotient of $\mbf{H}$,
namely, $\mbf{H}/(\mbf{Mc}^+ \cap \mbf{H})$.
\item
The quotient Hopf monoid $\mbf{GP}^+/\mbf{Mc}^+$
is isomorphic to the indicator Hopf monoid of preposet cones and 
to the indicator Hopf monoid of poset cones:
\[
\mbf{GP}^+/\mbf{Mc}^+ \cong \mathbb{I}(\mbf{PP}) \cong \mathbb{I}(\mbf{P}).
\]
\end{enumerate}
\end{theorem}

\begin{proof}
We can compose the morphisms $\varphi: \mbf{GP}^+ \rightarrow \mbf{w\Sigma^*}$ with the projection $\mbf{w\Sigma^*} \rightarrow \mbf{\Sigma^*}$ that drops the weights, or equivalently, translates the plates to the origin. 
The resulting morphism $\psi: \mbf{GP}^+ \rightarrow \mbf{\Sigma^*}$ is surjective, and we claim that $\ker \psi = \mbf{Mc}^+[I]$, following Theorem \ref{thm:quotient}.

\medskip

\noindent 
$\supseteq$: We already saw that $\mbf{ie} \subseteq \ker \varphi \subseteq \ker \psi$, and $\psi(P) = \psi(v+P)$ since $P$ and $v+P$ have the same Brianchon-Gram decomposition up to translation.

\medskip

\noindent 
$\subseteq$: Take $a \in \ker \psi$. Analogously to Theorem \ref{thm:quotient}, we can write $a=b+c+d+p$ where $b, c,$ and $d$ are linear combinations of Brianchon-Gram generators, aligning generators, and translation generators ($P - (v+P)$) of $\mbf{Mc}^+[I]$, respectively, and $p$ is a linear combination of centered plates. Then  $p = a-b-c-d \in \ker \psi$ so $\psi(p) = 0$. But each centered plate $\operatorname{cone}(\ell)$ satisfies $\psi(\operatorname{cone}(\ell)) = (-1)^{|\ell|} \ell$ and there are no linear relations among ordered set partitions in $\mathbf{\Sigma^*}$, so in fact we must have $p = 0$ and $a = b+c+d \in \mbf{Mc}^+[I]$ as desired.

Again, the first three claims follow by the first isomorphism theorem of Hopf monoids.

4. The Brianchon-Gram theorem guarantees that the quotient $\mbf{GP}^+/\mbf{Mc}^+$ is spanned by the images of the preposet cones. Applying  Theorem \ref{thm:McMullen}.3 to the Hopf submonoid $\mbf{PP} \cong \mbf{CGP}^+_0 \subset \mbf{GP}^+$
of preposet cones, we get:
\[
\mbf{GP}^+/\mbf{Mc}^+ \cong \mbf{PP}/(\mbf{Mc}^+ \cap \mbf{PP}).
\]
But all preposet cones are centered at the origin, so $\mbf{Mc}^+ \cap \mbf{PP} = \mbf{ie} \cap \mbf{PP}$, and thus 
\[
\mbf{GP}^+/\mbf{Mc}^+ \cong \mbf{PP}/(\mbf{ie} \cap \mbf{PP}) = \mathbb{I}(\mbf{PP})
\]
by Theorem \ref{thm:quotient}.4, as desired. The isomorphism
$\mathbb{I}(\mbf{PP}) \cong \mathbb{I}(\mbf{P})$ is shown in Proposition \ref{prop:P=PP}.
 \end{proof}


With different goals in mind, Bastidas \cite{Bastidas20} proved a result analogous to part 1 of Theorem \ref{thm:McMullen} for the quotient $\mbf{GP}/\mbf{Mc}$ where only bounded polytopes are allowed. The difference between these two quotients may seem small at first sight, but their behavior is very different. For instance, every \textbf{bounded} generalized permutahedron $P$ in $\R^I$  maps to the same element of $\mbf{GP}^+/\mbf{Mc}^+ \cong \mbf{\Sigma}^*$, namely to
$(-1)^{|I|}\sum_{\ell \in \mbf{\Sigma^*}[I]} \ell$, thanks to the following result.

\begin{proposition} \label{prop:psi}
The image of an extended generalized permutahedron $P \in \mbf{GP}^+[I]$ in the quotient $\mbf{GP}^+/\mbf{Mc}^+ \cong \mbf{\Sigma^*}$ under the isomorphism of Theorem \ref{thm:McMullen}.4 is
\begin{eqnarray*}
\mbf{GP}^+ & \xrightarrow{\,\, \cong \,\, } & \mbf{GP}^+/\mbf{Mc}^+ \cong  \mbf{\Sigma^*} \\
P  &\longmapsto&   (-1)^{|I| - \dim \mathrm{Lin}(P)}\sum_
{\substack{\ell \, : \, P_\ell \text{ is } \\ \text{ rel. bounded}}} \ell.
\end{eqnarray*}
\end{proposition}

\begin{proof}
This follows readily from Proposition \ref{prop:phi}.
\end{proof}

\section{Cofreeness and universality} \label{sec:universal}

Aguiar and Ardila showed that many Hopf monoids on combinatorial objects embed into the Hopf monoid of extended generalized permutahedra \cite{AA17}. These include Hopf monoids of matroids, graphs, posets, multigraphs, simplicial complexes, and building sets, among others. 
Their work suggests that generalized permutahedra may satisfy some universality property in the category of Hopf monoids. We prove a concrete result in this direction by describing a universal property that characterizes the quotients $\mathbb I(\mbf{GP}^+) \cong \mbf{GP}^+/\mbf{ie}$ and 
$\mbf{GP}^+/\mbf{Mc}^+$.

This section assumes familiarity with the notion of cofree Hopf monoids as developed by Aguiar and Mahajan \cite{AM10}. We note that they proved the analogous result to Theorem \ref{thm: universality of Sigma} for $\mbf{\Sigma}^*$, which is isomorphic to $\mbf{GP}^+/\mbf{Mc}^+$.
In the Appendix, and Section \ref{sec:cofree} in particular, we summarize the relevant definitions and constructions.

Let $\mbf{E}^+$ and $\mbf{E}_{\mathbb{F}[t]}^+$ be the species with vector spaces
\[
\mbf{E}^+[I] = 
\begin{cases}
\langle 0 \rangle & \text{ if } I = \emptyset \\
\mathbb{F} & \text{ if } I \neq \emptyset,
\end{cases}
\qquad \qquad 
\mbf{E}_{\mathbb{F}[t]}^+[I] = 
\begin{cases}
\langle 0 \rangle & \text{ if } I = \emptyset \\
\mathbb{F}[t] & \text{ if } I \neq \emptyset.
\end{cases}
\]
and the natural maps between them, where $\mathbb{F}[t]$ is the polynomial ring with coefficients in $\mathbb{F}$.
Define $\mbf{E}_{\mathbb{F}\{t\}}^+$ similarly, where $\mathbb{F}\{t\}$ is the ring of \textbf{generalized polynomials} $\sum_{i=1}^n a_i t^{r_i}$ where $a_i \in \mathbb{F}$ and $r_i \in \R$. These species have the structure of positive monoids, with product given by the multiplication in the field or (generalized) polynomial ring.

\begin{theorem} \label{thm:Mcofree} The Hopf monoids $\mbf{\Sigma^*}$ and
$\mbf{GP}^+/\mbf{Mc}^+$ are cofree. They are isomorphic to the cofree Hopf monoid on $\mbf{E}^+$.
\end{theorem}

\begin{proof}
The first two Hopf monoids are isomorphic by Theorem \ref{thm:McMullen}.3. The second isomorphism was shown by Aguiar and Mahajan \cite[Proposition 12.58]{AM10} and it follows readily from the definitions.
\end{proof}

\begin{theorem} \label{thm:cofree} The Hopf monoids $\mbf{w\Sigma^*}$ and
$\mathbb I(\mbf{GP}^+) \cong \mbf{GP}^+/\mbf{ie}$ are cofree. They are isomorphic to the cofree Hopf monoid on $\mbf{E}_{\mathbb{F}\{t\}}^+$.
\end{theorem}

\begin{proof} 
We use the construction of cofree Hopf monoids explained in the Appendix.
Consider the species morphism 
\begin{eqnarray*}
\mbf{w\Sigma^*} & \longrightarrow & \mathcal{T}(\mbf{E}_{\mathbb{F}\{t\}}^+) \\
((w_1, \ldots, w_k), \ell_1 | \cdots | \ell_k) &\longmapsto & 
 (\ell_1 | \cdots | \ell_k, t^{w_1}\otimes \cdots \otimes t^{w_k}).
 \end{eqnarray*}
The pairs $(\ell_1 | \cdots | \ell_k, t^{w(\ell_1)} \otimes \cdots \otimes t^{w(\ell_k)})$ form a basis for the cofree Hopf monoid $\mathcal{T}(\mbf{E}_{\mathbb{F}\{t\}}^+)$, so this is a species isomorphism. Comparing the Hopf structures of $\mbf{w\Sigma^*}$ and $\mathcal{T}(\mbf{E}_{\mathbb{F}\{t\}}^+)$, described in Definition \ref{def:wSigma*} and Section \ref{sec:cofree}, immediately reveals that this is actually a Hopf monoid isomorphism.
\end{proof}

\begin{definition} \label{def:character}
A \textbf{character} $\zeta$ on a connected Hopf monoid $\mbf{H}$ consists   of maps $\zeta_I: \mbf{H}[I] \to \mathbb{F}$ that are natural, multiplicative, and unital in the sense of Definition \ref{def:character}. Similarly, a \textbf{(generalized) polynomial character} on $\mbf{H}$ consists of maps to the ring of (generalized) polynomials with the same properties.
\end{definition}

We define the \textbf{canonical character} $\beta$
on the Hopf monoid of ordered set partitions $\mbf{\Sigma}^*$ by
 \[\beta_I(\ell) = \begin{cases}
 1 & \text{if  $\ell$ has length one,} \\
 0 & \text{otherwise}.
 \end{cases} \]
Equivalently, we define the \textbf{canonical character} $\beta$ on
$\mbf{GP}^+/\mbf{Mc}^+$ by 
by
 \[
 \beta([\one_P]) = \begin{cases}
 (-1)^{|I|-\dim \textrm{Lin}(P)} 
 & \text{if $P$ is relatively bounded} \\ 
 0 & \text{if $P$ is relatively unbounded}.
 \end{cases} 
 \]
 for each extended generalized permutahedron $P$,  where $\textrm{Lin}(P)$ is the lineality space of $P$. 
 This is well-defined and matches the canonical character of $\mbf{\Sigma^*}$ by Proposition \ref{prop:psi}.

 \begin{theorem}\label{thm: universality of Sigma}
The terminal object in the category of Hopf monoids with characters is $(\mbf{GP}^+/\mbf{Mc}^+, \beta)$. 

Explicitly:
For any connected Hopf monoid $\mbf{H}$ and any character $\zeta$ on $\mbf{H}$,  there exists a unique Hopf morphism $\hat{\zeta}: \mbf{H} \to \mbf{GP}^+/\mbf{Mc}^+$ such that $\beta \circ \hat{\zeta} = \zeta$.
\end{theorem}

\begin{proof} 
A character is equivalent to a multiplicative map from $\mbf{H}_+$ to $\mbf{E}^+$. The result follows from Theorem \ref{thm: cofree universality}.
\end{proof}

Similarly, we define the \textbf{canonical} (generalized polynomial) \textbf{character} $\beta$ on the Hopf monoid of weighted ordered set partitions $\mbf{w\Sigma}^*$ by
 \[\beta_I(w, \ell) = \begin{cases}
 t^{w_1} & \text{if  $\ell$ has length one,} \\
 0 & \text{otherwise}.
 \end{cases} \]
Equivalently, we define the \textbf{canonical} (generalized polynomial) \textbf{character} $\beta$ on the indicator Hopf monoid of extended generalized permutahedra $\mathbb{I}(\mbf{GP}^+)$ by
 \[
 \beta(\one_P) = \begin{cases}
 (-1)^{|I|-\dim \textrm{Lin}(P)} t^{p} 
 & \text{if $P$ is relatively bounded and lies on hyperplane $\displaystyle \sum_{i \in I} x_i = p$ in $\R^I$,} \\ 
 0 & \text{if $P$ is relatively unbounded}.
 \end{cases} 
 \]
 where $\textrm{Lin}(P)$ is the lineality space of $P$. 
 This is well-defined and matches the canonical character of $\mbf{w\Sigma^*}$ by Proposition \ref{prop:phi}. We obtain the following analog to Theorem \ref{thm: universality of Sigma}.

%
%

\begin{theorem}\label{thm: universality of wSigma}
The terminal object in the category of Hopf monoids with generalized polynomial characters is $(\mathbb{I}(\mbf{GP}^+), \beta)$, the indicator Hopf monoid of extended generalized permutahedra with the canonical character.
%
%
\end{theorem}

\begin{proof} A generalized polynomial character is equivalent to a multiplicative map from $\mbf{H}_+$ to $\mbf{E}_{\mathbb{F}\{t\}}^+$. The result follows from Theorem \ref{thm: cofree universality}.
\end{proof}

Similarly, the terminal object in the category of Hopf monoids with polynomial characters is $(\mathbb{I}(\mbf{GP_\mathbb{N}}^+), \beta)$, where  $\mbf{GP}_\mathbb{N}^+ \subset \mbf{GP}^+$ is the Hopf monoid of \textbf{natural} extended generalized permutahedra for which the affine hulls of their faces are non-negative integral translates of root subspaces. Equivalently, these are the generalized permutahedra whose submodular function take an non-negative integral values.

The universality Theorems \ref{thm: universality of Sigma} and \ref{thm: universality of wSigma} explain why so many Hopf monoids are closely related to the Hopf monoid of generalized permutahedra, in ways that are compatible with functions that turn out to have valuative properties when they are viewed polyhedrally.

\begin{example} One consequence of Theorem \ref{thm: universality of wSigma} is that there is a natural bijection between generalized polynomial characters of $\mbf{H}$ and Hopf morphisms of the form $\phi: \mbf{H} \to \mathbb I(\mbf{GP}^+)$ given by the map that sends $\phi$ to the polynomial character $\beta \circ \phi$. 

As an example of this bijection, consider the Hopf submonoid $\mbf{GP} \subset\mbf{GP}^+$ consisting of bounded generalized permutahedra, and the map $\one_{-}: \mbf{GP} \to \mathbb I(\mbf{GP}^+)$ that sends a polytope to its indicator function. The corresponding generalized polynomial character is given by $\zeta = \beta \circ \one_{-}$. For any polytope $P \in \mbf{GP}$, the value of $\beta(\one_{P})$ is $(-1)^{|I|}t^p$ where $p$ is the real number such that $P$ lies on the hyperplane $\sum_{i} x_i = p$. 

This shows that the indicator function $P \mapsto \one_P$ corresponds to  the character $\zeta(P) = (-1)^{|I|}t^{\mu(P)}$ where $\mu: 2^I \rightarrow \mathbb{R}$ is the submodular function defining $P$ and $\mu(P) = \mu(I)$.
\end{example}

\section{Valuations from Hopf theory}\label{sec:valuations}

We now apply the results of the previous section to construct new valuations. First, we will need the following general fact about coideals in comonoids.

\begin{prop}\label{prop: Coideal factoring} Let $\mbf{C}$ be a comonoid, $\mbf{g}$ be a coideal, and $S_1 \sqcup \cdots \sqcup S_k = I$ be a set decomposition. If $f_1,\ldots, f_k$ are functions $f_i: \mbf{C}[S_i] \to R$, 
for some ring $R$ with multiplication $m$, define the function $f_1\cdots f_k: \mbf{C}[I] \to R$ by 
$f_1\cdots f_k = m \circ f_1 \otimes f_2 \otimes \cdots \otimes f_k \circ \Delta_{S_1,\cdots, S_k}$. Then
\[ 
\text{ If } f_i(\mbf{g}[S_i]) = 0 \text{ for $i = 1, \ldots, k$, then } 
f_1\cdots f_k(\mbf{g}[I]) = 0.
\]
\end{prop}
\begin{proof}
By the definition of coideal, we have
\begin{align*}
    \Delta_{S_1,\ldots,S_k}(\tb{g}[I]) &\subseteq \tb{g}[S_1] \otimes \tb{H}[S_2] \otimes \cdots \otimes \tb{H}[S_k] \\
 &+ \tb{H}[S_1] \otimes \tb{g}[S_2]  \otimes \cdots \otimes \tb{H}[S_k] \\
 & + \cdots \\ 
 & + \tb{H}[S_1] \otimes \tb{H}[S_2] \otimes \cdots \otimes \tb{g}[S_k] \\
\end{align*}
Since $f_i(\mbf{g}[S_i]) =  0 $ for each $i$, we have that $f_1 \otimes \cdots \otimes f_k \circ \Delta_{S_1,\ldots, S_k}(\mbf{g}[I]) =  0.$
\end{proof}

As a corollary, we obtain a proof of one of our main theorems, Theorem \ref{mainthm:vals}, which states that for a Hopf submonoid $\tb{H}$ of $\tb{GP}^+$,

\begin{center}
if $f_i: \tb{H}[S_i] \to R$ are strong valuations for $1 \leq i \leq k$, \\ then 
$f_1 \cdot \cdots \cdot f_k: \tb{H}[S_1 \sqcup \cdots \sqcup S_k] \to R$ is a strong valuation.
\end{center}

%

\begin{proof}[Proof of Theorem \ref{mainthm:vals}]
Let $\one_{-}: \mbf{H}\to \val({\mbf{H}})$ be the map that sends a polytope to its indicator function. Then by Theorem \ref{mainthm1}, we have that $\operatorname{Ker}(\one_{-})$ is a coideal of $\mbf{H}$ and $\val(\mbf{H}) \cong \mbf{H}/\operatorname{Ker}(\one_{-}).$ Since a function $f: \mbf{H}[I] \to A$ is a strong valuation if and only if $f(\operatorname{Ker}(\one_{-})) = 0$, the result follows from Proposition \ref{prop: Coideal factoring}.
\end{proof}

We now turn to two general applications of Theorem \ref{mainthm:vals}. The first is to the convolutions of linear species morphisms.

In what follows, if we have a collection of linear maps $g[I]: \mbf{F}[I] \to V$ to the same vector space $V$, we will often identify this with the species map $g$ from $\mbf{F}$ to the species $\mbf{V}[I] = V$ with trivial maps $\mbf{V}[f] = id$ for all $f: I \to J$.

\begin{definition} \label{def:convolution}
Let $f_1, \ldots, f_k$ be species maps $f_i: \mbf{GP}^+ \to A$ for some algebra $A$ with multiplication $m$. Their \textbf{convolution} $f_1 \star \cdots \star f_k:  \mbf{GP}^+ \to A$ is the species map given by
 \[ f_1 \star \cdots \star f_k[I](x) = \sum_{S_1 \sqcup \cdots \sqcup S_k = I} m \circ f_1[S_1] \otimes \cdots \otimes f_k[S_k] \circ \Delta_{S_1,\ldots,S_k}(x). \]
\end{definition}

Applying Theorem \ref{mainthm:vals}, we obtain the following corollary.

\begin{corollary}\label{cor: convolution valuations} Let $\mbf{H}$ be a submonoid of $\mbf{GP}^+$. Let $f_1,\ldots,f_k$ be species maps from $\mbf{H}$ to an algebra $A$. If each $f_i[I]$ is a strong valuation, then $f_1 \star \cdots \star f_k$ is a strong valuation.
\end{corollary}

Another application regards the character theory of Hopf monoids. This theory shows that multiplicative functions on a Hopf monoid give rise to polynomial invariants, quasisymmetric functions, and ordered set partitions associated to the elements in the Hopf monoid \cite{AA17}\cite{ABS06}\cite{AM10}. This construction has been used often to describe complicated combinatorial invariants in terms of simpler functions as well as to study combinatorial reciprocity; see \cite{AA17} or \cite{RG14} for examples.

\begin{definition} \label{def:character} Let $\mbf{H}$ be a connected Hopf monoid and $\mathbb{F}$ be a field. A \textbf{character} of $\mbf{H}$ with values in a field $\mathbb{F}$ is a collection of maps $\zeta_I:\mbf{H}[I] \to \mathbb{F}$ for each finite set $I$ satisfying the following properties.
\hspace{-1cm}
\begin{enumerate}
    \item (Naturality) For any bijection $\sigma: I \to J$, we have $\zeta_I(x) = \zeta_J(\mbf{H}[\sigma]\cdot x)$.
    \item (Multiplicativity) For any decomposition $S_1 \sqcup \cdots \sqcup S_k=I$, \, $\zeta_{S_1}(x_1) \cdots \zeta_{S_k}(x_k) = \zeta_I(x_1 \cdots x_k)$.
    \item (Unitality) The map $\zeta_{\emptyset}$ maps the unit of $\mbf{H}[\emptyset]$ to the unit of $\mathbb{F}$.
\end{enumerate}
\end{definition}

\begin{definition} \label{def:charinvariants}
Let $\zeta$ be a character of a connected Hopf monoid $\mbf{H}$. 
\begin{enumerate}
    \item The \textbf{polynomial invariant associated to} $\zeta$ is the function that maps $h \in \mbf{H}[I]$ to the unique polynomial $f_{\zeta}(h)(t)$ such that for any positive integer $k$
     \[  f_{\zeta}(h)(k) = \zeta^{*k}(h).\]
    \item The \textbf{quasisymmetric function associated to} $\zeta$ is the function that maps $h$ to the quasisymmetric function $\Phi_{\zeta}(h)$ given by
     \[ \Phi_{\zeta}(h) = \sum_{S_1 \sqcup \cdots \sqcup S_k = I} m \circ \zeta_{S_1} \otimes \cdots \otimes \zeta_{S_k} \circ \Delta_{S_1,\ldots, S_k}(h) M_{|S_1|, |S_2|, \ldots, |S_k|}.\]
     \item The \textbf{ordered set partition invariant associated to} $\zeta$ is the function that maps $h$ to the linear combination of ordered set partitions $O_{\zeta}(h)$ given by
     \[ O_{\zeta}(h) = \sum_{S_1 \sqcup \cdots \sqcup S_k = I} m \circ \zeta_{S_1} \otimes \cdots \otimes \zeta_{S_k} \circ \Delta_{S_1,\ldots, S_k}(h) [S_1 \sqcup S_2 \sqcup \cdots \sqcup S_k].\]
\end{enumerate}
\end{definition}

Combining Theorem \ref{mainthm:vals} with these invariants associated to a character, we obtain the following.

\begin{corollary}\label{cor: character theory valuation} Let $\mbf{H}$ be a submonoid of $\mbf{GP}^+$. Let $\zeta$ be a character of $\mbf{H}$ such that $\zeta[I]$ is a strong valuation. Then the three maps 
\[
    h \mapsto f_{\zeta}(h), \qquad
    h \mapsto \Phi_{\zeta}(h), \qquad
    h \mapsto O_{\zeta}(h) \qquad  \text{ for } h \in \mbf{H}[I]
\]
are strong valuations.
\end{corollary}

We conclude this section by showing that valuative characters on generalized permutahedra form a group.
Aguiar, Bergeron, and Sottile showed that the characters of a combinatorial Hopf algebra form a group.
Aguiar and Ardila \cite{AA17} extended character theory to Hopf monoids, giving several combinatorial consequences. The key structural result is the following.

\begin{proposition}\cite{ABS06, AA17} The set of characters $\mathbb X(\mbf{H})$ of a (connected) Hopf monoid form a group under convolution. The identity is the character $e$ where $e[I]=0$ for $I \not= \emptyset$ and $e[\emptyset]$ is the isomorphism $\mbf{H}[\emptyset] \cong \F$. 
The inverse of a character $\zeta$ is $\zeta^{-1} = \zeta \circ S$, 
%
where $S$ is the antipode of $\mbf{H}$.
\end{proposition}

\noindent 
We have done all the work needed to show that character theory interacts nicely with valuations.

\begin{proposition}\label{prop:valchar} Let $\mbf{H}$ be a Hopf submonoid of $\mbf{GP}^+$. The  characters of $\mbf{H}$ that are strong valuations $\mathbb X(\mbf{H})^{val}$ form a subgroup of the character group $\mathbb X(\mbf{H})$.
\end{proposition}

\begin{proof} The identity character is trivially a valuation. Corollary \ref{cor: convolution valuations} implies that valuative characters are closed under convolution. A character being a valuation is the same as $\zeta(\mbf{ie}[I]) = \langle 0 \rangle$. By Theorem \ref{mainthm1}, we know that $\mbf{ie}$ is an ideal and a coideal. This implies that
 \[ S(\mbf{ie}[I]) \subseteq \mbf{ie}[I].\]
Therefore, for any valuative character $\zeta$ we have
$\zeta^{-1}(\mbf{ie}[I]) = \zeta \circ S( \mbf{ie}[I]) \subseteq \zeta(\mbf{ie}[I]) = \langle 0 \rangle$.
This shows that $\mathbb X(\mbf{H})^{val}$ is closed under taking inverses.
\end{proof}


\section{Valuations on generalized permutahedra}\label{sec:valGP}

We now use this Hopf theoretic framework to give simple, unified proofs for some new and some known valuations on generalized permutahedra. We recall that for the class of generalized permutahedra, weak valuations and strong valuations coincide \cite{DF10}.

\subsection{Chow classes in the permutahedral variety}

The braid fan $\Sigma_I$ has an associated toric variety, called the \textbf{permutahedral variety} $X_I$. The Chow ring of $X_I$ was described by McMullen \cite{McM93} and Fulton and Sturmfels \cite{FS94}\footnote{using slightly different conventions} in terms of \textbf{Minkowski weights}: these are the functions assigning a weight $w_\sigma$ to each $k$-dimensional face of $\Sigma_I$, subject to a certain balancing condition. In the braid fan we have a face $\sigma_{S_1|\cdots|S_k}$ for each ordered set partition $S_1|\cdots|S_k$ of $I$.

Fulton and Sturmfels constructed a linear\footnote{The vector space of indicator functions of generalized permutahedra forms an algebra with product given by Minkowski sums of polytopes. With this structure, their map becomes an algebra isomorphism.} isomorphism $\theta$ between the space of indicator functions of rational generalized permutahedra in $\R^I$ up to translation and the Chow ring $A^\cdot(X_{\Sigma_I}) \otimes \Q$ of the permutahedral variety $X_{\Sigma_I}$ 
\cite{FS94}. If $D$ is a line bundle of $X_I$ where $\mathcal{O}(D)$ is generated by its sections with corresponding generalized permutahedron $P_D$, then $\theta(P_D)$ is the exponential $\operatorname{exp}(D)$. Explicitly, the element (viewed as a Minkowski weight) $\theta(P)$ is given by
\[
\theta(P)_{S_1 | \cdots | S_k} = v_{S_1 | \cdots |S_k} \mathrm{NVol}(P_{S_1|\cdots|S_k}),
\]
where $P_{S_1|\cdots|S_k}$ is the face of $P$ maximized by any direction $w \in \sigma_{S_1|\cdots|S_k}$, 
$\mathrm{NVol}(P_{S_1|\cdots|S_k})$ is its $(|I|-k)$-dimensional volume in the suitable translate of the subspace given by $\sum_{s \in S_i} x_s= 0$ for all $i$, and $v_{S_1 | \cdots |S_k}$ is an explicit constant not depending on $P$.

Their proof uses the fact, due to McMullen \cite{McM93}, that $\theta$ is valuative. We now give a simple Hopf theoretic proof of this fact.

\begin{proposition}
The map $P \mapsto \theta(P)$ is a valuation.
\end{proposition}

\begin{proof}

Aguiar and Ardila \cite{AA17} showed that the iterated coproduct for the Hopf monoid $\mathbf{GP}$ is
\[
\Delta_{S_1,\ldots, S_k}(P) =P[F_0, F_1] \otimes \cdots \otimes P[F_{k-1},F_k]
\]
where $F_i = S_1 \sqcup \cdots \sqcup S_i$, and $P[F_0, F_1], \ldots, P[F_{k-1}, F_k]$ are the polytopes in $\R^{S_1}, \ldots, \R^{S_k}$ whose product is $P_{S_1|\cdots|S_k}$. Therefore
\begin{align*}
\theta(P)_{S_1 | \cdots | S_k} &= v_{S_1 | \cdots |S_k} \mathrm{NVol}(P_{S_1|\cdots|S_k}) \\
&= v_{S_1 | \cdots |S_k} m_{\R} \circ \mathrm{NVol}[S_1] \otimes \cdots \otimes \mathrm{NVol}[S_k] \circ \Delta_{S_1,\ldots,S_k}(P), 
\end{align*}
which is valuative by Theorem \ref{mainthm:vals}, since normalized volume is valuative and constant scalings of valuative functions are valuative.
\end{proof}


\subsection{The universal Tutte character of generalized permutahedra}

%

One of the most important invariants of a matroid is the \textbf{Tutte polynomial} defined by Tutte \cite{Tutte67} and Crapo \cite{Crapo69}. It is the universal polynomial satisfying a deletion-contraction recurrence. We will define and study it and many related invariants in Section \ref{sec:valM}. 

In this section we focus on a generalization of the Tutte polynomial for generalized permutahedra, due to Dupont, Fink, and Moci \cite{DFM17} 
\footnote{Their construction applies to a class of objects called \emph{minor systems}, which includes comonoids; we have adapted their definitions to $\mbf{GP}$.} We will restrict our attention to the species $\mbf{GP}_{\N}$ of generalized permutahedra whose submodular functions take values in $\N$. This can be adapted to all generalized permutahedra by using generalized polynomials, as was done in the universality results of Section \ref{sec:universal}.

\begin{definition} \cite{DFM17}   
\begin{enumerate}
\item
 A \textbf{Tutte-Grothendieck invariant}\footnote{This definition is equivalent to the universality property described in Proposition 3.20 of \cite{DFM17}. To see the equivalence, note that every norm to a ring $R$ factors uniquely through the universal norm of $\mbf{GP}_\N$. Thus, a norm is equivalent to a map from the Grothendieck monoid $U(\mbf{GP}_\N)$ of $\mbf{GP}_\N$ into the ring $R$. By Proposition 8.2 of \cite{DFM17}, the monoid $U(\mbf{GP}_\N)$ embeds into $\F[x,y,y^{-1}]$ and so every norm is induced by a ring morphism $f: \F[x,y,y^{-1}] \to \R$.} on generalized permutahedra is a linear species morphism $\Phi: \mbf{GP}_\N \to R$ to a ring $R$ such that $\Phi[\emptyset](1) = 1_R$ and  there exist two ring morphisms $f_1, f_2: \F[x,y,y^{-1}] \to R$ such that 
  \[\Phi(P) = f_1\left(xy^{\mu(P|_i)}\right) \cdot \Phi(P/_i) + f_2\left(xy^{\mu(P/_{(I-i)})}\right)\Phi(P|_{I-i}). \]
for all $i \in I$. 
\item
The \textbf{universal Tutte character} $\mathcal{T}[I]: \mbf{GP}_\N[I] \to \F[x_1,y_1,y_1^{-1},x_2,y_2,y_2^{-1}]$ is given by
\[
\mathcal{T}[I](P)  = x_2^{|I|}y_2^{\mu_P(I)} \sum_{A \subseteq I} \left ( \frac{x_1}{x_2} \right )^{|A|} \left ( \frac{y_1}{y_2} \right )^{\mu_P(A)}.
\]
for a generalized permutahedron $P$ in $\R^I$, where $\mu_P$ is the submodular function associated to $P$.
\end{enumerate}
\end{definition}

Any Tutte-Grothendieck invariant is an evaluation of the universal Tutte character:

\begin{theorem}\cite{DFM17} If $\Phi:\mbf{GP}_\N \to R$ is a Tutte-Grothendieck invariant of $\mbf{GP}_\N$, there is a map 
    \[\alpha: \F[x_1,y_1,y_1^{-1},x_2,y_2,y_2^{-1}] \to R\]
 such that $\Phi(P) = \alpha( \mathcal{T}[I](P))$ for all $P \in \mbf{GP}_\N[I]$.
\end{theorem}

Our main result of this section is that the universal Tutte polynomial is a valuation.

\begin{proposition}\label{thm: Tutte character is a valuation} The universal Tutte character $\mathcal{T}$ is a valuation on $\mbf{GP}_\N$. In particular, every Tutte-Grothendieck invariant of $\mbf{GP}_\N$ is a valuation.
\end{proposition}

\begin{proof} 
Let  $N_1[I], N_2[I]: \mbf{GP}_\N[I] \to \F[x_1,y_1,y_1^{-1},x_2,y_2,y_2^{-1}]$ be the characters given by 
\[
	N_1[I](P) = x_1^{|I|}y_1^{\mu_P(I)}, \qquad 
	N_2[I](P) = x_2^{|I|}y_2^{\mu_P(I)},
\]
for a generalized permutahedron $P$ in $\R^I$, where $\mu$ is the submodular function of $P$ and $\mu(P) = \mu(I)$. By construction, $\mathcal{T}$ is the convolution of the two morphisms $N_1 \star N_2$.
By Corollary \ref{cor: convolution valuations} it suffices to show that the maps $N_1$ and $N_2$ are valuations.

Any subdivision $\mathcal{P}$ of $P \subseteq \R^I$  is contained in the hyperplane where $\sum x_i = \mu(P)$, so every polytope $P_i \in \mathcal{P}$ must also be contained in that hyperplane. Thus $N_1$ and $N_2$ are constant on any subdivision, so they are weak valuations by Proposition \ref{prop: Euler characteristic}, and strong valuations  by Theorem \ref{thm: GP strong weak equiv}.
The result follows.
\end{proof}

\noindent
Using the tools of \cite{DFM17}, this theorem can be generalized to any linearized subcomonoid of $\mbf{GP}_\N$ and using generalized polynomial rings this can further be generalized to $\mbf{GP}$.

\subsection{The Tutte polynomial of a matroid and of a matroid morphism}

As an application, we now give a proof that the Tutte polynomial for matroids, matroid morphisms, and flag matroids are all valuations.
From the point of view of geometry, (representable) matroids are naturally connected to the Grassmannian. Extending this relationship to the various partial flag varieties gives rise to the notion of flag matroids. For a thorough discussion see \cite{BGW03}.

A \textbf{matroid morphism} is a pair of matroids $M \rightarrow M'$ on the same ground set such that every flat of $M'$ is a flat of $M$; these two matroids are called \textbf{concordant}. More generally, a \textbf{flag matroid} $\mathcal{M}$ consists of $k$ matroids $M_1, \ldots, M_k$ of different ranks such that every pair is concordant. The \textbf{flag matroid polytope} of $\mathcal{M}$ is the Minkowski sum of the corresponding matroid polytopes:
\[
P(\mathcal{M}) = P(M_1) + \cdots + P(M_k) = \{a_1 + \cdots + a_k \, : \, a_i \in P(M_i) \text{ for } i=1, \ldots, k\}.
\]
The flag matroid polytope $P(\mathcal{M})$ is itself a generalized permutahedron. \cite{BGW03}.

Las Vergnas \cite{Ver80} introduced the \textbf{Tutte polynomial of a matroid morphism} $M \rightarrow M'$:
\[
T_{M \rightarrow M'}(x,y,z) = \sum_{A \subseteq E} (x-1)^{r(M')-r_{M'}(A)} (y-1)^{|A| - r_M(A)} z^{(r(M)-r_M(A)) - (r(M') - r_{M'}(A))}
\]
and showed that it specializes to many quantities of interest. See also \cite{Ardila07semimatroids}. The Tutte polynomial of a matroid $M$ is obtained by setting $M=M'$.

\begin{proposition}\label{prop:Tutte}
The Tutte polynomial of a matroid, the Las Vergnas Tutte polynomial of a matroid morphism, and the universal Tutte character of a flag matroid are valuations.
\end{proposition} 

\begin{proof} 
This follows from 
Proposition \ref{thm: Tutte character is a valuation} and
the fact, shown in \cite{DFM17}, that the Tutte polynomial and the Las Vergnas Tutte polynomial are reparameterizations of the universal Tutte character for matroids and matroid morphisms.
\end{proof}

Using the relationship between flag matroids and the flag variety Dinu, Eur, and  Seynnaeve defined a \textbf{K-theoretic Tutte polynomial} for flag matroids \cite{DES20}. They showed that their polynomial is also valuative, but it is not a Tutte-Grothendieck invariant. It would be interesting to explain its relationship with the Hopf algebraic framework of this paper.

\section{Valuations on matroids}\label{sec:valM}

The subdivisions of a matroid polytope into smaller matroid polytopes arise naturally in various algebro-geometric contexts, for example, the compactification of the moduli space of hyperplane arrangements due to Hacking, Keel, and Tevelev \cite{HKT} and Kapranov \cite{Kap}, the compactification of fine Schubert cells in the Grassmannian due to Lafforgue \cite{Lafforgue1,Lafforgue2}, the K-theory of the Grassmannian \cite{SpeyerKtheory}, the stratification of the tropical Grassmannian \cite{SS04} and the study of tropical linear spaces by Ardila and Klivans \cite{ArdilaKlivans} and Speyer  \cite{Speyer}. 

The study of valuations on matroids was initiated by Speyer in \cite{Speyer, SpeyerKtheory} in order to understand the constraints on matroid subdivisions. He  discovered several valuations on matroids -- some coming from the K-theory of the Grassmannian -- and used them to prove bounds on the $f$-vector of a tropical linear space. With this paper as motivation, many authors have constructed other valuations of matroids. We now show that many of these valuations arise easily from our construction. We note that for the class of matroids, weak valuations and strong valuations coincide \cite{DF10}.

Two key matroid invariants are the \textbf{Tutte polynomial} and the \textbf{characteristic polynomial}:
%
\begin{eqnarray*}
T_M(x,y) &=& \sum_{A \subseteq E} (x-1)^{r(M)-r_{M}(A)} (y-1)^{|A| - r_M(A)}, \\ \chi_M(t) &=& \sum_{\substack{F \subseteq M \\ \text{$F$ flat}}} \mu(\emptyset, F) t^{rk(M) - rk(F)} = (-1)^r T_M(1-t,0)
\end{eqnarray*}
where $\mu$ is the M\"obius function of the lattice of flats of $M$. We saw in Proposition \ref{prop:Tutte} that $T_M$ is a valuation. Since all matroids involved in a matroid subdivision lie on the same hyperplane $\sum_i x_i = r$, they must have the same rank, and hence the above expression shows that $\chi_M(t)$ is a valuation as well.

\subsection{The Chern-Schwartz-MacPherson cycles of a matroid}

The deep connection between matroids and tropical geometry, which stem from the fact that the Bergman fan of a matroid is a tropical fan \cite{ArdilaKlivans},  leads to many old and new invariants of matroids coming from geometry. A very interesting example is the Chern-Schwartz-MacPherson (CSM) cycle of a matroid defined by L\'opez de Medrano, Rinc\'on, and Shaw \cite{LRS}.\footnote{The CSM cycle of a matroid was originally defined as a tropical cycle; we describe it as a Minkowski weight.} 

The \textbf{beta invariant} $\beta(M)$ of a matroid is the coefficient of $x^1y^0$ in the Tutte polynomial of $M$. The beta invariant of a flag $\mathcal{F} = \{F_1 \subset \cdots \subset F_k\}$ is $\beta(M[\mathcal{F}]) =  \prod_{i=0}^{k} \beta(M[F_i, F_{i+1}])$ where  $M[A,B] = (M|_B)/_A$ for $A \subseteq B$; this is non-zero if and only if $\mathcal{F}$ is a flag of flats.

\begin{definition}
Let $M$ be matroid of rank $r$ on ground set $I$. 
For $0 \leq k \leq r-1$, the $k$-th \textbf{Chern-Schwartz-MacPherson cycle} $\CSM_k(M)$ is the $k$-dimensional Minkowski weight on the braid fan $\Sigma_I$ given by 
\[
\CSM_k(M)(S_1 | \cdots | S_k) = (-1)^{r-k}  \beta(M[\mathcal{F}]) 
\]
for each ordered set partition $S_1|\cdots|S_k$ of $I$, where $F_i = S_1 \sqcup \cdots \sqcup S_i$ for $1 \leq i \leq k$.
\end{definition}

To interpret this as a Minkowski weight on the braid fan $\Sigma_I$, we note that the faces of this fan are in natural bijection with the ordered set partitions of $I$.\footnote{The fact that the function above is indeed a Minkowski weight on this fan is a non-trivial result in \cite{LRS}.} 
We obtain a much simpler proof of a theorem of L\'opez de Medrano, Rinc\'on, and Shaw.

\begin{theorem}\cite{LRS} 
For any fixed $k$, the $k$th Chern-Schwartz-MacPherson class $\CSM_k(M)$ is a valuation of matroids.
\end{theorem}

\begin{proof} Since the Tutte polynomial is a valuation by Proposition \ref{prop:Tutte}, the $\beta$ invariant is also a valuation; this was first observed by Speyer \cite{Speyer}. 
Theorem \ref{mainthm:vals} then implies that  the function
 \[
 m \circ \beta^{\otimes k} \circ \Delta_{S_1,\ldots,S_k}(M) = \beta(M[\mathcal{F}])
 \]
is also a valuation for any set partition $S_1 \sqcup \cdots \sqcup S_k$.
Since a matroid polytope $P(M)$ lies on the hyperplane $\sum_i x_i = r(M)$, all the matroids in a matroid subdivision must have the same rank. It follows that $\CSM_k(M) = (-1)^{r(M)-k} m \circ \beta^{\otimes k} \circ \Delta_{S_1,\ldots,S_k}(M)$ is a valuation as well.
\end{proof}

\subsection{The volume polynomial of a matroid}

One of the most recent celebrated results in matroid theory is the construction of the combinatorial Chow ring of a matroid by Adiprasito, Huh, and Katz \cite{AHK15}. In the case when $M$ is a realizable matroid, this ring agrees with the Chow cohomology ring of the wonderful compactification of the hyperplane arrangement associated to $M$. 
For each loopless matroid, Eur constructed a multivariate polynomial which encodes all the information of its combinatorial Chow ring \cite{Eur20}. 


%

The \textbf{characteristic polynomial} of a loopless matroid $M$ is given by \[ \chi_M(t) = \sum_{\substack{F \subseteq M \\ \text{$F$ flat}}} \mu(\emptyset, F) t^{rk(M) - rk(F)} = (-1)^r T_M(1-t,0)\]
where $\mu$ is the M\"obius function of the lattice of flats of $M$ and $T_M$ is the Tutte polynomial.
If $M$ has a loop, we set $\chi_M(t) = 0$. This is a multiple of $t-1$, and the \textbf{reduced characteristic polynomial} is $\overline{\chi}_M(t) = \frac{\chi_M(t)}{t-1}.$
Let $\mu^i(M)$ denote the unsigned coefficient of $t^i$ in the reduced characteristic polynomial of $M$.

%
%


\begin{definition}\cite{Eur20} Let $I$ be a finite set and $\R[t_F]$ be the polynomial ring on variables $t_F$ for $F \subset I$. The \textbf{volume polynomial}\footnote{The motivation for this definition is algebro-geometric; this is a non-trivially equivalent formulation.} of a matroid $M$ is 
 \[ 
 VP_M(\mbf{t}) = \sum_{\substack{\emptyset = F_0 \subset F_1 \subset \cdots \subset F_k \subset F_{k+1} \\d_1 + \cdots + d_k = d }} (-1)^{d-k}\binom{d}{d_1, \cdots, d_k} \prod_{i} \binom{d_i - 1}{\hat{d_i} - r(F_i)} \mu^{\hat{d_i} - r(F_i)}(M[F_i,F_{i+1}]) t_{F_1} \cdots t_{F_k},
 \]
 where the sum over flags of flats of $M$ and over sets of positive integers $d_i$ with $\sum d_i = d$, and we denote $\hat{d_j} = \sum_{i=1}^j d_i$.
\end{definition}

%


\begin{theorem}\cite{Eur20} 
The volume polynomial $VP_M(\mbf{t})$ is a valuation of matroids.
\end{theorem}

\begin{proof}
Since the Tutte polynomial is a valuation,  the function $\mu^e(M)$ is also a valuation for fixed $e$.
Once again, the matroids involved in a matroid subdivision have a fixed rank, so the term corresponding to a fixed choice of $F_1, \ldots, F_k$ and $d_1, \ldots, d_k$ is a constant multiple of $\prod \mu^{e_i}(M[F_i, F_{i+1}]) = \mu^{e_1}\star \cdots \star \mu^{e_k}(M)$ for fixed $e_1, \ldots, e_k$; this is a valuation by Theorem \ref{mainthm:vals}.
\end{proof}

Eur's proof is similar in spirit, but he relies on the universality of the Derksen-Fink invariant. 


\subsection{The Kazhdan-Lusztig polynomial of a matroid}\label{KL}

\begin{definition}\label{def:KL polynomial}\cite{EPW} The \textbf{Kazhdan-Lusztig polynomial} of a matroid $M$ is the unique polynomial $P_M(t)$ satisfying the following conditions for all matroids:

\begin{enumerate}
	\item If $M$ is the trivial matroid of rank $0$ , then $P_M(t) = 1$.
	\item If $r(M) > 0$, then $\operatorname{deg}(P_M(t)) < \frac{1}{2} r(M)$.
	\item For every matroid $M$ on $I$,
	 \[ t^{r(M)}P_M(t^{-1}) = \sum_{F \subseteq I \textrm{ flat}}\chi_{M|_F}(t)P_{M/_F}(t). \] 
\end{enumerate}
\end{definition}

\begin{definition}\label{def:invKL polynomial}\cite{GX20} The \textbf{inverse Kazhdan-Lusztig polynomial} of a matroid $M$ is the unique polynomial $Q_M(t)$ satisfying the following conditions for all matroids:

\begin{enumerate}
	\item If $M$ is the trivial matroid of rank $0$, then $Q_M(t) = 1$.
	\item If $r(M) > 0$, then $\operatorname{deg}(Q_M(t)) < \frac{1}{2} r(M)$.
	\item For every matroid $M$ on $I$,
	 \[ (-t)^{r(M)}Q_M(t^{-1}) = \sum_{F \subseteq I \textrm{ flat}} (-1)^{r(M|_F)} Q_{M|_F}(t)t^{r(M/_F)} \chi_{M/_F}(t^{-1}). \] 
\end{enumerate}
\end{definition}

\begin{remark} In \cite{EPW} and \cite{GX20}, these polynomials are only defined for loopless matroids. The definitions we have given extend their definitions to the case of all matroids. Note that $P_M = 0$ whenever $M$ is non-trivial and has a loop.
\end{remark}

\begin{theorem} 
The inverse Kazhdan-Lusztig polynomial $Q_M(t)$ is a valuation of matroids.
\end{theorem}

\begin{proof} We proceed by induction on the size of the ground set $I$. When $|I| = 1$, there are two matroids on $I$. Their matroid polytopes are both points. Thus, every function is trivially a valuation on these matroids. Now suppose $Q_M(t)$ is a strong valuation for matroids on ground sets of size less than $k$,  and consider a ground set $I$ with size $|I| = k$.

Define
 \begin{eqnarray*}
 R_M(t) &=& \sum_{\substack{F \not = I \; \textrm{flat}}} (-1)^{r(M|_F)} Q_{M|_F}(t)t^{r(M/_F)} \chi_{M/_F}(t^{-1}). \\
\end{eqnarray*}
If $S \subsetneq I$ is not a flat of $M$, then the contraction $M/_S$ will have a loop, so $\chi_{M/_S}(t)=0$. Thus 

 \begin{eqnarray*}
R_M(t) &=& \sum_{\substack{S \sqcup T = I \\ T \not = \emptyset}} m \circ \left((-1)^{r(-)}Q_{-}(t) \otimes t^{r(-)}\chi_{-}(t^{-1}) \right) \circ \Delta_{S,T}(M).
 \end{eqnarray*}
 For each decomposition $S \sqcup T=I$ with $T \not = \emptyset$, $Q_{-}(t)$ is a valuation for matroids on $S$ by the induction hypothesis, and $\chi_{-}(t)=(-1)^{r(-)}T_{-}(1-t,0)$ is a valuation for matroids on $T$. Since the matroids in a matroidal subdivision have the same rank, the assignments 
 $M \mapsto (-1)^{r(M)}Q_M(t)$ and 
 $M \mapsto t^{r(M)}\chi_M(t^{-1})$ are also valuations on $S$ and $T$ respectively. Theorem \ref{mainthm:vals} then shows that $R_{-}(t)$ is a valuation. In view of Definition \ref{def:invKL polynomial}(3) of $Q_M(t)$, the function
\[ 
(-1)^{r(M)}R_M(t) = t^{r(M)} Q_M(t^{-1}) - Q_M(t)
\]
is also a valuation.

Let $\sum_{i} \pm \one_{M_i} = 0$ be a relation coming from a matroidal subdivision, where all the matroids have rank $r$. Then, we have
 \[\sum_{i}\pm \left (t^r Q_{M_i}(t^{-1}) - Q_{M_i}(t) \right) = 0; \]
that is,
 \[ \sum_{i}\pm t^r Q_{M_i}(t^{-1}) = \sum_{i}\pm Q_{M_i}(t).\]
But each term in the left hand side has degree greater than $r/2$ and each term in the right hand side has degree less than $r/2$ by Definition \ref{def:invKL polynomial}(2), so both sides must equal $0$. Thus 
 $\sum_{i} \pm Q_{M_i} = 0$ and 
 $M \mapsto Q_M$ is a weak valuation. For matroids, weak and strong valuations agree \cite{DF10}, so this is also a strong valuation.
\end{proof}

\begin{theorem} 
The Kazhdan-Lusztig polynomial $P_M(t)$ is a valuation of matroids.
\end{theorem}

\begin{proof} We proceed by induction on the size of the ground set $M$. As with the inverse Kazhdan-Lusztig polynomial, the base case trivially holds. Now suppose that $M \mapsto P_M(t)$ is valuative for all matroids on ground sets of size less than $k$, and consider a ground set $I$ with $|I| = k$.

If a matroid $M$ contains a loop $e$, then its matroid polytope lies on the hyperplane $x_e=0$, and all matroids in a matroid subdivision of $M$ will also contain that loop. Thus the Kazhdan-Lusztig polynomial is valuative on any such subdivision, because it equals 0 on all of the matroids involved. If $M$ is loopless, then Gao and Xie \cite{GX20} show that
 \[
 P_M(t) = -\sum_{\substack{F \neq \emptyset \\ \text{flat}}} (-1)^{\operatorname{r}(M|_F)}  Q_{M|_F}(t) \cdot P_{M/_F}(t).
 \]
Since $P_M$ vanishes whenever $M$ has loops and $M/_S$ will have a loop if $S$ is not a flat, we can rewrite this equation as
\begin{align*}
P_M(t) &= - \sum_{S \neq \emptyset} (-1)^{\operatorname{r}(M|_S)}  Q_{M|_S}(t) \cdot P_{M/_S}(t) \\
	&= - \sum_{S \neq \emptyset} m \circ  \left( (-1)^{\operatorname{r}(-)}Q_{-} \otimes P_{-}(t) \right)  \circ \Delta_{S,T}(M).
\end{align*}
By induction, the functions in the summand are valuative. Theorem \ref{mainthm:vals} then allows us to conclude that $M \mapsto P_M(t)$ is valuative.
%
\end{proof}

\begin{example}\label{ex:u36 subdivision} Consider the following matroid subdivision described in \cite{BJR09}. Let $U_{3,6}$ denote the uniform matroid on ground set $[6]$. Let $M_1$ be the Schubert matroid with maximal element $\{2,4,6\}$. The bases of this matroid are all subsets $1 \leq a < b < c \leq 6$ with $a \leq 2$, $b \leq 4$, $c \leq 6$. Let $\sigma$ be the permutation $345612$. Then, the matroids $M_1$, $\sigma \cdot M_1$, and $\sigma^2 \cdot M_1$, with $\sigma$ acting by relabelling the ground set and bases, are the maximal dimensional matroids of a subdivision of the uniform matroid $U_{3,6}$.

The other matroids in this subdivision are the matroid $M_2$ with bases
 \[\mathcal{B}(M_2) = \{134, 135, 136, 145, 146, 234, 235, 236, 245, 246 \}, \]
two isomorphic matroids given by $\sigma \cdot M_2$ and $\sigma^2 \cdot M_2$ and a final matroid $M_3$ with bases
 \[ \mathcal{B}(M_3) = \{135,136,145,146,235, 236, 245, 246\}.\]
This subdivision gives the inclusion-exclusion relation among indicator functions
 \[ \one_{U_{3,6}} = \one_{M_1} + \one_{\sigma \cdot M_1} + \one_{\sigma^2 \cdot M_1} - \one_{M_2} - \one_{\sigma \cdot M_2} - \one_{\sigma^2 \cdot M_2} + \one_{M_3}. \]
Using the methods of \cite{EPW}, we compute the Kazhdan-Lusztig polynomials
\[
	P_{U_{3,6}} = 9t + 1 
	\qquad
	P_{M_1} = P_{\sigma \cdot M_1} = P_{\sigma^2 \cdot M_1} = 3t + 1
	\qquad
	P_{M_2} = P_{\sigma \cdot M_2} = P_{\sigma^2 \cdot M_2} = 1 
	\qquad
	P_{M_3} = 1,
\]
which indeed satisfy the inclusion-exclusion relation
 $9t +1 = 3(3t + 1) - 3(1) + 1.$
\end{example}

\subsection{The motivic zeta function of a matroid}\label{sec:motivic}

In \cite{JKU19}, Jensen, Kutler, and Usatine constructed three motivic zeta functions for matroids which, in the case of realizable matroids, coincide with the Igusa zeta functions of hyperplane arrangements. The \textbf{local motivic zeta function} of a matroid $M$ on ground set $I$ is
\[
Z_M^0(q,t) = \sum_{w \in \Z_{>0}^E} \chi_{M_w}(q) q^{-r(M)-\operatorname{wt}_M(w)}t^{|W|},
\]
where $M_w$ is the matroid of $w$-maximal bases of $M$, and $\operatorname{wt}_M(w)$ is the weight of those bases. 
The \textbf{motivic zeta function} $Z_M(q,t)$ and \textbf{reduced motivic zeta function} $\overline{Z}_M(q,t)$ are defined similarly, and can be obtained from $Z_M^0(q,t)$ through the following relations.
\[ 
Z_M(q,t) = q^{-1}(q-1) \left (\frac{1}{1 - q^{- r(M)}t^{|I|}} \right ) \overline{Z}_M(q,t), 
\quad 
Z_M^0(q,t) = q^{-1}(q-1) \left (\frac{q^{-r(M)}t^{|I|}}{1 - q^{- r(M)}t^{|I|}} \right ) \overline{Z}_M(q,t)
\]

%

%

\begin{theorem} The motivic zeta functions
$Z_M(q,t), Z_M^0(q,t), \overline{Z}_M(q,t)$ are valuations.
\end{theorem}

\begin{proof} 
To show that a function is valuative, it suffices to show that it is valuative on matroids of a fixed ground set and fixed rank. Thus, from the above relations, to prove that the three motivic zeta functions are valuative, it suffices to prove that $M \mapsto q^{r(M)} Z_M^0(q,t)$ is a strong valuation. 

 We proceed by induction on $|I|$. 
 Again,  every function is trivially a valuation on the two matroids with $|I|=1$. 
 Now suppose this assignment is a strong valuation for matroids on ground sets of size less than $k$, and consider a ground set $I$ with $|I| = k$. We use the following recurrence, proved in \cite[Theorem 1.12]{JKU19}:
\[
q^{r(M)}Z_M^0(q,t) = \frac{(q-1)q^{-r(M)}t^{|I|}}{1 - q^{-r(M)}t^{|I|}} 
 	\left ( \overline{\chi}_M(q) + \sum_{\hat{0} \subsetneq F \subsetneq I  \textrm{ flat}} \overline{\chi}_{M/_F}(q)q^{r(M|_F)}Z_{M|_F}^0(q,t)\right ).
\]
Assume momentarily that $M$ is loopless, so $\hat{0}=\emptyset$. 
If $S \subsetneq I$ is not a flat of $M$, then the contraction $M/_S$ will have a loop, so $\chi_{M/_S}(t)=0$. Thus the equation above is equivalent to 
 \begin{eqnarray*}
 q^{r(M)}Z_M^0(q,t) &=&  \frac{(q-1)q^{-r(M)}t^{|I|}}{1 - q^{-r(M)}t^{|I|}} \left ( \chi_M(q) + \sum_{\emptyset \subsetneq S \subsetneq I }\overline{\chi}_{M/_S}(q)q^{r(M|_S)}Z_{M|_S}^0(q,t)\right ) \\
  &=& \frac{(q-1)q^{-r(M)}t^{|I|}}{1 - q^{-r(M)}t^{|I|}} \left ( \overline{\chi}_M(q) + \sum_{\emptyset \subsetneq S \subsetneq I }m \circ (q^{r(-)}Z_{-}^0(q,t) \otimes \overline{\chi}_{-}(q)) \otimes \Delta_{S,T}(-)\right ).
 \end{eqnarray*}
 If $M$ is non-trivial and has a loop, then both equations above say $0=0$, and the equivalence is still valid.
 Since the matroids in a matroid subdivision have the same rank and ground set, we may ignore the factor $\frac{(q-1)q^{-r(M)}t^{|I|}}{1 - q^{-r(M)}t^{|I|}}$.
 For each decomposition $S \sqcup T=I$ with $T \not = \emptyset$, the map $M \mapsto q^{r(M)}Z_M^0(q,t)$ is a valuation for matroids on $S$ by the induction hypothesis, and $\overline{\chi}_{-}(q)=(-1)^{r(-)}T_{-}(1-q,0)/(q-1)$ is a valuation for matroids on $T$ by Proposition \ref{prop:Tutte}. By Theorem \ref{mainthm:vals}, $ q^{r(M)}Z_M^0(q,t)$ is a valuation on $I$ as desired. 
\end{proof}

\subsection{The Billera-Jia-Reiner quasisymmetric function of a matroid}

For a matroid $M$, a function $f$ from the ground set $I$ of $M$ to the natural numbers $\N$ is $M$-\textbf{generic} if $M$ has a unique basis $B$ that minimizes $f(B) = \sum_{b \in B} f(b)$. 

\begin{definition}
The \textbf{Billera-Jia-Reiner quasisymmetric function} \cite{BJR09} of a matroid $M$ on ground set $I$ is the quasisymmetric function given by
 \[ F(M,\mathbf{x}) = \sum_{\substack{w: I \to \N \\ \text{$M$-generic}}} \prod_{b \in B} x_{w(b)}.
 \]
\end{definition}

Billera, Jia, and Reiner showed this is the quasisymmetric function $\Phi_\zeta$ associated by Definition \ref{def:charinvariants} to the character 
 \[ \zeta(M) = \left \{ \begin{array}{ll}
    1 & \text{if $M$ has a unique basis} \\
    0 & \text{otherwise}
    \end{array} \right.\]
on the Hopf monoid of matroids $\mbf{M}$.

\begin{theorem}\cite{BJR09}
The Billera-Jia-Reiner quasisymmetric function is a valuation on matroids.
\end{theorem}

\begin{proof}
By Corollary \ref{cor: character theory valuation} it suffices to show that the map $\zeta$ is a strong valuation on matroids.
We use the following useful lemma proved by Ardila, Fink, and Rinc\'on \cite{AFR}. For any closed convex set $X \subset \R^I$, the function 
 \[ j_X(M) = \begin{cases}
    1 & \text{if $P(M) \cap X = \emptyset$,} \\
    0 & \text{otherwise.}
    \end{cases}\]
is a valuation for the matroids on $I$.

Let $B_{r,I}$ be the subset of $\{0,1\}^I$ consisting of those vectors with exactly $r$ ones. For each point $b \in B_{r,I}$, consider the valuation $i_b = j_{\text{conv}(B_{r,I} - b)}$. For matroids of rank $r$ on $I$, the function $i_b(M)$ is equal to $1$ if and only if $b$ is the unique basis of $M$. Therefore $\zeta = \sum_{b \in B_{r,I}} i_b$ is a valuation.
\end{proof}

\subsection{The Derksen-Fink invariant of a matroid}

A \textbf{valuative invariant} on matroids is a valuation $f$ on matroids such that $f(M) = f(N)$ whenver $M$ and $N$ are isomorphic. Derksen and Fink constructed a valuative invariant on matroids which is universal among all valuative invariants in the sense that any other valuative invariant is obtained from theirs by a specialization \cite{DF10}. 

\begin{definition} Let $M$ be a matroid on $I$ with $|I| = n$ and let $\ell$ be a linear order of $I$. The \textbf{rank jump} function with respect to $\ell$ is the function $\mathrm{rkjump}_\ell: \mbf{M}[I] \to \{0,1\}^{n}$ given by
 \[ 
 (\mathrm{rkjump}_\ell(M))_i =
 r_M(\{\ell_1, \ldots, \ell_i\}) - r_M(\{\ell_1, \ldots, \ell_{i-1}\}).
 \]
The \textbf{Derksen-Fink invariant} is the function $\mathcal{G}: \mbf{M}[I] \mapsto \R[U_{\alpha} \; \lvert \; \alpha \in \{0,1\}^n]$ given by
 \[ \mathcal{G}(M) = \sum_{\ell} U_{\mathrm{rkjump}_\ell(M)}, \]
 where the sum is over all linear orders $\ell$ on $I$.
\end{definition}

Let us give a short proof that $\mathcal{G}$ is indeed a valuation using Theorem \ref{mainthm:vals}. To do this, we will identify the vector space $\R[U_{\alpha} \; \lvert \; \alpha \in \{0,1\}^n]$ as a subspace of the algebra $\R\langle x, y \rangle$ of noncommutative polynomials in $x$ and $y$ through the map
 \[U_{\alpha} \mapsto z_1 z_2 \cdots z_n, \]
where
 \[z_i = \begin{cases}
 x & \text{if $\alpha_i = 0$} \\
 y & \text{if $\alpha_i = 1$}
 \end{cases} \]
With this, $\mathcal{G}$ extends to a map which we also denote by $\mathcal{G}$ from $ \mbf{M}[I]$ into $\R\langle x , y \rangle$.

\begin{theorem} The Derksen-Fink invariant is a valuation on matroids.
\end{theorem}

\begin{proof} 
For each singleton $\{a\}$ define the map $f: \mbf{M}[\{a\}] \to \R\langle x, y \rangle$ by
 \[ 
 f(M) = \begin{cases}
 x & \text{if  $r(M) = 0$,} \\
 y & \text{if $r(M) = 1$.}
 \end{cases}
 \]
There are only two matroids on $\{a\}$, namely a loop and a coloop, so their matroid polytopes have no matroid subdivisions. Therefore, $f$ is trivially a valuation. 

Now, notice that for each linear order $\ell$, the map $m_{A} \circ f^{\otimes n} \circ \Delta_{\{\ell_1\}, \ldots, \{\ell_n\}}$, which is a valuation by Theorem \ref{mainthm:vals}, sends $M$ to the noncommutative polynomial identified with $U_{\mathrm{rkjump}_\ell(M)}$. Summing over all possible $\ell$ we obtain the desired result.
\end{proof}

%
%
%

\section{Valuations on poset cones}\label{sec:valP}

We now study valuations on poset and preposet cones. Recall from Theorem \ref{thm:CGP} that the map $p \mapsto \operatorname{cone}(p)$ is a bijection between (pre)posets and (not necessarily) pointed conical generalized permutahedra where the origin is in the lineality space. Furthermore, this map induces Hopf monoid isomorphisms
\[
\mbf{PP} \cong \mbf{CGP}_0^+ \qquad \mbf{P} \cong \mbf{PCGP}_0^+
\]
from (pre)posets to (not necessarily) pointed generalized permutahedra where $0$ is (in) the apex. We call these (pre)poset cones, and identify the isomorphic pairs of Hopf monoids above.

\begin{proposition} \label{prop:P=PP}
\begin{enumerate} 
\item
There is an equality of Hopf monoids $\mathbb I(\mbf{P}) = \mathbb I(\mbf{PP})$.
\item
The indicator functions of the totally ordered preposets $\ell$ on ground set $I$ form a basis for $\val(\mbf{PP}[I])$. For any preposet $q$, the expansion of $\one_{q}$ in this basis is
 \[ \one_{q} = \sum_{\ell \text{ prelin. ext. of $q$}} (-1)^{|q| - |\ell|} \one_{\ell}. \]
\item
A valuation $f$ on poset cones extends uniquely to a \textbf{strong valuative extension} $\hat{f}$ on preposet cones.
The assignment $f \mapsto \hat{f}$ is compatible with Theorem \ref{mainthm:vals} and Corollary \ref{cor: character theory valuation}. 
\end{enumerate}
\end{proposition}

The compatibility of the assignment $f \mapsto \hat{f}$ with Theorem \ref{mainthm:vals} is the following:
Suppose $\hat{f}_1, \ldots, \hat{f}_k$ are valuations on preposets which extend the valuations $f_1, \ldots, f_k$ on posets. Then for any ordered set partition $S_1 \sqcup \cdots \sqcup S_k$, the valuation on preposets given by $m_{A} \circ \hat{f_1} \otimes \cdots \otimes \hat{f_k} \circ \Delta_{S_1,\ldots,S_k}$ extends the valuation on posets given by $m_{A} \circ f_1 \otimes \cdots \otimes f_k \circ \Delta_{S_1,\ldots,S_k}$. 
Similarly, the compatibility with Corollary \ref{cor: character theory valuation} is the following: If $\zeta$ is a character on posets and $\hat{\zeta}$ is its extension to preposets, then the extension of the poset invariants $f_\zeta, \Phi_\zeta,$ and $O_\zeta$ are the preposet invariants
$f_{\hat{\zeta}}, \Phi_{\hat{\zeta}},$ and $O_{\hat{\zeta}}$, respectively. 

\begin{proof} 1. We prove the equality of species by proving that the vector spaces are isomorphic on any fixed finite ground set $I$.

\medskip
\noindent
$\subseteq$: Every poset is a preposet.

\medskip
\noindent
$\supseteq$: We need to prove that every preposet cone equals a linear combination of poset cones in $\mathbb I (\mbf{PP})$. We proceed by reverse induction on $|q|$. If $|q| = |I|$ then $q$ is a poset and the result is trivial. For $|q|<|I|$, consider an equivalence class $A$ of $q$ of size at least $2$ and an element $a \in A$. Define the preposets obtained from $q$ as follows:
\begin{eqnarray*}
q_+: && \text{Make $a$ greater than $A-a$ and keep all other relations of $q$ intact.} \\
q_-: && \text{Make $a$ less than $A-a$ and keep all other relations of $q$ intact.} \\
q_\pm: && \text{Make $a$ incomparable to $A-a$ and keep all other relations of $q$ intact.} 
\end{eqnarray*}
Let $C = \operatorname{cone}(q_{\pm})$. Let $b \in A-a$; notice that neither $e_a-e_b$ nor $e_b-e_a$ is in the cone $C$. We have 
\[
\operatorname{cone}(q_+) = C + \R_{\geq 0}(e_a-e_b), \qquad 
\operatorname{cone}(q_-) = C + \R_{\leq 0}(e_a-e_b), \qquad 
\operatorname{cone}(q) = C + \R(e_a-e_b)
\]
and one readily verifies that:
\[
\operatorname{cone}(q_+) \cap \operatorname{cone}(q_-) = 
\operatorname{cone}(q_\pm), \qquad 
\operatorname{cone}(q_+) \cup \operatorname{cone}(q_-) = 
\operatorname{cone}(q).
\]
The only nontrivial claim here is that $\operatorname{cone}(q_+) \cap \operatorname{cone}(q_-) \subseteq \operatorname{cone}(q_{\pm}).$ This follows by observing that for any point $x \in \operatorname{cone}(q_+)$ such that
$x \notin \operatorname{cone}(q_\pm)$, there is a hyperplane $H$ for which  $x$ is in the positive half-space and $ \operatorname{cone}(q_\pm)$ is in the negative half-space. But then $e_a-e_b$ is in the positive half-space, so $\operatorname{cone}(q_-)$ is in the negative half-space and cannot contain $x$. 

It follows that
\[
\one_q + \one_{q_\pm} = \one_{q_+} + \one_{q_-}.
\]
Since $|q_\pm| = |q_+| = |q_-| = |q|+1$, by induction, the cones $\one_{q_\pm}, \one_{q_+}, \one_{q_-}$ are linear combinations of poset cones. Therefore, so is $\one_q$.
The result follows by induction.

\medskip
\noindent
2. This follows from the proof of Lemma \ref{lem:gens}. The linear relation says that the aligning generators. The cones of the totally ordered preposets of $I$ are the centered plates; their indicator functions are linearly independent by \cite[Theorem 2.7]{ESS}.

\medskip
\noindent
3. This follows readily, since valuations on $\mbf{(P)P}$ correspond to linear functions on $\mathbb I(\mbf{(P)P})$. The compatibility follows readily from the definitions.
%
%
%
\end{proof}

%
%

%

All of the valuations on posets studied in this section will be built out of the following simple valuation. Say a preposet is a \textbf{preantichain} if there are no $x,y$ such that $x<y$; in other words, if $x \leq y$ then $y \leq x$. Note that for posets, this restricts to the usual notion of antichains.

\begin{definition}\label{def:alpha}
The antichain and preantichain characters are defined as follows.
\begin{enumerate} 
\item 
The \textbf{antichain character}\footnote{When we regard posets $\mbf{P} \cong \mbf{PCGP}^+_0 \subset \mbf{GP}$ as a Hopf submonoid of generalized permutahedra, the antichain character of $\mbf{P}$ is the restriction of the basic character of $\mbf{GP}$ defined in \cite{AA17}.} $\alpha: \mbf{P}[I] \to \F$ on posets is given by
 \[ \alpha(p) = \begin{cases}
 1 & \text{if $p$ is an antichain} \\
 0 & \text{otherwise}.
 \end{cases}\]
\item
The \textbf{preantichain character} $\hat{\alpha}: \mbf{PP}[I] \to \F$ on preposets is given by
 \[ \hat{\alpha}(q) = \begin{cases}
 (-1)^{|I| - |q|} & \text{if $q$ is a preantichain} \\
 0 & \text{otherwise}.
 \end{cases}\]
\end{enumerate}
\end{definition}

We now show that these characters are indeed valuations.

\begin{proposition} \label{prop:preanti}
The preantichain character $\hat{\alpha}$ and the antichain character $\alpha$ are strong valuations on preposets $\mbf{PP}$ and posets $\mbf{P}$, respectively. Furthermore, the preantichain character $\hat{\alpha}$ on preposets is the strong valuative extension of the antichain character ${\alpha}$ on posets given by  Proposition \ref{prop:P=PP}.3.
\end{proposition}

\begin{proof} 
By Lemma \ref{lem:gens} and Theorem \ref{thm:quotient}, to prove that  $\hat{\alpha}$ is a strong valuation,  it suffices to check that $\hat{\alpha}$ is  zero on any aligning generator $C-\al^-(C)$ corresponding to $C = \operatorname{cone}(q)$ for a preposet $q$. 

If $q$ is not a preantichain, then none of its linear extensions is a preantichain either, so we have $\hat{\alpha}(C-\al^-(C)) = 0 - 0 = 0$.
If $q$ is a preantichain then its only prelinear extension that is a preantichain is the trivial antichain $t$ consisting of all elements in one equivalence class; therefore we have 
$\hat{\alpha}(C-\al^-(C)) = \hat{\alpha}(q-(-1)^{|q|-|t|} \, t) = (-1)^{|I|-|q|} - (-1)^{|q|-1}(-1)^{|I|-1} = 0$. 

We conclude that $\hat{\alpha}$ is a valuation on preposets. Restricting to posets, we obtain that $\alpha$ is a valuation on posets as well. 

Since $\alpha$ and $\hat{\alpha}$ agree on posets, $\hat{\alpha}$ must be the strong valuative extension of $\alpha$.
\end{proof}

\subsection{The order polynomial}

As a first application, we study the order polynomial of posets.
The \textbf{(strict) order polynomial} of the poset $p$ is defined as the unique polynomial such that for any positive integer $k$ we have
\[ \Omega^{(s)}(p)(k) = \text{number of (strictly) order-preserving maps $p \to [k]$.}\]

\begin{proposition}\cite{AA17} The associated polynomial $\Omega_{\alpha}(p)(t)$ to the antichain character $\alpha$ is the \textbf{strict order polynomial} $\Omega^s(p)(t)$.
\end{proposition}

This Hopf-theoretic interpretation readily gives the following result.

\begin{proposition}
The order polynomial and strict order polynomial are strong valuations on posets.
\end{proposition}

\begin{proof}
The antichain character is a strong valuation by Proposition \ref{prop:preanti}, so Corollary \ref{cor: character theory valuation} implies that the strict order polynomial is a strong valuation.
Stanley's reciprocity theorem (which is explained Hopf-theoretically in \cite{AA17}) says that $\Omega(p)(n) = (-1)^{|I|} \Omega^s(p)(-n)$, so the order polynomial is a strong valuation as well.
\end{proof}

\subsection{The poset Tutte polynomial}

For any antichain $A \subseteq I$ of $p$, let 
 \[
 J_\geq(A) = \{x \in p \; \lvert \; x \geq y \text{ for some $y \in A$}\}, \qquad 
 J_>(A) = \{x \in p \; \lvert \; x > y \text{ for some $y \in A$} \}. 
 \]
For any poset $p$, let $\mathcal{A}(p)$ denote the set of antichains of $p$.

\begin{definition}\cite{Gor93} The \textbf{Tutte polynomial} of a poset $p$ on ground set $I$ is 
 \[T_p(x,y) = \sum_{A \in \mathcal{A}(p)} x^{|J_\geq(A)|}(y+1)^{|J_>(A)|}.\]
\end{definition}

The set of lower ideals of a poset forms an antimatroid, and the Tutte polynomial for posets is a special case of the Tutte polynomial of antimatroids; see \cite{KSL91}. We now show it is a strong valuation.



\begin{proposition} 
The poset Tutte polynomial $T_p(x,y)$ is a strong valuation of posets.
\end{proposition}

\begin{proof} 
Let
\[
f_1(p) = 1, \qquad 
f_2(p) = \alpha(p) \cdot x^{|p|}, \qquad
f_3(p) = x^{|p|}(y+1)^{|p|}.
\]
where $\alpha(p)$ is the antichain character of Definition \ref{def:alpha}. 
%
Their convolution is
\[
 f_1 \star f_2 \star f_3(p) 
 = \sum_{\substack{S_1 \sqcup S_2 \sqcup S_3 = I \\ S_2 \text{ antichain}}} x^{|S_2|}(x(y+1))^{|S_3|}. 
= T_p(x,y).
\]
 where we sum over ordered set partitions $S_1 \sqcup S_2 \sqcup S_3$ where $S_1$ is a lower ideal of $p$, $S_2$ is an antichain and a lower ideal of $p - S_1$, and $S_3$ is an upper ideal of $p$.
For a fixed ground set $I$, the functions $f_1, f_2, f_3$ are constant multiples of $1$ and $\alpha(p)$, which are strong valuations thanks to Proposition \ref{prop: Euler characteristic} and \ref{prop:preanti}. Therefore their convolution is a strong valuation by Corollary \ref{cor: convolution valuations}.
\end{proof}


\begin{corollary} 
The following quantities and their dual quantities are strong valuations on posets:
\begin{itemize}
        \item The number of order ideals of $p$ of size $k$.
        \item The number of antichains of size $k$.
               \item The number of maximal elements of $p$.
        \item The generating function $\displaystyle G_p(s,t) = \sum_{A \text{ antichain}} s^{|J_{\leq}(A)|}t^{|A|}$.
\end{itemize}
\end{corollary}

\begin{proof}
This follows from the fact that these quantities are 
the coefficient of $t^{|I| - k}$ of $T_p(t,0)$, 
the coefficient of $t^k$ in $T_p(t,t^{-1} - 1)$, 
the exponent of $T_p(t,-1) = (t+1)^M$ -- which equals $\frac{d}{dt}(T_p(t,-1))|_{t=0}$ -- and 
$T_{p}(st,t^{-1}-1)$, respectively \cite{Gor93}.
Since the cone of the reverse poset $-p$ is $\operatorname{cone}(-p) = - \operatorname{cone}(p)$, the dual quantities are also strong valuations.
\end{proof}

%
%

\subsection{The Poincar\'e polynomial}

%
%

Let $\Sigma_I$ be the braid arrangement in $\R^I$ and $\mathcal{L}(\Sigma_I)$ be its lattice of intersections, ordered by reverse inclusion; its minimum element is $V=\R^I$.
For a poset $p$, consider the open braid cone 
 \[\sigma^o_p =\{ x \in \R^I \; \lvert \; x_i > x_j \text{ whenever } i \geq_p j\}.\]
 Its closure is dual to the poset cone of $p$.

\begin{definition} \cite{DBKR19, Zaslavsky75} Let $p$ be a poset on $I$. The \textbf{interior intersection lattice} of the braid arrangement $\Sigma_I$ with respect to the poset $p$ is the sublattice of intersections which meet the interior of the cone $\sigma^o(p)$ 
\[
\mathcal{L}_{p}(\Sigma_I) = \{ X \in \mathcal{L}(\Sigma_I) \; \lvert \; X \cap \sigma^o_p \not = \emptyset\}.
\]
ordered by reverse inclusion. The \textbf{Poincar\'e polynomial} of poset $p$ is
 \[
 \Poin(p,t) := \sum_{\substack{X \in \mathcal{L}_{p}(\Sigma_I)}} |\mu(V,X)| \, t^k =:  \sum_{k} c_k(p) \, t^k.
 \]
\end{definition}

Zaslavsky showed that $\Poin(p,1)$ is the number of chambers of the braid arrangement $\Sigma_I$ that lie inside the cone $\sigma^o_p$.
In order to relate the Poincar\'e polynomial to a valuation, we use the following formula of Dorpalen-Barry, Kim, and Reiner.

\begin{definition} Let $p$ be a poset on ground set $I$. A \textbf{transverse ordered set partition} of $p$ is an ordered set partition $S_1 \sqcup \cdots \sqcup S_k = I$ such that $p|_{S_i}$ is an antichain and $S_i$ is a lower ideal of $p|_{S_i \sqcup \cdots \sqcup S_k}$ for each $i$. A \textbf{transverse unordered set partition} of $p$ is an unordered set partition $\{S_1, \ldots S_k\}$ such that there exists an ordering that makes it into a transverse ordered set partition of $p$.
\end{definition}

\begin{proposition}\cite{DBKR19} Let $p$ be a poset on $I$. Let $\Pi_p^{T}$ denote the set of transverse \textbf{unordered} set partitions of $p$. Then,

\[\Poin(p,t) = \sum_{\{S_1, \ldots, S_k\} \in \Pi_p^{T}} \; \prod_{i =1}^k \left ( |S_i| - 1\right)! \; t^{k} \]
\end{proposition}

Before we describe the connection between the Poincar\'e polynomial and valuations, we will describe an ordered analog of the Poincar\'e polynomial.

\begin{proposition}\label{prop: ordered poincare polynomial} Let $p$ be a poset on $I$ and $\Sigma_p^T$ denote the set of transverse \textbf{ordered} set partitions of $p$. The function $\Phi$ given by
 \[ \Phi(p,t) = \sum_{ S_1 \sqcup \cdots \sqcup S_k \in \Sigma_p^T} \; \prod_{i=1}^k \left ( |S_i| - 1 \right)! \; t^k\]
is a strong valuation on posets.
\end{proposition}

\begin{proof} Let $\alpha$ be the antichain character. The function 
 \[ p \mapsto m \circ \alpha^{\otimes k} \circ \Delta_{S_1,\ldots,S_k}(p). \]
equals $1$ if $S_1 \sqcup \cdots \sqcup S_k = I$ is a transverse ordered set partition of $p$, and equals $0$ otherwise. By Theorem \ref{mainthm:vals}, this function is a strong valuation. The function $\Phi$ is a linear combination of these; we have
 \[\Phi(p,t) = \sum_{S_1 \sqcup \cdots \sqcup S_k} \; \prod_{i=1}^k \left (|S_i| - 1 \right)! \; t^k \cdot \left( m \circ \alpha^{\otimes k} \circ \Delta_{S_1,\ldots,S_k}(p) \right ), \]
summing over all ordered set partitions of $I$. Therefore it is a strong valuation.
\end{proof}

The functions $\Poin(p,t)$ and $\Phi(p,t)$ have a similar form, but the former is given by a sum over transverse unordered set partitions while the latter is given by a sum over transverse ordered set partitions. 
To prove results about $\Poin(p,t)$, we will pay more careful attention to the labelling of the poset.

\begin{definition} 
Let $\ell$ be a linear order on the ground set $I$. An ordered set partition $S_1 \sqcup \cdots \sqcup S_k = I$ is $\ell$-\textbf{increasing} if 
$\min_\ell S_1 <_\ell \cdots <_\ell \min_\ell S_k$.
\end{definition}

With these definitions, define the function $\Phi_\ell$ by
\begin{align*}
\Phi_\ell(p,t) &= \sum_{\substack{S_1 \sqcup \cdots \sqcup S_k \in \Sigma_{p}^T \\ \ell-\text{increasing}}} \; \prod_{i=1}^k \left ( |S_i| - 1 \right)! \; t^k \\
    &=  \sum_{\substack{S_1 \sqcup \cdots \sqcup S_k \\ \ell-\text{increasing}}} \; \prod_{i=1}^k \left ( |S_i| - 1 \right)! \; \left ( m \circ \alpha^{\otimes k} \circ \Delta_{S_1,\ldots,S_k}(p) \right )
\end{align*}
By a similar argument as Proposition \ref{prop: ordered poincare polynomial}, we have that $\Phi_\ell$ is a strong valuation for each $\ell$.

\begin{proposition} For any linear extension $\ell$ of the poset $p$ we have
 \[\Poin(p,t) = \Phi_{\ell}(p,t). \]
\end{proposition}

\begin{proof} We prove this by constructing a bijection between the $\ell$-increasing transverse ordered set partitions and the unordered transverse set partitions.
Let $\{T_1, \ldots, T_k\}$ be a transverse unordered set partition. We will find an ordered set partition $S_1 \sqcup \cdots \sqcup S_k$ such that for all $i$ we have that $S_i = T_j$ for some $j$. Let $\ell_1$ denote the minimal element of $I$ with respect to $\ell$. In order for $S_1 \sqcup \cdots \sqcup S_k$ to be $\ell$-increasing, it must be the case that $S_1$ is the part $T_j$ which contains $\ell_1$. Because $p$ is $\ell$-increasing, $S_1$ will be a lower ideal of $p$. Recursively, to determine $S_i$, let $\ell_{i}$ be the minimal element of $I - S_1 - S_2 - \cdots - S_{i-1}$ and note that $S_i$ must be the part $T_j$ that contains $\ell_i$. By construction, $S_1 \sqcup \cdots \sqcup S_k$ will be an ordered transverse set partition. Further, it is clear from this construction that this is the only ordering of $\{T_1,\ldots, T_k\}$ with the required properties.

This gives a bijection between $\ell$-increasing ordered transverse set partitions and unordered transverse set partitions. Therefore the formulas for $\Poin(p,t)$ and $\Phi_{\ell}(p,t)$ coincide.
\end{proof}

Since every poset $p$ is properly labelled with respect to some linear order, this gives a way of studying $\Poin(p,t)$ for any poset $p$ using valuations. We illustrate this general principle with the following concrete result.

\begin{corollary}\label{cor:Poincareval}
The Poincar\'e polynomial is a weak valuation for posets:
If the poset cone of a poset $p$ can be subdivided into the poset cones of the posets $p_1, \ldots, p_k$, then 
\[
\Poin(p,t) = \sum_{i=1}^k (-1)^{c(p_i)-c(p)}\Poin(p_i,t).
\]
where $c(q)$ is the number of connected components of the Hasse diagram of $q$.
\end{corollary}

\begin{proof}
Let $\ell$ be a linear extension of $p$. Then $\operatorname{cone}(\ell)$ contains $\operatorname{cone}(p)$ and hence it contains $\operatorname{cone}(p_i)$ for each preposet $p_i$, so $\ell$ is also a linear extension for them. Thus for each one of these posets the Poincar\'e polynomial coincides with $\Phi_\ell$, which is a strong valuation. Since the poset cone of $q$ has dimension $|I|-c(q)$, the desired equation follows.
\end{proof}

\begin{figure}[h]
\centering
\includegraphics[width=4cm]{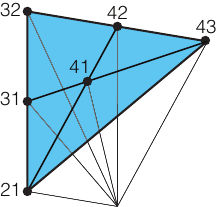}\qquad \qquad
\includegraphics[width=10cm]{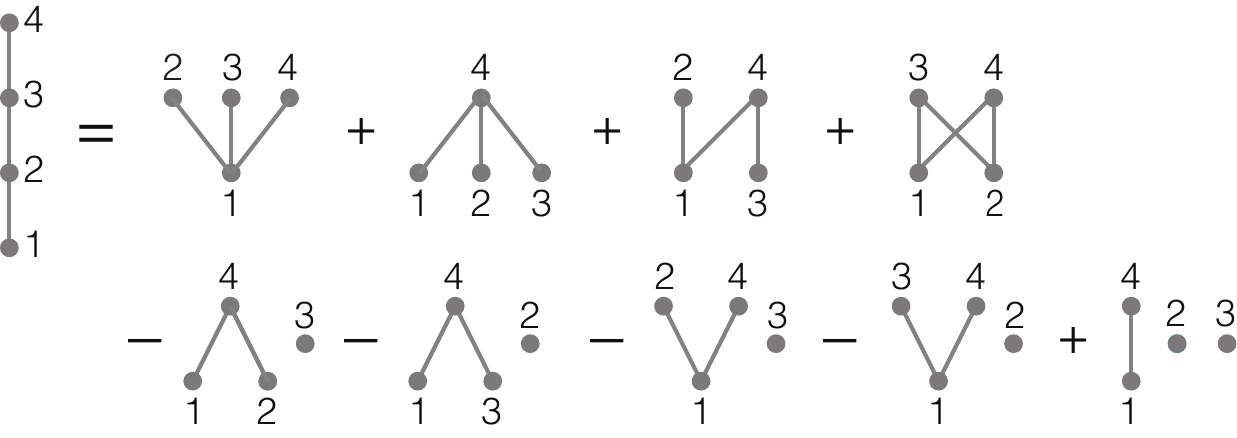}
\caption{A poset subdivision of a poset cone and the corresponding relation on posets in $\mbf{ie} \cap \mbf{P}$.
\label{fig:posetsubdiv}}
\end{figure}

\begin{example}\label{ex:posetvaluation}
Figure \ref{fig:posetsubdiv} shows a subdivision of the poset cone of the chain $1<2<3<4$ into four full-dimensional poset cones and five lower dimensional ones. Since Poincar\'e polynomials of posets are weak valuations by Corollary \ref{cor:Poincareval}, we obtain the following relation between the corresponding Poincar\'e polynomials:
\begin{eqnarray*}
1 &=& (1+3t+2t^2)+ (1+3t+2t^2) + (1+3t+t^2)+(1+2t+t^2) \\
&& - (1+4t+3t^2) - (1+4t+3t^2) - (1+4t+3t^2) - (1+4t+3t^2) + (1+5t+6t^2) .
\end{eqnarray*} 
\end{example}

%
%
%
%
%
%


\section{Building sets and nestohedra}\label{sec:valBS}

Building sets are a combinatorial abstraction of the notion of connectedness. They were introduced independently by Schmitt, seeking methods of understanding the chromatic polynomial \cite{Sch95}, and by De Concini and Procesi, in order to study the wonderful compactification of a hyperplane arrangement \cite{CP96}. Postnikov defined a polytope that encodes the combinatorial structure of a building set, called a nestohedron. \cite{postnikov09}
In this section we show that the $f$-polynomial of a nestohedron is strongly valuative, and we use this to show that there are no subdivisions of a nestohedron into smaller nestohedra. 

As explained in \cite{AA17}, nestohedra do not form a Hopf submonoid of $\mbf{GP}$. Thus it will be more convenient for us to work with a larger class of objects, namely, \textbf{multinestohedra} and the corresponding \textbf{building multisets}.

\begin{definition}
A \textbf{building multiset} $\mathcal{B}$ on ground set $I$ is a multiset of subsets of $I$ satisfying the following axioms:
\begin{itemize}
    \item If $A, B \in \mathcal{B}$ and $A \cap B \not = \emptyset$, then $A \cup B \in \mathcal{B}$.
    \item For all $i \in I$, $\{i\} \in \mathcal{B}$.
\end{itemize}
The \textbf{multinestohedron} of a building multiset $\mathcal{B}$ is the generalized permutahedron
 \[
 N_{\mathcal{B}} = \sum_{J \in \mathcal{B}} \Delta_J, 
 \]
where $\Delta_J$ is the simplex given by $\Delta_J = \operatorname{conv}(e_j \; \lvert \; j \in J)$, and the Minkowski sum contains repeated summands corresponding to the repeated subsets in $\mathcal{B}$.
\end{definition}

A \textbf{building set} is a building multiset with no repeated subsets, and its corresponding polytope is called a \textbf{nestohedron}.
The \textbf{simplification} of a building multiset $\mathcal{B}$ is the building set $\overline{\mathcal{B}}$ obtained by forgetting the multiplicities of the subsets in $\mathcal{B}$. The multinestohedron $N_{\mathcal{B}}$ has the same normal fan as the nestohedron $N_{\overline{\mathcal{B}}}$.

Several important polytopes are nestohedra; for example:


\noindent $\bullet$
The permutahedron, for the building set containing all subsets of $I$.

\noindent $\bullet$ 
The associahedron, for the building set consisting of all intervals $[i,j]$ of $\{1, \ldots, n\}$ for $i < j$.

\noindent $\bullet$
The \textbf{graph associahedron} of Carr and Devadoss \cite{CD06}, for the \textbf{graphical building set} of a graph $G$, which consists of the subsets $I$ of the vertex set for which the graph $G|_I$ is connected. 

\bigskip


The species of building multisets has the structure of a Hopf monoid, defined as follows.
Consider any decomposition $S \sqcup T = I$.
For a building multiset $\mathcal{B}_1$ on ground set $S$ and a building multiset $\mathcal{B}_2$ on ground set $T$, let 
 \[m_{S,T}(\mathcal{B}_1,\mathcal{B}_2) = \mathcal{B}_1 \sqcup \mathcal{B}_2\] where $\sqcup$ denotes the disjoint union of multisets.
For a building multiset $\mathcal{B}$ on ground set $I$, let
 \[ \mathcal{B}|_S = \{A \subseteq S \; \lvert \; A \in \mathcal{B}\},\]
where the multiplicity of $A$ in $\mathcal{B}|_S$ is the multiplicity of $A$ in $\mathcal{B}$. Let
\[\mathcal{B}/_S = \{C \cap T \; \lvert \; C \in \mathcal{B} \}, \]
where the multiplicity of $B \in \mathcal{B}/_S$ is the total number of $C \in \mathcal{B}$ such that $C \cap T = B$, counted with multiplicities. One readily verifies that $\mathcal{B}_1 \sqcup \mathcal{B}_2, \mathcal{B}|_S$, and $\mathcal{B}/_S$ are building sets.

The \textbf{coopposite} $\mbf{H}^{cop}$ of a Hopf monoid $\mbf{H}$ has the same product and the reverse coproduct of $\mbf{H}$; in Sweedler notation, if $\Delta_{S,T}(z) = \sum z|_S  \otimes z/_S$ in $\mbf{H}$, then $\Delta_{S,T}(z) = \sum z/_T \otimes z|_T$ in $\mbf{H}^{cop}$.

\begin{proposition} The linear species $\mbf{BMS}^{cop}[I] = \F\{\text{building multisets on $I$}\}$ forms a Hopf monoid with multiplication maps
 \[m_{S,T}(\mathcal{B}_1, \mathcal{B}_2) = \mathcal{B}_1 \sqcup \mathcal{B}_2, \]
and comultiplication maps
 \[ \Delta_{S,T}(\mathcal{B}) = \mathcal{B}/_S \otimes \mathcal{B}|_S.\]
The map $\mathcal{B} \mapsto N_{\mathcal{B}}$ induces an embedding of $\mbf{BMS}^{cop}$ into $\mbf{GP}$ as Hopf monoids.
\end{proposition}

\begin{proof} The species $\mbf{BMS}^{cop}$ is a Hopf submonoid of the coopposite Hopf monoid of hypergraphs $\mbf{HG}^{cop}$ given in \cite{AA17}, and the map above is the restriction of the analogous map from the coopposite $\mbf{HG}^{cop}$ to the Hopf monoid $\mbf{HGP}$ of hypergraphic polytopes described there. 
\end{proof}

We denote our Hopf monoid $\mbf{BMS}^{cop}$ because it naturally extends the coopposite $\mbf{BS}^{cop}$ of the Hopf monoid of building sets $\mbf{BS}$ defined in \cite{AA17}.

%
%

\subsection{The $f$-vector}

For any polytope $P$, let 
\[
f_P(t) = \sum_{\text{$F$ face of $P$}} t^{\dim F}
\]
denote the \textbf{$f$-polynomial} of the polytope. For any building multiset $\mathcal{B}$, let $f_{\mathcal{B}}(t)$ denote the $f$-polynomial of the hypernestohedron $N_{\mathcal{B}}$. Its coefficients constitute the \textbf{$f$-vector} of $P$.

The $f$-polynomial of a hypernestohedron can be computed recursively as follows. If $\mathcal{B}$ can not be written as a disjoint union of other building multisets, we say that $\mathcal{B}$ is \textbf{connected}. Otherwise, $\mathcal{B}$ will have a unique factorization $\mathcal{B} = \mathcal{B}_1 \sqcup \cdots \sqcup \mathcal{B}_k$. The building multisets $\mathcal{B}_i$ are called the \textbf{connected components} of $\mathcal{B}$.

\begin{proposition} \cite{postnikov09} \label{prop: f-polynomial recurrence} The $f$-polynomial of a hypernestohedron $N_{\mathcal{B}}$ is the unique polynomial satisfying the following properties:
\begin{enumerate}
    \item If $\mathcal{B}$ is equal to a singleton, then $f_{\mathcal{B}}(t) = 1$.
    \item If $\mathcal{B}$ is disconnected with connected components $\mathcal{B}_1, \ldots, \mathcal{B}_k$, then
     \[f_{\mathcal{B}}(t) = f_{\mathcal{B}_1}(t) \cdots f_{\mathcal{B}_k}(t).\]
    \item If $\mathcal{B}$ is connected, then
     \[f_{\mathcal{B}}(t) = \sum_{S \subsetneq I} t^{|I| - |S| - 1}f_{\mathcal{B}|_S}(t). \]
\end{enumerate}
\end{proposition}

This recursion was first proved for nestohedra but extends to hypernestohedra since the $f$-polynomial of a hypernestohedron $N_{\mathcal{B}}$ is equal to the $f$-polynomial of the nestohedron $N_{\overline{\mathcal{B}}}$.

\begin{theorem} The $f$-polynomial $f_{N_{\mathcal{B}}}(t)$ is a strong valuation on hypernestohedra.
\end{theorem}


\begin{proof} We proceed by induction on the size of the ground set $I$. When $|I| = 1$ every hypernestohedron is a point, so $f_-$ is trivially a strong valuation. Now suppose that $f_-$ is a strong valuation for all $J$ with $|J| < k$, and consider a finite set $I$ with $|I| = k$. 

First, assume $\mathcal{B}$ is disconnected with connected components $\mathcal{B}_1, \ldots, \mathcal{B}_k$. Let $S_1 \sqcup \cdots \sqcup S_k$ be the ordered set partition where $S_i$ is the ground set of $\mathcal{B}_i$. Then
$\Delta_{S_1, \ldots, S_k}(\mathcal{B}) = \mathcal{B}_1 \otimes \ldots \otimes \mathcal{B}_k$ and the recurrence of Proposition \ref{prop: f-polynomial recurrence}.2  takes the form
 \[ f_{\mathcal{B}}(t) = m \circ ( f_{-}(t) \otimes \ldots \otimes f_{-}(t) )\circ \Delta_{S_1, \ldots, S_k} (\mathcal{B}).\]
By the inductive hypothesis, the function $f_{-}(t)$ is a strong valuation on ground sets of size less than $k$. Therefore, $f_{-}(t)$ is a strong valuation on connected building multisets on $I$ by Theorem \ref{mainthm:vals}.

We now turn to the case when $\mathcal{B}$ is connected. Let $g$ be the function on building multisets given by $g(\mathcal{B}) = t^{|T|-1}$ for building multisets $\mathcal{B}$ on $T$. Now the recurrence of Proposition \ref{prop: f-polynomial recurrence} can be written as
 \[ f_{\mathcal{B}}(t) = \sum_{\substack{S \sqcup T = I \\ T \not = \emptyset}} m \circ (f_{-}(t) \otimes g) \circ \Delta_{S,T}(\mathcal{B}).\]
For each decomposition $S \sqcup T = I$, the function $f_{-}$ is strongly valuative by the inductive hypothesis, and the function $g$ is constant so it is also strongly valuative.  Theorem \ref{mainthm:vals} implies that $f_-$ is also strongly valuative for building multisets on the ground set $I$. This completes the induction.
\end{proof}

This has the following consequence for nestohedral subdivisons.

\begin{corollary}\label{cor: nestohedra no subdivisions} Hypernestohedra have no subdivisions into other hypernestohedra.
\end{corollary}

\begin{proof} Since $f_{-}(t)$ is a strong valuation, it is also a weak valuation. If a hypernestohedron $N$ of dimension $d$ had a non-trivial subdivision $\mathcal{P}$, then $\mathcal{P}$ would contain more than one hypernestohedron of dimension $d$. Then, the coefficient of $t^d$ of $f_{N}(t)$ would be $1$ and the coefficient of $t^d$ of $f_{\mathcal{P}}(t)$ would be larger than one. This is a contradiction.
\end{proof}

\section{Acknowledgments}

We thank Marcelo Aguiar, Jos\'e Bastidas, Carolina Benedetti, Chris Eur, Alex Fink, Nick Proudfoot, Vic Reiner, Felipe Rinc\'on, Raman Sanyal, and Mariel Supina for very helpful conversations on this topic.
Part of this work was carried out while FA was on sabbatical at the Universidad de Los Andes in Bogot\'a. He is very thankful to the Simons Foundation and San Francisco State University for funding this visit, and to Los Andes and the caf\'es of Bogot\'a for providing a wonderful setting to work on this project.

\begin{small}

\bibliographystyle{plain}
\bibliography{valuations.bib}

\end{small}

\newpage

\section{Appendix: Hopf algebraic background} \label{sec:appendix}

\subsection{Hopf monoids}

In this appendix we give the precise definition of a Hopf monoid. We also prove the First Isomorphism Theorem in this setting.

\bigskip
\noindent
\textsf{\textbf{Species}}.
A (connected) \textbf{linear species} $\mathbf{F}$ is a functor from the category of finite sets with bijections to the category of vector spaces over $\F$ such that $\mbf{F}[\emptyset] \cong \F$. Explicitly, this consists of the following data.

\begin{itemize}
    \item For each finite set $I$, a vector space $\mathbf{F}[I]$ called the structures of type $F$ on label set $I$.
    \item For each bijection $f:I \to J$ an isomorphism 
     \[\mbf{F}(f): \mbf{F}[I] \to \mbf{F}[J], \]
     such that  $\mbf{F}[\text{id}] = \text{id}$ and for any two bijections $f:I \to J$ and $g:J \to K$ we have
      \[ \mbf{F}[g \circ f] = \mbf{F}[g] \circ \mbf{F}[f].\]
     \end{itemize}

\medskip

   A \textbf{morphism of linear species} $\alpha$  from $\mbf{F_1}$ to $\mbf{F_2}$ is a natural transformation of functors. In other words, $\alpha$ is a collection of linear maps $\alpha[I]: \mbf{F_1}[I] \to \mbf{F_2}[I]$ such that the following diagram commutes

    \begin{center}
    \begin{tikzcd}
    \mbf{F_1}[I] \arrow{r}{\alpha[I]} \arrow{d}{\mbf{F_1}[f]}
    &\mbf{F_2}[I] \arrow{d}{\mbf{F_2}[f]}\\
    \mbf{F_1}[J] \arrow{r}{\alpha[J]} & \mbf{F_2}[J]
    \end{tikzcd}
    \end{center}
    for any two sets $I,J$ and any bijection $f: I \to J$.

If we have a collection of linear maps $g[I]: \mbf{F}[I] \to V$ to the same vector space $V$, we will often identify this with the species map $g$ from $\mbf{F}$ to the species $\mbf{V}[I] = V$ with trivial maps $\mbf{V}[f] = id$ for all $f: I \to J$.

\bigskip
\noindent \textsf{\textbf{Monoids}}.
A (connected) \textbf{linear monoid} $(\mbf{M},m)$ is a linear species equipped with a collection of linear maps
 \[ m_{S, T}: \mbf{M}[S] \otimes \mbf{M}[T] \to \mbf{M}[I],\]
for each decomposition $I = S \sqcup T$. These maps must satisfy the following axioms:
\vskip 2ex
\noindent $\bullet$  \textit{(Naturality)} Let $I$ and $J$ be two sets and $f: I \to J$ be a bijection. Let $I = S \sqcup T$ be a decomposition and let $f|_S$ and $f|_T$ be the restrictions of $f$ to $S$ and $T$, respectively. This gives us a decomposition of $J = f(S) \sqcup f(T)$ and a pair of bijections $f|_S: S \to f(S)$ and $f|_T : T \to f(T)$ Then, we have the following commutative diagram
\begin{center}
\begin{tikzcd}[column sep=large]
\mbf{M}[S] \otimes \mbf{M}[T] \arrow{r}{m_{S,T}} \arrow{d}{\mbf{M}[f|_S] \otimes \mbf{M}[f|_T]}
&\mbf{M}[I] \arrow{d}{\mbf{M}[f]}
\\
\mbf{M}[f(S)] \otimes \mbf{M}[f(T)] \arrow{r}{m_{f(S), f(T)}} 
&\mbf{M}[J]
\end{tikzcd}
\end{center}
\noindent $\bullet$ \textit{(Unitality)} We have $\mbf{M}[\emptyset] \cong \F$. Denote the unit of that vector space by $1$. For any $x \in \mbf{M}[I]$ and for the two trivial decompositions $I = I \sqcup \emptyset$ and $I = \emptyset \sqcup I$, we have
\begin{align*}
1 \cdot x = x \cdot 1 &= x
\end{align*}
 \noindent $\bullet$ \textit{(Associativity)} Let $I = R \sqcup S \sqcup T $ be a decomposition of the index set $I$. Then the following diagram commutes

\begin{center}
\begin{tikzcd}[column sep=large]
\mbf{M}[R] \otimes \mbf{M}[S] \otimes \mbf{M}[T] \arrow{r}{\text{id} \;\otimes\; m_{S,T}} \arrow{d}{m_{R,S} \;\otimes \;\text{id}}
& \mbf{M}[R] \otimes \mbf{M}[S \sqcup T] \arrow{d}{m_{R, S \sqcup T}} 
\\
\mbf{M}[R \sqcup S] \otimes \mbf{M}[T] \arrow{r}{m_{R \sqcup S, T}}
&\mbf{M}[I]
\end{tikzcd}
\end{center}
This allows us to define a multiplication map $m_{S_1,S_2, \ldots, S_k}$ for any set decomposition $I = S_1 \sqcup \cdots \sqcup S_k$.

\medskip

A \textbf{morphism of monoids} from $\mbf{M_1}$ to $\mbf{M_2}$ is a species morphism $\alpha:M_1 \rightarrow M_2$ that is compatible with the monoid structure; that is, 
 \[ 
 \alpha[I] \circ m_{S,T} = m_{S,T} \circ \alpha[S]\otimes \alpha[T]; 
 \]
equivalently, for any $x \in \mbf{M_1}[S]$ and $y \in \mbf{M_1}[T]$ we have $\alpha(x \cdot y) = \alpha(x) \cdot \alpha(y)$.

\bigskip
\noindent \textsf{\textbf{Comonoids}}.
A (connected) \textbf{linear comonoid} $(\mbf{C},\Delta)$ is a linear species equipped with a collection of linear maps
 \[ \Delta_{S,T}: \mbf{C}[I] \to \mbf{C}[S] \otimes \mbf{C}[T]. \]
for each set decomposition $I = S \sqcup T$. These functions must satisfy the following axioms:
\vskip 2ex

\noindent
$\bullet$ \textit{(Naturality)} Let $I$ and $J$ be two sets and $\sigma: I \to J$ be a bijection. Let $I = S \sqcup T$ be a decomposition and let $\sigma|_S$ and $\sigma|_T$ be the restrictions of $\sigma$ to $S$ and $T$, respectively. Then, we have the following commutative diagram

    \begin{center}
    \begin{tikzcd}[column sep=large]
    \mbf{C}[I] \arrow{r}{\Delta_{S,T}} \arrow{d}{\mbf{C}[\sigma]}
    &\mbf{C}[S] \otimes \mbf{C}[T] \arrow{d}{\mbf{C}[\sigma|_T] \otimes \mbf{C}[\sigma|_S]}\\
    \mbf{C}[J] \arrow{r}{\Delta_{\sigma(S), \sigma(T)}}
    &\mbf{C}[\sigma(S)] \otimes \mbf{C}[\sigma(T)]
    \end{tikzcd}
    \end{center}

    \noindent
$\bullet$     \textit{(Counitality)} We have $\mbf{C}[\emptyset] \cong \F$. Denote the (co)unit of that vector space by $1$. For any $x \in \mbf{C}[I]$ and  the two trivial decompositions $I = I \sqcup \emptyset$ and $I = \emptyset \sqcup I$ we have
    \begin{align*}
    \Delta_{I,\emptyset}(x) &= x \otimes 1,\\
    \Delta_{\emptyset,I}(x) &= 1 \otimes x.
    \end{align*}

    \noindent
    $\bullet$     \textit{(Coassociativity)} Let $I = R \sqcup S \sqcup T$ be a decomposition of the index set $I$ into three. Then the following diagram commutes

    \begin{center}
    \begin{tikzcd}[column sep=large]
    \mbf{C}[I] \arrow{r}{\Delta_{R, S \sqcup T}} \arrow{d}{\Delta_{R \sqcup S, T}}
    &\mbf{C}[R] \otimes \mbf{C}[ S\sqcup T] \arrow{d}{\text{id} \; \otimes \; \Delta_{S,T}}
    \\
    \mbf{C}[R \sqcup S] \otimes \mbf{C}[T] \arrow{r}{\Delta_{R,S} \; \otimes \; \text{id}}
    &\mbf{C}[R] \otimes \mbf{C}[S] \otimes \mbf{C}[T]
    \end{tikzcd}
    \end{center}

\medskip

A \textbf{morphism of comonoids} is a species morphism $\alpha: \mbf{C_1} \rightarrow \mbf{C_2}$ that is compatible with the comonoid structure; that is, 
 \[ 
 \Delta_{S,T} \circ \alpha[I] = \alpha[S] \otimes \alpha[T] \circ \Delta_{S,T}
 \]
or equivalently,  for any $c \in \mbf{C_1}[I]$ we have $\Delta_{S,T}(\alpha(c)) = \sum \alpha(c|_S) \otimes \alpha(c/_S)$ in Sweedler notation.

\bigskip
\noindent \textsf{\textbf{Hopf monoids}}.
A linear species $\mbf{H}$ is a \textbf{Hopf monoid} if it is a monoid and a comonoid, and those structures are compatible in the following sense.

\vskip 2ex

\noindent $\bullet$ \textit{(Compatibility)} Let $I = S_1 \sqcup S_2$ and $I = T_1 \sqcup T_2$ be two decompositions of $I$. Let $A = S_1 \cap T_1$, $B= S_1 \cap T_2$, $C = S_2 \cap T_1$, and $D = S_2 \cap T_2$ be their pairwise intersections. Then, we have the commutative diagram

\begin{center}
\begin{tikzcd}[column sep=tiny]
\tb{H}[S_1] \otimes \tb{H}[S_2] \arrow{r}{m_{S_1, S_2}} \arrow{d}{\Delta_{A, B}\; \otimes\; \Delta_{C, D}}
& \tb{H}[I] \arrow{r}{\Delta_{T_1,T_2}}
& \tb{H}[T_1] \otimes \tb{H}[T_2] \\
\tb{H}[A] \otimes \tb{H}[B] \otimes \tb{H}[C] \otimes \tb{H}[D] \arrow{rr}{\text{id} \; \otimes \; \beta \; \otimes \ \text{id}}
&
&\tb{H}[A] \otimes \tb{H}[C] \otimes \tb{H}[B] \otimes \tb{H}[D] \arrow{u}{m_{A,C}\; \otimes \;m_{B,D}}
\end{tikzcd}
\end{center}
where $\beta$ is the braiding map $\beta(x \otimes y) = (y \otimes x)$.

\medskip

A  \textbf{Hopf morphism} is a species morphism $\alpha: \mbf{H_1} \rightarrow \mbf{H_2}$ that is a monoid morphism and a comonoid morphism.

\medskip

The \textbf{antipode} of a Hopf monoid $\mbf{H}$
is the map $s[I]: \mbf{H}[I] \to \mbf{H}[I]$ given by
 \[s[I](x) = \sum_{S_1 \sqcup \cdots \sqcup S_k = I } (-1)^k m_{S_1,\ldots,S_k} \circ \Delta_{S_1,\ldots,S_k}(x). 
 \]
In general this formula has a large amount of cancellation. One major question is to give a combinatorial description of the antipode that is cancellation-free and grouping-free.

\subsection{Hopf ideals and quotients and the First Isomorphism Theorem}\label{sec:cofree}


\noindent $\bullet$ 
An \textbf{ideal} of a monoid $\mbf{M}$ is a subspecies $\mbf{g}$ such that for any set partition $S \sqcup T = I$ we have
 \[m_{S,T}(\mbf{g}[S] \otimes \mbf{M}[T]) \subset \mbf{g}[I] \qquad \textrm{ and } \qquad 
m_{S,T}(\mbf{M}[S] \otimes \mbf{g}[T]) \subset \mbf{g}[I]. \]

\noindent $\bullet$ 
A \textbf{coideal} of a comonoid $\mbf{C}$ is a subspecies $\mbf{g}$ such that for any set partition $S \sqcup T = I$ we have
 \[ \Delta_{S, T}(\mbf{g}[I]) \subset \tb{C}[S] \otimes \tb{g}[T] + \tb{g}[S] \otimes \tb{C}[T].\]

\noindent $\bullet$ 
A \textbf{Hopf ideal} of a Hopf monoid $\mbf{H}$ is a subspecies $\mbf{g}$ that is both an ideal and a coideal.

\medskip

Let $\mbf{F}$ be a species and $\mbf{g}$ be a subspecies. Let $\mbf{F}/\mbf{g}$ denote the species given by the vector spaces $\mbf{F}/\mbf{g}[I] = \mbf{F}[I]/\mbf{g}[I]$ with the natural maps between them. For any $x \in \mbf{F}[I]$ let $[x]$ denote the class of $x$ in the vector space quotient $\mbf{F}[I]/\mbf{g}[I]$.

\begin{enumerate}
    \item If $\mbf{F}$ is a monoid and $\mbf{g}$ is an ideal, then $\mbf{F}/\mbf{g}$ inherits the structure of a monoid called the \textbf{monoid quotient} of $\mbf{F}$ by $\mbf{g}$ given by
     \[m_{S,T}([x],[y]) = [m_{S,T}(x,y)] \qquad \text{ for } x \in \mbf{F}[S], y \in \mbf{F}[T]. \]
     \item If $\mbf{F}$ is a comonoid and $\mbf{g}$ is a coideal, then $\mbf{F}/\mbf{g}$ inherits the structure of a comonoid called the \textbf{comonoid quotient} of $\mbf{F}$ by $\mbf{g}$ given by
      \[ \Delta_{S,T}([x]) = [\Delta_{S,T}(x)] \qquad \text{ for } x \in \mbf{F}[I].\]
    \item If $\mbf{F}$ is a Hopf monoid and $\mbf{g}$ is a Hopf ideal, then the \textbf{Hopf quotient} of $\mbf{F}$ by $\mbf{g}$ is the Hopf monoid given by the comonoid and monoid structures above.
\end{enumerate}

Noether's First Isomorphism Theorem holds for Hopf monoids in the following formulation.

\begin{theorem}\label{thm:Noether} (The First Isomorphism Theorem) Let $\mbf{H_1}$ and $\mbf{H_2}$ be two linear Hopf monoids. Let $f: \mbf{H_1} \to \mbf{H_2}$ be a Hopf monoid morphism. Then,

\begin{itemize}
	\item The image of $f$ is a Hopf submonoid of $\mbf{H_2}$.
	\item The kernel of $f$ is a Hopf ideal of $\mbf{H_1}$
	\item The quotient $\mbf{H_1}/\operatorname{Ker}(f)$ is isomorphic to $\operatorname{Im}(f)$ as Hopf monoids.
\end{itemize}
\end{theorem}

\begin{proof} To show that the image of $f$ is a Hopf submonoid $\mbf{H_2}$, we need to show that the image is closed under multiplication and comultiplication. For multiplication, let $S \sqcup T$ be a decomposition of $I$ and let $x \in \mbf{H_1}[S]$ and $y \in \mbf{H_1}[T]$. Then, since $f$ is a Hopf monoid morphism we have the two equations
 \[ m_{S,T}(f[S](x) \otimes f[T](y)) = f[I](m_{S,T}(x \otimes y)) \qquad \Delta_{S,T}(f[I](x)) = (f[S] \otimes f[T]) (\Delta_{S,T}(x))\]
and hence the image is closed under multiplication and comultiplication.

To show that the kernel of $f$ is a Hopf ideal, let $x \in \operatorname{Ker}(f[I])$ and $y \in \mbf{H_1}$. Then, 
 \[f[I](m_{S,T}(x \otimes y)) = m_{S,T}(f[S](x) \otimes f[T](y)) = m_{S,T}(0 \otimes f[T](y)) = 0. \]
This shows that the kernel is an ideal. Similarly, if $x \in \mbf{Ker}(f[I])$, then
 \[ (f[S] \otimes f[T]) (\Delta_{S,T}(x)) = 0 = \Delta_{S,T}(f[I](x)).\]
This means that {}
 \begin{eqnarray*}
 \Delta_{S,T}(x) &\subseteq& \operatorname{Ker}(f[S] \otimes f[T]) \\
&=&  \operatorname{Ker}(f[S]) \otimes \mbf{H_1}[T] + \mbf{H_1}[S] \otimes \operatorname{Ker}(f[T])
\end{eqnarray*}
and hence the kernel is also a comonoid ideal. The equality follows from \cite[Section 1.19]{Greub}.

Finally, by the First Isomorphism Theorem for vector spaces we have a well-defined linear isomorphism from $\mbf{H_1}[I]/\operatorname{Ker}(f[I])$ to $\operatorname{Im}(f[I])$. The previous two statements of this theorem show that this is also a Hopf isomorphism.
\end{proof}

\subsection{Cofree Hopf Monoids and Universality}\label{sec:cofree}

An important aspect of combinatorial Hopf algebras is the theory of characters developed by Aguiar, Bergeron, and Sottile \cite{ABS06}. 
This gives a method of converting multiplicative functions on a Hopf algebra into quasisymmetric function invariants. We now describe Aguiar and Mahajan's generalization of this theory to Hopf monoids \cite[Section 11.4]{AM10}.

\bigskip
\noindent \textsf{\textbf{Cofree Hopf monoids.}}
A \textbf{positive monoid} is a linear species $\mbf{q}$ such that $\text{dim}(\mbf{q}[\emptyset]) = 0$, equipped with a multiplication map $m$ that satisfies all the axioms of a monoid except for unitality.  The \textbf{tensor species} $\mathcal{T}^{\vee}(\mbf{q})$ on a positive monoid $\mbf{q}$  is the linear species 
generated by ordered set partitions $\ell$ decorated with a $\mbf{q}$-structure on each part of $\ell$; that is, 
\[ 
\mathcal{T}^\vee(\mbf{q})[I] = \espan \{(\ell_1|\cdots|\ell_k , x_1  \otimes \cdots \otimes x_k) \; \lvert \; \ell_1 \sqcup \cdots \sqcup \ell_k = I, \, x_i \in \mbf{q}[A_i] \text{ for } 1 \leq i \leq k\}.
\]

The tensor species $\mathcal{T}^\vee(q)$ has a comultiplication map given by
\[\Delta_{S,T}(\ell_1 | \cdots | \ell_k, , x_1\otimes \cdots \otimes x_k) = 
\left \{ 
 \begin{array}{ll}
 (\ell_1 | \cdots | \ell_i, , x_1\otimes \cdots \otimes x_i) 
 \otimes
 (\ell_{i+1} | \cdots | \ell_k, , x_{i+1}\otimes \cdots \otimes x_k) &  \text{if } S=\ell_1 \sqcup \cdots \sqcup \ell_i. \\ 
 0 & \text{otherwise.}
 \end{array} \right.
 \]
It also has a multiplication map defined as follows.  For ordered set partitions $\ell$ of $S$ and $m$ on $T$, 
  \[m_{S,T}( (\ell, x_1 \otimes \cdots \otimes x_j) \otimes (m, y_1 \otimes \cdots y_k)) =  \sum_{\substack{n \textrm{ quasishuffle} \\ \textrm{of $\ell$ and $m$}}} (n, z_1 \otimes \cdots \otimes z_h), \]
%
where
 \[ z_i = \begin{cases}
 x_a & \text{if $n_i = \ell_a$,} \\
 y_b & \text{if $n_i = m_b$,} \\
 m_{\mbf{q}[n_i]}(\ell_a \otimes m_b) & \text{if $n_i = \ell_a \sqcup m_b$.}
 \end{cases}\]
This makes $\mathcal{T}^\vee(\mbf{q})$ into a Hopf monoid. This is the \textbf{cofree Hopf monoid on }$\mbf{q}$. We say that a Hopf monoid is \textbf{cofree} if it is isomorphic to $\mathcal{T}^{\vee}(\mbf{q})$ for some positive monoid.

\bigskip
\noindent \textsf{\textbf{Universality.}}
Let $\beta: \mathcal{T}^\vee(\mbf{q}) \to \mbf{q}$ be the projection map given by
 \[\beta(\ell_1| \cdots | \ell_k, x_1 \otimes \cdots \otimes x_k)) = \begin{cases}
 x_1 & \text{if } k = 1 \\
 0 & \text{otherwise}
 \end{cases}. \]

For any Hopf monoid $\mbf{H}$, we can construct a positive monoid $\mbf{H}_+$ whose underlying species agrees with $\mbf{H}$ whenever $I \not= \emptyset$ and otherwise $\mbf{H}[\emptyset] = \langle 0 \rangle$. This inherits the structure of a positive monoid from the multiplication of $\mbf{H}$.
The cofree Hopf monoid on $\mbf{q}$ satisfies the following universality result.

\begin{theorem}\label{thm: cofree universality}\cite[Theorem 11.23]{AM10} 
Let $\mbf{H}$ be a Hopf monoid and let $\mbf{q}$ be a positive monoid. For every multiplicative map $\zeta: \mbf{H}_+ \to \mbf{q}$, there exists a unique Hopf morphism $\hat{\zeta}: \mbf{H} \to \mathcal{T}^\vee(\mbf{q})$ such that $\beta \circ \hat{\zeta} = \zeta$. 
Furthermore,
 \[ 
 \hat{\zeta}(x) = \sum_{\ell_1 \sqcup \cdots \sqcup \ell_k = I} (\ell_1|\cdots|\ell_k, \zeta(x_1) \otimes \cdots \otimes \zeta(x_k)),\]
summing over the ordered set partitions $\ell = \ell_1 | \cdots | \ell_k$ of $I$, where $\Delta_{\ell_1,\ldots,\ell_k}(x) = x_1 \otimes \cdots \otimes x_k$.
\end{theorem}

\end{document}